%% file: m.tex
\documentclass[reqno,12pt]{amsart}

\include{environ}
\include{symbols}

\begin{document}

\title[Convergence of adaptive methods for nonlinear PDE]
      {Local Convergence of Adaptive Methods \\
       for Nonlinear Partial Differential Equations}

\author[M. Holst]{Michael Holst}
\email{mholst@math.ucsd.edu}

\author[G. Tsogtgerel]{Gantumur Tsogtgerel}
\email{gantumur@math.ucsd.edu}

\author[Y. Zhu]{Yunrong Zhu}
\email{zhu@math.ucsd.edu}

\address{Department of Mathematics\\
         University of California San Diego\\ 
         La Jolla CA 92093}

\thanks{MH was supported in part by NSF Awards~0715146 and 0915220,
by DOE Award DE-FG02-04ER25620, and by DOD/DTRA Award HDTRA-09-1-0036.}
\thanks{GT and YZ were were supported in part by NSF Award~0715146.}

\date{\today}


\input{abs}

\maketitle

\clearpage

{\footnotesize
\tableofcontents
}
\vspace*{-0.5cm}

\input{intro}

\input{weak}

\input{adaptive}

\input{convergence}

\input{examples}
\input{contraction}
\input{contraction-ex}

\input{conc}

\input{ack}

\bibliographystyle{abbrv}
\bibliography{../bib/books,../bib/papers,../bib/mjh,../bib/library,../bib/ref-gn,../bib/coupling,../bib/pnp}



\end{document}

%% file: environ.tex
\usepackage{color} 
\usepackage{ifpdf}
\ifpdf
    \usepackage[pdftex]{graphicx}
    \usepackage[pdftex]{hyperref}
    \hypersetup{
        unicode=false,          
        pdftoolbar=true,        
        pdfmenubar=true,        
        pdffitwindow=false,     
        pdfstartview={FitH},    
        pdftitle={MCP Article},      
        pdfauthor={Michael Holst},   
        pdfsubject={Mathematics},    
        pdfcreator={Michael Holst},  
        pdfproducer={Michael Holst}, 
        pdfkeywords={PDE, analysis, mathematical physics}, 
        pdfnewwindow=true,      
        colorlinks=true,        
        linkcolor=red,          
        citecolor=blue,         
        filecolor=magenta,      
        urlcolor=cyan           
    }

\else
    \usepackage{graphicx}
    \usepackage{pstricks}
    
    \newcommand{\href}[2]{#2}
\fi

\usepackage{times}
\usepackage{amsfonts}

\usepackage{amsmath}
\usepackage{amsthm}
\usepackage{amssymb}
\usepackage{amsbsy}
\usepackage{amscd}

\usepackage{enumerate}
\usepackage{verbatim}
\usepackage{subfigure}

\newtheorem{theorem}{Theorem}[section]
\newtheorem{corollary}[theorem]{Corollary}
\newtheorem{lemma}[theorem]{Lemma}

\newtheorem{assumption}[theorem]{Assumption}
\newtheorem{proposition}[theorem]{Proposition}

\newtheorem{definition}[theorem]{Definition}
\newtheorem{example}[theorem]{Example}
\newtheorem{remark}[theorem]{Remark}

\numberwithin{equation}{section}  

  \newcounter{mnote}
  \let\oldmarginpar\marginpar
    \renewcommand\marginpar[1]{\-\oldmarginpar[\raggedleft\footnotesize #1]%
    {\raggedright\footnotesize #1}}

\definecolor{myblue}{rgb}{0.2,0.2,0.7}
\definecolor{mygreen}{rgb}{0,0.6,0}
\definecolor{mycyan}{rgb}{0,0.6,0.6}
\definecolor{myred}{rgb}{0.9,0.2,0.2}
\definecolor{mymagenta}{rgb}{0.9,0.2,0.9}
\definecolor{mywhite}{rgb}{1.0,1.0,1.0}
\definecolor{myblack}{rgb}{0.0,0.0,0.0}

%% file: symbols.tex
\newcommand*{\ang}[1]{\left\langle #1 \right\rangle}
\newcommand{\beq}{\begin{equation}}
\newcommand{\eeq}{\end{equation}}
\newcommand{\beqa}{\begin{eqnarray}}
\newcommand{\eeqa}{\end{eqnarray}}
\newcommand{\tbar}{|\hspace*{-0.15em}|\hspace*{-0.15em}|}

\newcommand{\leqs}{\leqslant}      
     
\newcommand{\geqs}{\geqslant}      
      
\renewcommand{\div}{{\operatorname{div}}}

\newcommand{\eps}{\varepsilon}

\newcommand{\tiR}{\mbox{{\tiny $R$}}}

\newcommand{\N}{{\mathbb N}}       
\newcommand{\R}{{\mathbb R}}       

\newcommand{\cG}{{\mathcal G}}

\newcommand{\cL}{{\mathcal L}}
\newcommand{\cM}{{\mathcal M}}

\newcommand{\cS}{{\mathcal S}}
\newcommand{\cT}{{\mathcal T}}

\newcommand{\bD}{{\bf D}}

\newcommand{\bu}{{\bf u}}
\newcommand{\bv}{{\bf v}}



%% file: abs.tex
\begin{abstract}
In this article we develop convergence theory for a general class
of adaptive approximation algorithms for abstract nonlinear 
operator equations on Banach spaces, and then use the theory to
obtain convergence results for practical adaptive finite element
methods (AFEM) applied to several classes of nonlinear elliptic equations.
In the first part of the paper,
we develop a weak-* convergence framework for nonlinear 
operators, whose Gateaux derivatives are locally Lipschitz and satisfy a local inf-sup condition. The framework can be viewed as extending the recent convergence results for
linear problems of Morin, Siebert and Veeser 
to a general nonlinear setting.
We formulate an abstract adaptive approximation algorithm 
for nonlinear operator equations in Banach spaces with local structure.
The weak-* convergence framework is then applied to this class of 
abstract locally adaptive algorithms, giving a general convergence result.
The convergence result is then applied to a standard AFEM algorithm in
the case of several semilinear and quasi-linear scalar elliptic equations 
and elliptic systems, including: a semilinear problem with subcritical nonlinearity, the 
steady Navier-Stokes equations, and a quasilinear problem with
nonlinear diffusion.
This yields several new AFEM convergence results for these nonlinear 
problems.
In the second part of the paper we develop a second abstract convergence 
framework based on strong contraction, extending the recent contraction 
results for linear problems of Cascon, 
Kreuzer, Nochetto, and Siebert and of Mekchay and Nochetto 
to abstract nonlinear problems.
We then establish conditions under which it is possible to apply the 
contraction framework to the abstract adaptive algorithm defined earlier,
giving a contraction result for AFEM-type algorithms applied to nonlinear
problems.
The contraction result is then applied to a standard AFEM algorithm in
the case of several semilinear scalar elliptic equations,
including: a semilinear problem with subcritical nonlinearity, the 
Poisson-Boltzmann equation, and the Hamiltonian constraint in
general relativity, yielding AFEM contraction results in each case.
\end{abstract}

%% file: intro.tex
\section{Introduction}
   \label{sec:intro}

Due to the pioneering work of 
Babuska and Rheinboldt~\cite{Babuska.I;Rheinboldt.W1978}, 
adaptive finite element methods (AFEM) based on {\em a posteriori} error
estimators become standard tools in solving PDEs arising in scientific 
and engineering.
A standard adaptive algorithm has the general iterative structure:
\begin{equation}{\label{eq:adaptive}}
 \textsf{Solve} \longrightarrow \textsf{Estimate} \longrightarrow \textsf{Mark} \longrightarrow \textsf{Refine}
\end{equation}
where 
\textsf{Solve} computes the discrete solution 
       $u_k$ in a subspace $X_k \subset X$;
\textsf{Estimate} computes certain error estimators based on $u_k$, which
        are reliable and efficient in the sense that they are good
        approximation of the true error $u-u_k$ in the energy norm; 
\textsf{Mark} applies certain marking strategies based on the
        estimators; and finally, 
\textsf{Refine} divides each marked element and completes the mesh to
        to obtain a new partition, and subsequently
        an enriched subspace $X_{k+1}$.
The fundamental problem with the adaptive procedure~\eqref{eq:adaptive} 
is guaranteeing convergence of the solution sequence.
For {\em a posteriori} error analysis, we refer to the 
books~\cite{Ainsworth.M;Oden.J2000,Verfurth.R1996,Repin.S2008} 
and the references cited therein.

The first convergence result for~\eqref{eq:adaptive} was obtained by 
Babuska and Vogelius~\cite{Babuska.I;Vogelius.M1984} for 
linear elliptic problems in one space dimension.
The multi-dimensional case was open
until D\"orfler~\cite{Dorfler.W1996} proved convergence of~\eqref{eq:adaptive}
for Poisson equation, under the assumption that the initial mesh was fine 
enough to resolve the influence of data oscillation.
This result was improved by Morin, Nochetto, and 
Siebert~\cite{Morin.P;Nochetto.R;Siebert.K2002}, in which the convergence was proved without conditions on the 
initial mesh, but requiring the so-called \emph{interior node property}, together
with an additional marking step driven by data oscillation.
Since these seminal papers, a number of substantial steps have been taken
to generalize these convergence results for linear elliptic problems in 
various directions.
Of particular interest to us here are the following.
In~\cite{Morin.P;Siebert.K;Veeser.A2007,Morin.P;Siebert.K;Veeser.A2008,Siebert.K2009} the asymptotic convergence results were obtained for a general class adaptive methods for a large class of linear problems.
The theory does not require marking due 
to oscillation, or the interior node property, and allows more general marking strategies than what had been used in
D\"{o}rfler's arguments, with different 
{\em a posteriori error} estimators.
In another direction, it was showed by Binev, Dahmen and 
DeVore~\cite{Binev.P;Dahmen.W;DeVore.R2004} for the first time that
AFEM for Poisson equation in the plane has optimal computational complexity 
by using a critical coarsening step.
This result was improved by Stevenson~\cite{Stevenson.R2007} by showing the optimal 
complexity in general spatial dimension without coarsening step.
These error reduction and optimal complexity results were improved recently 
in several aspects in~\cite{Cascon.J;Kreuzer.C;Nochetto.R;Siebert.K2008}.
In the analysis of~\cite{Cascon.J;Kreuzer.C;Nochetto.R;Siebert.K2008}, 
the artificial assumptions of interior node and extra marking due to 
data oscillation were removed, and the convergence result is applicable to general 
linear elliptic equations.
The main ingredients of this new convergence analysis are the global upper
bound on the error give by the {\em a posteriori} estimator, 
orthogonality (or possibly only quasi-orthogonality) of the underlying
bilinear form arising from the linear problem, and a type of error indicator 
reduction produced by each step of AFEM.
We refer to~\cite{Nochetto.R;Siebert.K;Veeser.A2009} for a recent survey of 
convergence analysis of AFEM for linear elliptic PDEs which gives an overview
of all of these results through late 2009.

There are a number of recent and not-so-recent articles concerning 
{\em a posteriori} error analysis for nonlinear partial differential 
equations;
cf.~\cite{Bank.R;Smith.R1993,Verfurth.R1994,Pousin.J;Rappaz.J1994,Bansch.E;Siebert.K1995,Repin.S2000a,Krizek.M;Nemec.J;Vejchodsky.T2001,Rappaz.J2006,Kim.K2007,Repin.S2008,Charina.M;Conti.C;Fornasier.M2008}.
However, to date there have been only a handful of AFEM convergence results
for nonlinear problems.
Some of the results are: AFEM convergence for a scalar problem involving
the $p$-Laplacian was shown in~\cite{Veeser.A2002,Diening.L;Kreuzer.C2008};
AFEM convergence for a class of convex nonlinear problems arising
in elasticity in~\cite{Carstensen.C;Orlando.A;Valdman.J2006,Carstensen.C2009};
and AFEM convergence for the nonlinear Poisson-Boltzmann equation
in~\cite{Chen.L;Holst.M;Xu.J2007}.
These results
typically
involve problem-specific handling of the nonlinearity.
A recent article in a more general direction is the paper of 
Ortner and Praetorius~\cite{Ortner.C;Praetorius.D2008} where the convergence 
analysis of an adaptive algorithm for a large class of nonlinear equations is 
discussed based on energy minimization, including the cases lacking an 
Euler-Lagrange equation due to low differentiability properties of the energy.
However, their argument is tailored specifically for non-conforming finite 
element methods, with some remaining obstacles for the conforming case.

In this article we develop convergence theory for a general class
of adaptive approximation algorithms for abstract nonlinear 
operator equations on Banach spaces, and then use the theory to
obtain convergence results for practical adaptive finite element
methods (AFEM) applied to several classes of nonlinear elliptic equations.
In the first part of the paper,
we develop a weak-* convergence framework for nonlinear 
operators, whose Gateaux derivatives are locally Lipschitz and satisfy a local inf-sup condition.
The framework can be viewed as extending the recent convergence results for
linear problems of Morin, Siebert and Veeser~\cite{Morin.P;Siebert.K;Veeser.A2007,Morin.P;Siebert.K;Veeser.A2008,Siebert.K2009}
to a general nonlinear setting.
We formulate an abstract adaptive approximation algorithm 
for nonlinear operator equations in Banach spaces with local structure.
The weak-* convergence framework is then applied to this class of 
abstract locally adaptive algorithms, giving a general convergence result.
The convergence result is then applied to a standard AFEM algorithm in
the case of several semilinear and quasi-linear scalar elliptic equations 
and elliptic systems, including a semilinear problem with polynomial nonlinearity, the 
steady Navier-Stokes equations, and a more general quasilinear problem.
This yields several new AFEM convergence results for these nonlinear 
problems.

A disadvantage of the weak-* convergence framework is that it does not 
give information on adaptive finite element convergence {\em rate};
strict error contraction results are key to complexity analysis of 
specific instances of the AFEM algorithms.
To allow for complexity results of this type, in the second part of the 
paper we develop a second abstract convergence 
framework based on strong contraction, extending the recent contraction 
results for linear problems of 
Cascon, Kreuzer, Nochetto, and Siebert~\cite{Cascon.J;Kreuzer.C;Nochetto.R;Siebert.K2008}
and of Mekchay and Nochetto~\cite{Mekchay.K;Nochetto.R2005}
to abstract nonlinear problems.
We then establish conditions under which it is possible to apply the 
contraction framework to the abstract adaptive algorithm defined earlier,
giving a contraction result for AFEM-type algorithms applied to nonlinear
problems.
The contraction result is then applied to a standard AFEM algorithm in
the case of several semilinear scalar elliptic equations,
including a semilinear problem with polynomial nonlinearity, the 
Poisson-Boltzmann equation~\cite{Holst.M;McCammon.J;Yu.Z;Zhou.Y2009}
and the Hamiltonian constraint~\cite{Holst.M;Tsogtgerel.G2009}
in general relativity, yielding AFEM contraction results in each case.
%

The remainder of this paper is organized as follows.
In Section~\ref{sec:weak_framework}, we develop an abstract framework for 
ensuring that a sequence of Petrov-Galerkin (PG) approximations to the 
nonlinear problem converges to the solution 
of a nonlinear equation, by ensuring the weak-* convergence to zero of the 
sequence of corresponding nonlinear residuals.
This involves first establishing {\em a priori} estimates and
a general convergence result in Section~\ref{subsec:weak_apriori}, together
with recalling some (mostly standard) {\em a posteriori} error estimates
in Section~\ref{subsec:weak_aposteriori}.
In Section~\ref{sec:adaptive}, we present a class of abstract
adaptive algorithms which (under reasonable assumptions) fit into both the 
weak-* convergence framework
developed in Section~\ref{sec:weak_framework} and the contraction framework
developed in Section~\ref{sec:contraction-abstract}.
The class of algorithms is general enough to include both classical
adaptive finite element methods (AFEM) for two- and three-dimensional
elliptic systems, as well as AFEM algorithms for geometric elliptic PDE on
Riemannian manifolds (cf.~\cite{Holst.M;Tsogtgerel.G2009,Holst.M2001}).
In Section~\ref{sec:convergence}, we give the main convergence results for the
class of adaptive algorithms described in Section~\ref{sec:adaptive}.
In particular, we prove that
the adaptive algorithm generates a sequence of approximate solutions which
converge strongly to the solution, by showing that the corresponding sequence
of nonlinear residuals weak-* converges to zero.
We present a sequence of examples in Section~\ref{sec:examples} to illustrate 
the weak-* convergence framework.
In Section~\ref{sec:contraction-abstract}, we outline a second distinct 
abstract framework for
ensuring that a sequence of approximations to the nonlinear problem
produced by an adaptive algorithm converges to the solution of a nonlinear 
equation, by ensuring strict contraction of the {\em quasi-error}
(the sum of the error norm and the error indicator).
This framework is based on establishing strengthened Cauchy and 
quasi-orthogonality-type 
inequalities for successive PG approximations produced by adaptive 
algorithms in Sections~\ref{subsec:global-quasi}--\ref{subsec:local-quasi},
together with a general abstract contraction result derived in 
Section~\ref{subsec:contraction-contract}.
The contraction result is an abstraction of the contraction arguments used
in~\cite{Mekchay.K;Nochetto.R2005,Cascon.J;Kreuzer.C;Nochetto.R;Siebert.K2007,Holst.M;Tsogtgerel.G2009,Holst.M;McCammon.J;Yu.Z;Zhou.Y2009}, suitable for use
with approximation techniques for nonlinear problems.
As in these existing frameworks, it is based
on establishing:
(1) quasi-orthogonality;
(2) error indicator bound on the error;
(3) a type of indicator reduction.
We prove that under these assumptions, the adaptive algorithm generates 
a sequence of approximate solutions for which the quasi-error strictly 
contracts.
Finally, we present several examples of increasing difficulty
in Section~\ref{sec:contraction-ex} to illustrate this framework.

%% file: weak.tex
\section{An Abstract Weak* Convergence Framework}
\label{sec:weak_framework}

In this section, we focus on developing a general convergence framework
for abstract nonlinear equations.
To explain the problem class, the adaptive approximation algorithm,
and the set of convergence results we wish to establish,
let $X$ and $Y$ be real Banach spaces (complete normed
vector spaces over the field $\mathbb{R}$) with norms $\|\cdot\|_X$
and $\|\cdot\|_Y$, respectively.
Denote the topological dual spaces of bounded linear functionals on 
$X$ and $Y$ as $X^*$ and $Y^*$ respectively.
In this paper, we are interested in the convergence of a general class of
adaptive algorithms for solving the nonlinear equation: 
$$\mbox{Find}\;\; u\in X, \mbox{ such that } F(u) = 0,$$
or in a {\em weak form}:
\begin{equation}
\label{eqn:model}
\mbox{Find}\;\; u\in X,  \mbox{ such that } \langle F(u), v\rangle = 0,
\forall v\in Y,
\end{equation}
based on placing some minimal conditions on the first 
(Gateaux or Frechet) derivative of $F$.
We note that~\eqref{eqn:model} often itself arises naturally through
Gateaux differentiation of a scalar-valued energy $J:X\rightarrow \mathbb{R}$,
as the \emph{Euler Condition} for stationarity of $J(u)$,
although we will consider the general case here
whereby there may not be an underlying energy functional.
In any case, recall (cf.~\cite{Vain73,Kesa89,OrRh70})
that the Gateaux variation of $F$ at $u \in X$
in the direction $w \in X$ is given as:
\begin{equation}
\label{eqn:gateaux-variation}
F^{\prime}(u)w
= \left. \frac{d}{d\epsilon} F(u+\epsilon w) \right|_{\epsilon=0},
\end{equation}
and recall that when they exist as bounded linear operators, the 
Gateaux and Frechet derivatives at $u$ in the direction $w$
agree with $F'(u)$ above, uniquely generated by~\eqref{eqn:gateaux-variation}.
Note that in general, the solution to equation~\eqref{eqn:model} may not be unique.
In this paper,we are interested in the \emph{locally unique solution}, which is unique in a neighborhood:
\begin{definition}
\label{def:local_unique}
We say $u\in X$ is a \emph{locally unique} solution to \eqref{eqn:model} in a neighborhood $U\subset X$ of $u,$ if $u$ is the only solution of \eqref{eqn:model} in $U.$
\end{definition}
Our aim now is to show that: 
\emph{For any convergent sequence $\{u_{k}\}$ in $X,$ if the residuals $F(u_{k})$ of the nonlinear 
equation~\eqref{eqn:model} weak-* converge to zero, 
then the sequence converges to the solution of~\eqref{eqn:model}.}
Based on this abstract convergence result, the remainder of this section will 
be devoted to establishing existence, {\em a priori} error estimates, and 
{\em a posteriori} error estimates, for Petrov-Galerkin approximations
to equation~\eqref{eqn:model}.

The following simple theorem will form the basis for our convergence analysis.
\begin{theorem}\label{thm:abstract_convergence}
For a continuous (nonlinear) map $F:X\to Y^*,$ suppose that $u\in X$ is a locally unique solution to~\eqref{eqn:model} in a neighborhood $U\subseteq X$ of $u.$
Let $\{u_k\}\subset U$ be a sequence converging to some $u_*\in U$,
such that
\begin{equation}\label{eqn:weak_convergence_residual}
\lim_{k\to \infty}\langle F(u_k), v\rangle =0, \qquad \forall v\in Y.
\end{equation}
Then we have $u_*=u$.
\end{theorem}

\begin{proof}
We have
\begin{equation*}
\begin{split}
\ang{F(u_*),v}
&=
\ang{F(u_*)-F(u_k),v}
+ \ang{F(u_k),v}
\\
&\leqs
\|F(u_*)-F(u_k)\|_{Y^*} \|v\|_{Y}
+ |\ang{F(u_k),v}|.
\end{split}
\end{equation*}
The conclusion follows by the continuity of $F,$ \eqref{eqn:weak_convergence_residual} and uniqueness of $u$ in $U.$
\end{proof}

One of our central goals in the paper is now to develop a practical way to 
generate the sequence $\{u_k\}$ satisfying the conditions in 
Theorem~\ref{thm:abstract_convergence}.
To this end, we introduce two sequences of nested (finite-dimensional) 
subspaces
$$X_0\subset X_1\subset\ldots\subset X \mbox{ and }Y_0\subset Y_1\subset\ldots\subset Y,$$
where $\dim (X_{k}) = \dim (Y_{k})$ for each $k\in \N.$
In addition, we introduce the spaces $(X_{\infty}, Y_{\infty}):$ 
$$
X_{\infty} = \overline{\bigcup_{k} X_{k}}^{\|\cdot\|_{X}},\qquad \mbox{and} \qquad Y_{\infty} = \overline{\bigcup_{k} Y_{k}}^{\|\cdot\|_{Y}}.
$$ 
We focus on a class of approximation methods whereby the sequence of
approximations $\{u_k\in X_k\}\subset X$ to the exact solution $u \in X$
to~\eqref{eqn:model} are generated by solving the Petrov-Galerkin (PG) problems
\begin{equation}\label{eqn:nonlinear_galerkin}
\mbox{Find } u_k \in X_k, \mbox{ such that } \langle F(u_k),  v_{k}\rangle =0,\qquad  \forall v_{k}\in Y_k.
\end{equation}

We next consider conditions on $F$ to establish well-posedness 
of~\eqref{eqn:nonlinear_galerkin}, and derive {\em a priori} 
error estimates for the approximations $u_k \approx u$.

\subsection{A Priori Error Estimates}
\label{subsec:weak_apriori}

Let $G$ be a $C^1$ mapping from $X\to Y^*,$ understood as an approximation of $F.$
Assume $G$ satisfies the following conditions:

\begin{enumerate}
\item[(H1)] There exists a constant $\delta>0$ such that $G'$ satisfies
$$\left\|G'(u) - G'(x)\right\|_{\mathcal{L}(X,Y^*)} \leqs L\|u-x\|_X,\;\; \forall x\in X \mbox{ with } \|u-x\|_X\leqs \delta.$$
\item[(H2)] $G'(u)$ is an isomorphism from $X\to Y^*,$ and there exists a constant $M>0$ such that
$$\left\|{G'}^{-1}(u)\right\|_{\mathcal{L}(Y^*,X)} \leqs M.$$
\item[(H3)] $\|G(u)\|_{Y^*}\leqs C,$ where $C=\min\{\frac{\delta}{2M}, \frac{1}{4M^2 L}\}.$
\end{enumerate}
Assumptions (H2) and (H3) are stability and consistency conditions, respectively.
If $G$ satisfies (H1)-(H3), then we have the following lemma, 
similar to~\cite[Theorem 2]{Pousin.J;Rappaz.J1994}.
\begin{lemma}\label{lem:existence_and_priori_estimate}
Let $G$ satisfy the assumptions (H1)-(H3), then there exist a constant $\delta_0>0$ and a unique $u_G\in X$ such that
$G(u_G) = 0, \;\;\mbox{and } \; \|u-u_G\|_X\leqs \delta_0.$
Moreover, we have the following {\em a priori} error estimate:
$$\|u-u_G\|_X\leqs 2 \left\|G'(u)^{-1}\right\|_{\mathcal{L}(Y^*, X)}\|G(u)\|_{Y^*}.$$ 
\end{lemma}

\begin{proof}
We show existence and uniqueness by fixed-point argument.
Define first
$$T(x) = x - G'(u)^{-1} G(x), \qquad\forall x\in X.$$
This new operator $T$ is well-defined because $G'(u)$ is an isomorphism by
Assumption (H2).
Then for any $x_{1}, x_{2}\in X$ we have
\begin{eqnarray*}
\|T(x_{1}) - T(x_{2})\|_{X}
&=& \left\|(x_{1}-x_{2}) + G(u)'^{-1} (G(x_{2}) - G(x_{1}))\right\|_X\\
&=& \left\|(x_{1}-x_{2}) - G(u)'^{-1}\int_{0}^1 G'(x_{1}+t(x_{2}-x_{1})) (x_{1}-x_{2}) dt\right\|_X\\
&=& \left\|G'(u)^{-1} \int_{0}^1 (G'(u) - G'(x_{1}+t(x_{2}-x_{1}))) (x_{1}-x_{2}) dt\right\|_X\\
\end{eqnarray*}
 Let $\delta_0>0$ such that $\delta_0=\min\{\delta, \frac{1}{2LM}\}.$
We try to show that $T$ is a contraction mapping in the ball $B(u, \delta_{0})\subset X.$
By Assumption (H1),  we have
$$\|G'(u) - G'((x_{1}+t(x_{2}-x_{1}))\|_{\mathcal{L}(X,Y^*)} \leqs L\delta_{0},\;\; \forall x_{1}, x_{2}\in B(u, \delta_{0}).$$
Therefore, by the choice of $\delta_{0}$ and (H2) we have
\begin{eqnarray*}
\|T(x_{1}) - T(x_{2})\|_{X} &\leqs& L\delta_{0}\left\|G'(u)^{-1}\right\|_{\mathcal{L}(Y^*,X)} \|x_{1}-x_{2}\|_{X} \leqs \frac{1}{2} \|x_{1} - x_{2}\|_{X}.
\end{eqnarray*}
In addition, by using the above inequality and Assumption (H3), for any $x\in B(u, \delta_0)$ we have
 \begin{eqnarray*}
 \|T(x) -u\|_X &\leqs& \|T(x) - T(u)\|_X + \|T(u)-u\|_X\\
&\leqs& \frac{1}{2} \|x-u\|_X + \left\|G'(u)^{-1} G(u)\right\|_X\\
&\leqs& \frac{1}{2}\delta_0 + MC\leqs \delta_0.
\end{eqnarray*}
Therefore, $T$ is a contraction mapping from $B(u, \delta_0)$ to $B(u, \delta_0).$
Thus, there exists a unique $u_G\in B(u, \delta_0)$ such that $u_G = T(u_G)$, that is, $G(u_G) =0.$
Moreover, 
$$\|u-u_G\|_X = \|u-T(u_G)\|_{X} \leqs 2 \left\|G'(u)^{-1}\right\|_{\mathcal{L}(Y^*, X)}\|G(u)\|_{Y^*},$$
which completes the proof.
\end{proof}

Lemma~\ref{lem:existence_and_priori_estimate} provides us with an abstract framework for existence, uniqueness, and the {\em a priori} error estimate (giving continuous dependence) for the approximated scheme $G(x) =0.$
Based on this lemma, we now try to construct such a nonlinear operator $G$ for the Petrov-Galerkin formulation~\eqref{eqn:nonlinear_galerkin}.
This turns out to be nontrivial, since Petrov-Galerkin formulations are built only on the subspaces $(X_{k}, Y_{k}),$ whereas the operator $G:X\to Y^{*}$ is defined on the pair $(X, Y).$
Therefore, for each pair $(X_{k}, Y_{k}),$ we need to construct an operator $F_{k}:X\to Y^{*}$ such that the weak solution of $F_{k}(x) =0$ is equivalent to the solution of~\eqref{eqn:nonlinear_galerkin}.

To this end, let us first introduce a bilinear form
$b: X\times Y \to \mathbb{R}$ at $u \in X$:
\begin{equation}\label{eqn:b}
  b(x,y)=\langle F'(u)x, y\rangle,\;\; \forall x\in X, \;\; \forall y\in Y,
\end{equation}
which is the linearization of $F$ at $u.$
Denote by $\|b\|$ the norm of $b$:
$$\|b\|:=\sup\{b(x,y): x\in X,\; y \in Y \mbox{ s.t. } \|x\|_X=\|y\|_Y=1\} =\|F'(u)\|_{\mathcal{L}(X, Y^*)}.$$
We assume ``inf-sup'' conditions hold for $b$, i.e., 
there exists a constant $\beta_0>0$ such that
\begin{subequations}
\begin{equation}\label{eqn:nonlinear_continuous_infsup}
  \inf_{x\in X, \|x\|_X=1} \sup_{y\in Y, \|y\|_Y=1} b(x,y) =  \inf_{y\in Y, \|y\|_Y=1} \sup_{x\in X, \|x\|_X=1} b(x,y) =\beta_0>0.
\end{equation}
This condition is equivalent to assuming that $F'(u)$ is an isomorphism from $X$ to $Y^*$ with
$$\|F'(u)^{-1}\|_{\mathcal{L}(Y^*,X)} =\beta_0^{-1}.$$
In the finite-dimensional spaces $(X_k, Y_k),$ we assume that $b$ satisfies a discrete inf-sup condition of the form
\begin{equation}\label{eqn:nonlinear_discrete_infsup}
  \inf_{x\in X_k, \|x\|_X=1} \sup_{y\in Y_k, \|y\|_Y=1} b(x,y) = \inf_{y\in Y_k, \|y\|_Y=1} \sup_{x\in X_k, \|x\|_X=1} b(x,y)\geqs\beta_1>0.
\end{equation}
Based on these inf-sup conditions, we have that $b(\cdot,\cdot)$ also satisfies the following inf-sup condition for the pair of spaces $(X_{\infty}, Y_{\infty}).$

\begin{lemma}\label{lem:infty_infsup}
Let the bilinear form $b(\cdot, \cdot)$ satisfies the inf-sup condition~\eqref{eqn:nonlinear_discrete_infsup} on $(X_k, Y_k)$ for $k=1,2, \dots.$
Then it satisfies the inf-sup condition on $(X_{\infty}, Y_{\infty}):$
\begin{equation}\label{eqn:nonlinear_infty_infsup}
\inf_{x\in X_{\infty}, \|x\|_X=1}\sup_{y\in Y_{\infty}, \|y\|_{Y}=1} b(x,y) = \inf_{y\in Y_{\infty}, \|y\|_Y=1}\sup_{x\in X_{\infty}, \|x\|_{X}=1} b(x,y)\geqs\beta_1>0.
\end{equation}
\end{lemma}
\end{subequations}
\begin{proof}
See~\cite[Lemma 4.2]{Morin.P;Siebert.K;Veeser.A2007}.
\end{proof}

For each $k=0, 1, \cdots, \infty,$ inf-sup condition~\eqref{eqn:nonlinear_discrete_infsup} or~\eqref{eqn:nonlinear_infty_infsup} implies existence of two projectors
$$\Pi_k^X : X\to X_k \;\; \mbox{ and }\;\; \Pi_k^Y : Y\to Y_k,$$
defined by
\begin{eqnarray}
  b(x-\Pi_k^X x, y_k)=0 & \forall y_k\in Y_k & \forall x\in X, \label{eqn:Pi_kX}\\
  b(x_k, y-\Pi_k^Y y)=0 & \forall x_k\in X_k & \forall y\in Y.\label{eqn:Pi_KY}
\end{eqnarray}
These operators are stable in the following sense:
\begin{equation}\label{eqn:stability_of_Pik}
\|\Pi_k^X\|_{\mathcal{L}(X,X_k)}\leqs \frac{\|b\|}{\beta_1}\quad \mbox{and}\quad
\|\Pi_k^Y\|_{\mathcal{L}(Y,Y_k)}\leqs  \frac{\|b\|}{\beta_1}.
\end{equation}
In fact, take projector $\Pi_k^X$ as an example,  by the discrete inf-sup condition~\eqref{eqn:nonlinear_discrete_infsup}, we have
\begin{eqnarray*}
  \beta_1 \|\Pi_k^X x\|_{X} &\leqs& \sup_{y_k\in Y_k, \|y_k\|_Y=1} b(\Pi_k^X x, y_k)\\
  &=& \sup_{y_k\in Y_k, \|y_k\|_Y=1} b(x, y_k)\\
  &\leqs& \|b\| \|x\|_X.
\end{eqnarray*}
Moreover, the discrete inf-sup condition~\eqref{eqn:nonlinear_discrete_infsup} guarantees that
$$(\Pi_{k}^{X})^{2} = \Pi_{k}^{X} \mbox{  and  }(\Pi_{k}^{Y})^{2} = \Pi_{k}^{Y}.$$


Now we are ready to define the nonlinear operator $F_{k}:X\to Y^{*}$ for $k=0, 1, \ldots, \infty:$ 
\begin{equation}\label{eqn:def_of_Fk}
\langle F_k(x), y\rangle :=\langle F(x), \Pi_k^Y y\rangle + b(x, y-\Pi_k^Y y), \;\; \forall x\in X, \; y\in Y.
\end{equation}
By a direct calculation, we observe that 
\begin{equation}
\label{eqn:Fk-derivative}
\langle F_k'(x)w, y\rangle :=\langle F'(x)w, \Pi_k^Y y\rangle + \langle F'(u)w, y-\Pi_k^Y y\rangle.
\end{equation}
In particular, we have $F_{k}'(u) = F'(u).$
This operator $F_{k}$ gives rise to another nonlinear problem:
\begin{equation}\label{eqn:Fk-eq}
\mbox{Find}\;\; w\in X, \;\;\mbox{such that }\;\; \langle F_{k}(w), y\rangle =0, \quad \forall y\in Y.
\end{equation}
The equation~\eqref{eqn:Fk-eq} is posed on the whole spaces $(X, Y).$
However, it is not difficult to verify that the solution to~\eqref{eqn:nonlinear_galerkin} and the zero of~\eqref{eqn:def_of_Fk} are equivalent: 
\begin{lemma}[{\cite[Lemma 1]{Pousin.J;Rappaz.J1994}}]\label{lem:Fk}
$u_k\in X_k$ is a solution of~\eqref{eqn:nonlinear_galerkin} if and only if $u_k\in X$ is a solution of~\eqref{eqn:Fk-eq}.
\end{lemma}
\begin{proof}
We include the proof here for completeness.
If $u_k\in X_k\subset X$ is a solution to~\eqref{eqn:nonlinear_galerkin}, then 
$$\langle F(u_k), v_k\rangle =0,\quad \forall v_k\in Y_k.$$
Therefore,  $$\langle F(u_k), \Pi_k^Y y\rangle =0,\quad \forall y\in Y.$$
For the second term in~\eqref{eqn:def_of_Fk}, notice that $u_k\in X_k,$ and by the definition of $\Pi_k^Y,$ we have 
$$b(u_k, y-\Pi_k^Y y) =0, \quad \forall y\in Y.$$
Thus, $u_k\in X$ is a solution to~\eqref{eqn:Fk-eq}.

Conversely, let $w\in X$ satisfy $\langle F_k(w), y\rangle =0, \;\; \forall y\in Y, $ that is
$$ \langle F(w), \Pi_k^Y y\rangle + b(w, y-\Pi_k^Y y) =0, \;\; \forall y\in Y.$$
By choosing $y= v-\Pi_k^Y v, $ we obtain
$b(w, v- \Pi_k^Y v) =0, \quad \forall v\in Y.$
By the definition of $\Pi_k^Y$ and $\Pi_k^X,$ we then have 
$$b(w - \Pi_k^X w, v) = b(w, v- \Pi_k^Y v) =0, \quad \forall v\in Y.$$
Since the inf-sup condition holds for $b,$ we have $w = \Pi_k^X w \in X_k.$
On the other hand, by choosing $y=v_k\in Y_k,$ we then have 
$$\langle F(w), v_k\rangle =0, \quad \forall v_k\in Y_k,$$ which implies 
    that $w\in X_k$ is a solution to~\eqref{eqn:nonlinear_galerkin}.
\end{proof} 

Lemma~\ref{lem:Fk} shows that~\eqref{eqn:Fk-eq} is actually a reformulation of~\eqref{eqn:nonlinear_galerkin}, which posed in $(X_{k}, Y_{k})$, into the whole spaces $(X, Y).$
It enables us to obtain the well-posedness and a priori error estimate of~\eqref{eqn:nonlinear_galerkin} by applying Lemma~\ref{lem:existence_and_priori_estimate} to $F_{k}.$
More precisely, we have the following main result.
\begin{theorem}\label{thm:apriori_estimates}
  Suppose equation~\eqref{eqn:model} and the discretization~\eqref{eqn:nonlinear_galerkin} satisfy the inf-sup conditions~\eqref{eqn:nonlinear_continuous_infsup}, \eqref{eqn:nonlinear_discrete_infsup} respectively.
Moreover, suppose that $F'$ is Lipschitz continuous at $u,$ that is, 
  \begin{eqnarray*}
    &\exists \delta \mbox{ and } L \mbox{ such that for all } w\in X,\;\; \|u-w\|_X\leqs \delta\\
    & \|F'(u)-F'(w)\|_{\mathcal{L}(X,Y^*)} \leqs L\|u-w\|_X.
  \end{eqnarray*}
  If in addition the subspace $X_0$ satisfies the approximation condition
  \begin{equation}
  \label{eqn:X0}
\inf_{\chi_{0}\in X_{0}}\|u - \chi_{0}\|_X \leqs \|b\|^{-1}\left(1+ \frac{\|b\|}{\beta_1}\right)^{-1} \min\left\{\frac{\delta \beta_0}{2}, \frac{\beta_0^2}{4L}\right\},
  \end{equation} 
 then there exist a constant $\delta_1>0$ such that equation~\eqref{eqn:nonlinear_galerkin} has a locally unique solution $u_k\in X_k$ in $B(u, \delta_1)$ for any $k\geqs 0$ such that $X_0\subset X_k.$
Moreover, we have the {\em a priori} error estimates:
  \begin{equation}
    \|u-u_k\|_X \leqs \frac{2\|b\|}{\beta_0}\left(1+ \frac{\|b\|}{\beta_1}\right) \min_{\chi_k\in X_k}\|u-\chi_k\|_X.\label{eqn:quasi_optimal_Xk}
     \end{equation}
\end{theorem}

\begin{proof}
By Lemma~\ref{lem:Fk},  a solution to equation~\eqref{eqn:nonlinear_galerkin} is equivalent to a solution to the equation~\eqref{eqn:Fk-eq}.
By choosing $G=F_k$ in Lemma~\ref{lem:existence_and_priori_estimate}, we only need to verify Assumptions (H1)-(H3).

Note that~\eqref{eqn:Fk-derivative} implies $F_{k}'(u) = F'(u).$
Therefore, we have 
$$\|F'(u)^{-1}\|_{\mathcal{L}(Y^*,X)} =\beta_0^{-1}$$
 from the inf-sup condition~\eqref{eqn:nonlinear_continuous_infsup}.
The assumption (H2) follows.
Again, by~\eqref{eqn:Fk-derivative} we deduce that for any $w, x\in X$ and $y\in Y,$
$$\langle (F_k'(u) - F'_k(x))w, y\rangle = \langle (F'(u) - F'(x))w, \Pi_k^Y y\rangle.$$
Therefore, 
\begin{eqnarray*}
\left\|F_k'(u) - F_k'(x) \right\|_{\mathcal{L}(X, Y^*)} &\leqs& \left\|F'(u) - F'(x) \right\|_{\mathcal{L}(X, Y^*)} \left\|\Pi_k^Y\right\|_{\mathcal{L}(Y, Y_k)}\\
&\leqs&  \frac{\|b\|}{ \beta_1}  \left\|F'(u) - F'(x)\right\|_{\mathcal{L}(X, Y^*)}\\
&\leqs&  \frac{\|b\|}{ \beta_1}L  \|u - x\|_{X},
\end{eqnarray*}
where in the second inequality we used stability~\eqref{eqn:stability_of_Pik} of $\Pi_k^Y.$
Hence, $F_k$ satisfies (H1).

For Assumption (H3), we have
\begin{eqnarray*}
\|F_k(u)\|_{Y^*} &=& \sup_{v\in Y, \;\|v\|_Y=1} \langle F_k(u), v\rangle\\
&=& \sup_{v\in Y, \;\|v\|_Y=1} b(u,  v-\Pi_k^Y v)\\
&=& \sup_{v\in Y, \;\|v\|_Y=1} b(u-\Pi_k^X u,  v)\\
&\leqs& \|b\| \|u-\Pi_k^X u\|_X.
\end{eqnarray*}
By triangle inequality and stability~\eqref{eqn:stability_of_Pik} of $\Pi_k^X$, we have
$$\|u-\Pi_k^X u\|_X \leqs \|u-\chi_k\|_X+\|\Pi_k^X (u-\chi_k)\|_X \leqs \left(1+\frac{\|b\|}{\beta_1}\right)\inf_{\chi_k\in X_k}\|u-\chi_k\|_X.$$
Therefore, we obtain
$$\|F_k(u)\|_{Y^*} \leqs \|b\|\left(1+\frac{\|b\|}{\beta_1}\right)\inf_{\chi_k\in X_k}\|u-\chi_k\|_X.
$$
Notice that $X_0\subset X_k,$
and by assumption~\eqref{eqn:X0} we have 
$$\|F_k(u)\|_{Y^*} \leqs \|b\|\left(1+\frac{\|b\|}{\beta_1}\right)\inf_{\chi_0\in X_0}\|u-\chi_0\|_X\leqs \min\left\{\frac{\delta \beta_0}{2}, \frac{\beta_0^2}{4L}\right\}.$$
Hence, Assumption (H3) is satisfied.
Therefore by Lemma~\ref{lem:existence_and_priori_estimate}, the there exists a constant $\delta_1>0$ such that equation~\eqref{eqn:nonlinear_galerkin} has a locally unique solution $u_k\in X_k$ in $B(u, \delta_1)$ for any $k\geqs 0.$
Furthermore, we have the following a priori error estimate:
  \begin{equation*}
    \|u-u_k\|_X \leqs 2 \left\|F_{k}'(u)^{-1}\right\|_{\mathcal{L}(Y^*, X)}\|F_{k}(u)\|_{Y^*}\leqs\frac{2\|b\|}{\beta_0}\left(1+ \frac{\|b\|}{\beta_1}\right) \inf_{\chi_k\in X_k}\|u-\chi_k\|_X.
     \end{equation*}
     This completes the proof.
\end{proof}

\begin{remark}
\label{rk:initial_partition}
Theorem~\ref{thm:apriori_estimates} is similar to~\cite[Theorem 4]{Pousin.J;Rappaz.J1994}.
However, instead of assuming the approximation property
 $$\displaystyle\lim_{h\to 0} \inf_{x_k\in X_k}\|u-x_k\|_X =0$$
as used in their proof, we only assume that the initial subspace  $X_0$ satisfies~\eqref{eqn:X0}.
This is important because in the adaptive setting, we cannot (and of course, do not want to) guarantee that $h\to 0$ uniformly.
The assumption~\eqref{eqn:X0} is essentially the approximation property of the subspace $X_0,$ since that 
$$\inf_{\chi_0\in X_0}  \|u-\chi_0\|_X \leqs \|u-I_0^X u\|_X.$$
In most of the applications we consider, the finite element space $X_0$ has certain approximation property, i.e., $\|u-I_0^X u\|_X=O(h_0^{\alpha})$ for some $\alpha>0,$ where $I_0^X$ is inclusion or quasi-interpolation.
Therefore, the condition~\eqref{eqn:X0} can be satisfied by choosing the meshsize $h_0$ of the initial triangulation to be sufficiently small.
\end{remark}

Based on Theorem~\ref{thm:apriori_estimates}, there exists a locally unique solution $u_{\infty} \in B(u, \delta_1)\subset X_{\infty}$ with the test space $Y_{\infty}.$
In the remainder of this section, we will show that the PG solution sequence $\{u_k\in X_k\}$ converges to the solution $u_{\infty} \in X_{\infty}$ of~\eqref{eqn:nonlinear_galerkin} in $(X_{\infty}, Y_{\infty}).$
Therefore, we indeed constructed a convergent sequence $u_{k}\to u_{\infty}$ as $k\to \infty$ by the Petrov-Galerkin approximation.

With this $u_{\infty},$ let  us introduce another bilinear form $b_{\infty}(\cdot, \cdot): \; X_{\infty}\times Y_{\infty} \to \mathbb{R}$ as
\begin{equation}\label{eqn:b_infty}
b_{\infty}(x, y) := \langle F'(u_{\infty}) x, y\rangle,\;\; \forall x \in X_{\infty} \;; y\in Y_{\infty},
\end{equation}
which is formed by linearizing $F$ at $u_{\infty} \in X_{\infty}$.
Comparing with~\eqref{eqn:b}, we have 
\begin{eqnarray*}
\sup_{y\in Y_{\infty}, \|y\|_Y=1} b_{\infty}(x, y) &=&\sup_{y\in Y_{\infty}, \|y\|_Y=1}( b(x,y) + b_{\infty}(x, y) - b(x,y))\\
 &\ge& \beta_0 - \sup_{y\in Y_{\infty}, \|y\|_Y=1} \langle (F'(u_{\infty})-F'(u)) x, y\rangle\\
 &\ge& \beta_0 - \|F'(u) - F'(u_{\infty})\|_{\mathcal{L}(X, Y^*)}, \;\; \forall x\in X_{\infty}, \; \|x\|_X=1.
\end{eqnarray*}
Therefore, if $F'$ satisfies the Lipschitz continuity condition for some $\delta>0$ as stated in Theorem~\ref{thm:apriori_estimates}, 
then we can choose a constant $\delta_{1}>0$ sufficiently small such that the following inf-sup condition holds in $B(u, \delta_{1}):$
\begin{subequations}
\begin{equation}\label{eqn:b_infty_infsup}
\inf_{y\in Y_{\infty}, \|y\|_Y =1} \sup_{x\in X_{\infty}, \|x\|_x=1} b_{\infty}(x, y)=\inf_{x\in X_{\infty}, \|x\|_X =1} \sup_{y\in Y_{\infty}, \|y\|_Y=1} b_{\infty}(x, y) =\tilde{\beta}_0 >0.
\end{equation}
Similarly, we can show the discrete inf-sup condition holds in $B(u, \delta_{1}):$
\begin{equation}\label{eqn:b_infty_discrete_infsup}
\inf_{y\in Y_{k}, \|y\|_Y =1} \sup_{x\in X_{k}, \|x\|_X=1} b_{\infty}(x, y)=\inf_{x\in X_{k}, \|x\|_X =1} \sup_{y\in Y_{k}, \|y\|_Y=1} b_{\infty}(x, y) =\tilde{\beta}_1 >0.
\end{equation}
\end{subequations}
These inf-sup conditions imply that there exists stable projections $\tilde{\Pi}_{k}^{X}$ and $\tilde{\Pi}_{k}^{Y}$ similar to~\eqref{eqn:Pi_kX}-\eqref{eqn:Pi_KY}.
Same as before, we can define a sequence of nonlinear equations:
\begin{equation}
\label{eqn:galerkin-infty}
\mbox{Find } x\in X_{\infty}, \quad \mbox{such that } \langle \widetilde{F}_{k}(x), y\rangle =0, \forall y\in Y_{\infty},
\end{equation}
where 
$$\tilde{F}_{k}(x), y\rangle = \langle F(x), \tilde{\Pi}_{k}^{Y} y\rangle + b_{\infty}(x, y- \tilde{\Pi}_{k}^{Y} y).$$
Following the same lines of the proof of Lemma~\ref{lem:Fk}, 
one can show the solution to the nonlinear 
equation~\eqref{eqn:galerkin-infty} is the solution to the 
PG problem~\eqref{eqn:nonlinear_galerkin} for each $k = 0, 1, \ldots.$

In the proof of Theorem~\ref{thm:apriori_estimates} , if we replace $(X, Y)$ by $(X_{\infty}, Y_{\infty}),$  $u$ by $u_{\infty} $ and the inf-sup conditions~\eqref{eqn:nonlinear_continuous_infsup}-\eqref{eqn:nonlinear_discrete_infsup} by~\eqref{eqn:b_infty_infsup}-\eqref{eqn:b_infty_discrete_infsup}, then we have the following theorem.
\begin{theorem}
\label{thm:abstract}
Let the assumptions in Theorem~\ref{thm:apriori_estimates} be fulfilled.
 Then there exists a neighborhood $B(u, \delta_1)$ of $u$ such that the equation~\eqref{eqn:nonlinear_galerkin} has a locally unique solution $u_k\in X_k$ for each $k\geqs 0.$
 We also have the following {\em a priori} error estimate:
$$\|u_{\infty} - u_k\|_{X} \leqs C \inf_{\chi_k\in X_k}\|u_{\infty} - \chi_k\|_{X},\;\; \forall k=0, 1, \cdots.$$ 
Consequently, the PG sequence $\{u_k\}$ converges to $u_{\infty},$ that is, $\displaystyle\lim_{k\to\infty}u_k = u_{\infty}$ in $X.$
\end{theorem}
\begin{proof}
By the same argument as in Theorem~\ref{thm:apriori_estimates}, equation~\eqref{eqn:nonlinear_galerkin} has a locally unique solution $u_k\in X_k$ for each $k\geqs 0.$
Furthermore, we have the quasi-optimal estimate:
$$\|u_{\infty} - u_k\|_{X} \leqs C \inf_{\chi_k\in X_k}\|u_{\infty} - \chi_k\|_{X},\;\; \forall k=0, 1, \cdots.$$
By density of $\bigcup_{k=1}^{\infty}X_k$ in $X_{\infty},$ we then have $\displaystyle\lim_{k\to \infty} \|u_k-u_{\infty}\|_X =0.$
\end{proof}

\begin{remark}
Theorem~\ref{thm:abstract} confirms that the approximate sequence $\{u_k\}$ has a limit $u_{\infty}\in X_{\infty}.$
However, this $u_{\infty}$ does not necessarily coincide with the exact solution $u.$
Note that $u_{\infty} = u$ if and only if the residual $F(u_{\infty})=0.$
Obviously, this is the case when $X_{\infty} =X.$
However, in general adaptive settings, one has $X_{\infty}\neq X$.
Nevertheless, by Theorem~\ref{thm:abstract_convergence}, it suffices to verify the weak-* convergence: $F(u_k)\rightharpoonup 0.$
\end{remark}

\subsection{A Posteriori Error Estimates}
\label{subsec:weak_aposteriori}

Given any approximation $u_k$ of $u,$ the nonlinear residual $F(u_k)$ can be used to estimate the error
$\|u-u_k\|_{X}$, through the use of a
{\em linearization theorem}~\cite{LiRh94,Verfurth.R1994}.
An example due to Verf\"urth is the following.
\begin{theorem} \cite{Verfurth.R1994}
   \label{thm:verf}
Let $u\in X$ be a regular solution of~\eqref{eqn:model}
so that the Gateaux derivative $F'(u)$ is a linear
homeomorphism of $X$ onto $Y^*$.
Assume that $F'$ is Lipschitz continuous at $u,$ that is, 
  \begin{eqnarray*}
    &\exists \delta \mbox{ and } L \mbox{ such that for all } w\in X,\;\; \|u-w\|_X\leqs \delta\\
    & \|F'(u)-F'(w)\|_{\mathcal{L}(X,Y^*)} \leqs L\|u-w\|_X.
  \end{eqnarray*}
  Let $R=\min \{\delta, L^{-1} \|F'(u)^{-1}\|_{\mathcal{L}(Y^*,X)}, 2 L^{-1} \|F'(u)\|_{\mathcal{L}(X,Y^*)} \}$.
Then for all $u_k \in X$ such that $\|u-u_k\| < R$,
\begin{equation}
   \label{eqn:verf}
   C_1 \|F(u_k)\|_{Y^*}
   \leqs \| u - u_k \|_{X}
   \leqs C_2 \|F(u_k)\|_{Y^*},
\end{equation}
where $C_1=\frac{1}{2}\|F'(u)\|^{-1}_{\mathcal{L}(X,Y^*)}$
and $C_2=2 \| F'(u)^{-1} \|_{\mathcal{L}(Y^*,X)}$.
\end{theorem}
\begin{proof}
See~\cite{Verfurth.R1994}.\qed
\end{proof}
The linearization is controlled by the
choice of $\delta$ sufficiently small, where $\delta$ is the radius of an open
ball in $X$ about $u$.
The strength of the nonlinearity is represented by the
factors in~\eqref{eqn:verf} involving the linearization $F'(u)$ and
its inverse.
To build an asymptotic estimate of the error, one focuses on
two-sided estimates for the nonlinear residual
$\| F(u_k) \|_{Y^*}$ appearing on each side of~\eqref{eqn:verf}.

%% file: adaptive.tex
\section{A General Adaptive Algorithm} 
\label{sec:adaptive}

The analysis in Section~\ref{sec:weak_framework} reveals that under reasonable 
assumptions on the nonlinear operator $F(\cdot),$ the Petrov-Galerkin 
problem~\eqref{eqn:nonlinear_galerkin} is well-posed.
Moreover, given the nested subspaces $\{X_k\}$ and $\{Y_k\},$ the solution 
sequence $\{u_k\in X_k\}$ converges to the exact solution $u\in X$ if the 
corresponding residual sequence $\{F(u_k)\}\subset Y^{*}$ weak-* converges to zero,
that is
\begin{equation}
\label{eqn:residual_weak_conv}
\lim_{k\to \infty} \langle F(u_k), v\rangle =0,\quad \forall v\in Y.
\end{equation}
In this section, we show how to construct subspaces $(X_{k}, Y_{k})$
in an adaptive setting so as to ensure~\eqref{eqn:residual_weak_conv}.
In particular, based on a few assumptions on the algorithm, we show that the 
solution sequence generated by the algorithm produces a residual sequence
that satisfies~\eqref{eqn:residual_weak_conv}.

\subsection{The Setting: Banach Spaces with Local Structure}
\label{subsec:local-structure}

Since the algorithm to be analyzed is of a finite element type, we need to 
have as the spaces $X$ and $Y$ function spaces defined over a 
domain $\Omega$ in $\R^d$, or over a manifold.
The manifold setting is more general because a domain is trivially a manifold; 
however, in order to avoid the necessary differential geometric language to
also cover the case of geometric PDE on manifolds, we consider here the even 
more general setting of measure spaces, which allows for a simple and 
transparent discussion of the core ideas.
(In~\cite{HoTs07b,HoTs07a}, we consider specifically this geometric
PDE setting.)

Let $(\Omega,\Sigma, \mu)$ be a measure space, where $\Omega$ is a set
(a subset of $\R^d$ or a $d$-manifold), 
$\Sigma$ is a $\sigma$-algebra, and $\mu:\Sigma\to[0,\infty]$ is a measure.
Recall that a {\em $\sigma$-algebra} $\Sigma\subseteq 2^{\Omega}$ over 
$\Omega$ is a partition (a collection of subsets or elements) of $\Omega$ 
which contains $\Omega$, and is closed under the complement in $\Omega$ 
and countable union operations.
Then a \emph{measure} $\mu:\Sigma\to[0,\infty]$ is a function with 
$\mu(\emptyset)=0$ and additive under disjoint countable unions.
\begin{subequations}\label{eqn:group-subadd}
We say that $\cT$  is a {\em partition} (a set of subsets) of $\Omega$ 
with {\em elements} (simply connect subsets) $\{\tau\}_{\tau \in \cT}$ 
if $\bigcup_{\tau \in \cT}\overline{\tau}
    = \overline{\Omega}$ and $\tau_1 \cap \tau_2 = \emptyset$
for any $\tau_1, \tau_2 \in \cT$ such that $\tau_1 \neq \tau_2.$
We introduce the meshsize function $h_{\cT}$ associated to $\cT$ as 
$$ h_{\cT}(x) = \mu(\tau)^{\frac{1}{d}}, \quad \forall x\in \tau\in \cT.$$
Note that $h_{\cT}$ is well defined up to a $d$-dimensional Lebesgue measure zero skeleton.
Thus, we can understand $h_{\cT}\in L^{\infty}(\Omega)$ as a piecewise constant function. Given any subset $\cS \subset \cT,$ we denote $\Omega_{\cS} = \bigcup_{\tau\in \cS} \tau.$ Let $X(\tau)$ and be the finite element subspaces defined on each element $\tau\in \cT.$ We denote $X(\Omega_{\cS}):= \bigcup_{\tau\in \cS} X(\tau).$ For simplicity, let $X(\cT):=X(\Omega_{\cT}).$ We use similar notation for $Y$ and $Z$ below.

We assume now that the Banach spaces $X(\cT)$ and $Y(\cT)$ associated with the partition $\cT$ 
have certain local structures provided by the associated measure space
$(\Omega,\Sigma, \mu).$
In particular, we assume that the
induced norms $\|\cdot \|_{X}$ and $\| \cdot\|_{Y}$ are 
subadditive in the underlying domain:
\begin{equation}
\label{eqn:subadditive}
\begin{array}{lll}
\left\|\left\{\|w\|_{X(\tau)}\right\}_{\tau\in\cT} \right\|_{\ell^p} &\simeq& \|w\|_{X},\;\; \forall w\in X(\cT);\\
& & \\
\left\|\left\{\|v\|_{Y(\tau)}\right\}_{\tau\in\cT} \right\|_{\ell^q} &\simeq& \|v\|_{Y},\;\; \forall v\in Y(\cT).
\end{array}
\end{equation}
In addition, we assume that the norms are absolutely continuous 
with respect to the measure $\mu(\cdot)$ in the sense that, for any $w \in X$ and $v\in Y,$ there holds 
\begin{eqnarray}
\label{eqn:normcont}
\|w\|_{X(\omega)} \to  0 \mbox{  and  } \|v\|_{Y(\omega)} \to 0,  \mbox{  as  }  \mu(\omega) \to 0.
\end{eqnarray}
Furthermore, we assume that the abstract or generalized finite element spaces have the following local approximation property:
Let $\overline{Y}\subset Y$ be a dense subspace of $Y$;
we assume that for any partition $\cT$, there exists an 
interpolation operator $I_{\cT}: \overline{Y}\to Y(\cT)$ such that for all
$v\in \overline{Y}$,
\begin{equation}
\label{eqn:Yapp}
\|v-I_{\cT} v\|_{Y(\tau)} \lesssim \|h_{\cT}^s\|_{\infty, \tau}\|v\|_{\overline{Y}(\tau)},\;\; \forall \tau\in \cT,
\end{equation}
where $s>0$ is a constant.
\end{subequations}

The two most relevant examples of such Banach spaces with this type of
local structure are subspaces $X \subset L^p(\Omega)$, where $\Omega$
is either a bounded open subset of $\mathbb{R}^n$, or where $\Omega$
is a Riemannian $W^{t,q}$-manifold (a differentiable manifold with
metric in $W^{t,q}$), and where $\cT$ is a 
partition of $\Omega$ into elements $\tau$.
Such subspaces then include Sobolev spaces of scalar and vector
functions over domains and partitions in $\mathbb{R}^n$
(cf.~\cite{Lions.J;Magenes.E1973,Adams.R1978}),
as well as Sobolev spaces of $W^{s,p}$-sections 
of vector bundles over $\Omega$ and partition elements $\tau$
(see~\cite{Pala65,Hebey96,HNT07b} for a discussion of these spaces).
See~\cite{Holst.M2001,Holst.M;Nagy.G;Tsogtgerel.G2009,Holst.M;Tsogtgerel.G2009}
and Section~\ref{subsec:hc} for examples in the case of manifold domains.
We note that to show~\eqref{eqn:Yapp} holds in specific cases, it is
not enough to assume that $h_{\tau}(x)$ is sufficiently small, but also
that certain geometric (e.g. geodesic angle) conditions hold for the
elements $\{\tau\}$.
In this article, we assume that the subspace contruction schemes produce
partitions $\{\tau\}$ satisfying the appropriate geometric conditions
so that~\eqref{eqn:Yapp} holds.
Finally, we remark that an intermediate space $Z$ such that 
$X \subset Z$, with continuous (even compact) embedding 
\begin{equation}
\label{eqn:z-space}
X(\cT) \hookrightarrow Z(\cT),
\end{equation}
will sometimes play a critical role.
It is assumed that $Z$ has the same local structure as $X$ and $Y$ over
a measure space $(\Omega,\Sigma, \mu)$, in that
both~\eqref{eqn:subadditive} and~\eqref{eqn:normcont} hold for $Z$.
The role of $Z$ will usually be played by $L^p(\Omega)$ for suitably chosen
exponent $p$.

\subsection{The Algorithm: \textsf{SOLVE-ESTIMATE-MARK-REFINE}} 

We now formulate an adaptive algorithm based on enriching the local
structure using error indicators, partition marking, and partition refinement.
Let 
$X_{k}:=X(\cT_k)$ and $Y_{k}:=Y(\cT_k)$ be the abstract finite element spaces defined on the partition $\cT_k.$
Given an initial partition $\cT_0$ of the domain, the adaptive algorithm for 
solving equation~\eqref{eqn:model} is an iteration involving the following main steps: 
\begin{equation}
\label{eqn:adaptive}
\begin{array}{ll}
&(1)\;\; u_k := \textsf{SOLVE} \left(X_{k}, Y_{k}\right);\\
&(2)\;\; \{\eta(u_{k},\tau)\}_{\tau\in \cT_k} := \textsf{ESTIMATE}\left(u_k ,\cT_k\right); \\
&(3)\;\; \cM_k := \textsf{MARK}\left(\{\eta(u_{k},\tau)\}_{\tau\in \cT_k}, \cT_k\right); \\
&(4)\;\; \cT_{k+1} := \textsf{REFINE}\left(\cT_k, \cM_k, \ell\right), \textsf{ increment } k.
\end{array}
\end{equation}

We will handle each of the four steps as follows:
\begin{itemize}
\item \textsf{SOLVE}: We use standard inexact Newton + multilevel solvers for equation~\eqref{eqn:nonlinear_galerkin}
      to produce $u_k \in X_{k}$ on each partition $\cT_k$ (cf.~\cite{Bank.R;Rose.D1981,Holst.M2001,Deuflhard.P2004}).
      To simplify the analysis here, we assume that the discrete solution
      $u_k$ is the exact solution to~\eqref{eqn:nonlinear_galerkin}.
      
\item \textsf{ESTIMATE}: Given a partition $\cT_k$ and the corresponding output  $u_k \in X_{k}$ of the \textsf{SOLVE} modules,
      this module computes and outputs the {\em a posteriori} error estimator $\{\eta(u_{k}, \tau)\}_{\tau\in \cT_k},$ where for each element $\tau\in \cT_k$ the indicator $\eta(u_{k}, \tau)~\geqs~0.$
      
\item \textsf{MARK}:  Based on the {\em a posteriori} error indicators $\{\eta(u_{k}, \tau)\}_{\tau\in \cT_k},$ this module gives a strategy to choose a subset of elements $\cM_k$ of $\cT_k$ for refinement.

\item \textsf{REFINE}: Given the set of marked elements $\cM_k$ and the partition $\cT_k,$ this procedure produces a new partition $\cT_{k+1}$ by refining (subdividing) all elements in $\cM_k$ $\ell\geqs 1$ times.
Some other elements in $\cT_k\setminus \cM_k$ may also be refined based on 
some requirement of the partition, 
such as geometric relationships between neighboring elements 
(sometimes called {\em geometric conformity})
in order to support construction of the spaces $X(\cT)$.
This procedure is known as \emph{completion}.
\end{itemize}

Now we state some basic assumptions on these modules, which will be used in the convergence analysis in Section~\ref{sec:convergence}.

\subsubsection{\textsf{REFINE}} 

We suppose that refinement relies on unique quasi-regular element subdivisions.
More precisely, there exist constants $c_1, c_2\in  (0, 1)$ independent of the partition $\cT,$ such that  any element $\tau \in \cT$ can be subdivided into $n(\tau) \geqs 2$ subelements 
$\tau_1',\dots, \tau_{n(\tau)}'$ such that
\begin{subequations}
\label{eqn:group-refine} 
\begin{equation}
\label{eqn:subelement}
\overline{\tau} = \overline{\tau}_1'\cup \cdots \cup \overline{\tau}_{n(\tau)}',\quad \mu(\tau) = \sum_{i=1}^{n(\tau)} \mu(\tau_i'),
\end{equation}
and 
\begin{equation}
\label{eqn:subelem-meas}
c_1 \mu(\tau) \leqs \mu(\tau_i') \leqs c_2 \mu(\tau), \quad i = 1, \dots , n(\tau).
\end{equation}

We define now the class $\cG$ {\em admissible partitions} of $\Omega$ 
as the subclass of all partitions of $\Omega$ that satisfy the two properties:
\begin{itemize}
\item The partition is subordinate to (a refinement of) $\cT_0$;
\item The partition is locally quasi-uniform in the sense that
\begin{equation}
\label{eqn:local_quasi_uniform}
\sup_{\cT \in \cG} \max_{\tau\in \cT} \#N_{\cT}(\tau) \lesssim 1,\quad  \sup_{\cT \in \cG} \max_{\tau'\in N_{\cT}(\tau)} \frac{\mu(\tau)}{\mu(\tau')} \lesssim 1,
\end{equation}
where $ N_{\cT}(\tau) := \{\tau' \in \cT | \overline{\tau}' \cap \overline{\tau} \neq \emptyset\}$ denotes the set of neighboring elements of $\tau$ in $\cT.$
\end{itemize}
In addition, we suppose that the output partition 
$$\cT' := \textsf{REFINE}(\cT , \cM, \ell)$$ 
satisfies the requirement 
\begin{equation}
\label{eqn:refine}
\forall \tau \in \cM\subset \cT,  \tau \notin \cT', 
\end{equation}
\end{subequations} 
that is, each marked element of the input partition is subdivided at least 
once in the output partition.
Additional elements in $\cT \setminus \cM$ may be refined in order to 
fulfill some other requirements for partitions coming from class $\cG$;
for example, properties such as geometric conformity may need to
also hold in specific case of constructions of $X(\cT)$ over $\cT$
in order to ensure that~\eqref{eqn:Yapp} holds.

\subsubsection{\textsf{SOLVE}} 

We assume that the abstract finite element spaces $X(\cT )$ and $Y(\cT )$ 
build over $\cT$ have the following two natural properties.
Let $\cT, \cT' \in \cG$.
The spaces $X(\cT)$ and $Y(\cT)$ are called {\em conforming} if
\begin{subequations} 
\label{eqn:group-femspace}
\begin{equation}
\label{eqn:conform}
X( \cT)\subset X \mbox{ and  } Y(\cT )\subset Y, \mbox{ and } \dim X(\cT) = \dim Y(\cT),
\end{equation}
and are called {\em nested} if
\begin{equation}
 \label{eqn:nested}
\mbox{if  }\cT' \mbox{ is a refinement of } \cT  \mbox{ then  }X(\cT )\subset X(\cT' ) \mbox{ and } Y(\cT )\subset Y(\cT' ).
\end{equation}
We note that the underlying paritition $\cT$ does not need to be 
{\em geometrically conforming} in order for the spaces built over
$\cT$ to be conforming in the sense of~\eqref{eqn:conform}.
We also assume that the discrete inf-sup condition~\eqref{eqn:nonlinear_discrete_infsup} holds:
\begin{equation}
\label{eqn:discrete-infsup}
\inf_{x\in X(\cT), \|x\|_X =1} \sup_{y\in Y(\cT), \|y\|_Y =1} b(x, y)= \inf_{y\in Y(\cT), \|y\|_Y =1} \sup_{x\in X(\cT), \|x\|_X =1} b(x, y)\geqs \beta_1,
\end{equation}
with some constant $\beta_1 > 0.$
In most conforming finite element spaces in Sobolev spaces, this is an immediate consequence of the usual interpolation error estimates, cf.~\cite{Ciarlet.P1978}.
In Theorem~\ref{thm:apriori_estimates} for the well-posedness of the 
discrete equation, we require the space $X_{0}$ satisfies~\eqref{eqn:X0}: 
\begin{equation}
\label{eqn:initial_partition}
\inf_{\chi_{0}\in X_{0}} \|u - \chi_{0}\|_{X} \leqs \|b\|^{-1}\left(1+ \frac{\|b\|}{\beta_1}\right)^{-1} \min\left\{\frac{\delta \beta_0}{2}, \frac{\beta_0^2}{4L}\right\}, 
\end{equation}
where $\|b\| = \|F'(u)\|_{\mathcal{L}(X, Y^{*})},$ $\beta_0,\beta_{1}$ are the inf-sup constants in 
\eqref{eqn:nonlinear_continuous_infsup} and~\eqref{eqn:nonlinear_discrete_infsup} respectively, $L$ is the Lipschitz constant for $F'(u)$ and $\delta$ is the Lipschitz radius.
\end{subequations} 
Moreover, we suppose that the output 
$$
u_{\cT} := \textsf{SOLVE}\left(X(\cT ), Y(\cT )\right) 
$$ 
is the \emph{Petrov-Galerkin approximation} of $u$ with respect to $\left(X(\cT), Y(\cT)\right):$ 
$$u_{\cT} \in X(\cT ) :\quad \langle F(u_{\cT}) , v\rangle = 0, \;\;\forall v \in Y(\cT).$$
Thanks to~\eqref{eqn:conform}, \eqref{eqn:discrete-infsup} and the assumption on the initial partition~\eqref{eqn:initial_partition}, by Theorem~\ref{thm:apriori_estimates} the Petrov-Galerkin approximation $u_{\cT}$ exists, is unique, and is a $\|\cdot\|_X$ -quasi-optimal choice from $X(\cT ).$

\subsubsection{\textsf{ESTIMATE}} 

Now we make some assumptions on the output 
$$\{\eta (u_{\cT}, \tau)\}_{\tau \in \cT} := \textsf{ESTIMATE}(u_{\cT},\cT)$$ 
for any admissible partition $\cT \in \cG.$
First, we assume that  the following estimate holds for the Petrov-Galerkin approximation $u_{\cT}:$ 
\begin{subequations}
\label{eqn:group-estimate}
for any subset $\cS \subset \cT$ and $v\in Y,$
\begin{equation}\label{e:est-up-bnd}
\langle F(u_{\cT}), v\rangle \lesssim \eta(u_{\cT},\cS) \|v\|_{Y(\Omega_{\cS})}  +  \eta(u_{\cT},\cT\setminus \cS)\|v\|_{Y(\Omega_{\cT\setminus \cS})},
\end{equation}
where $\eta(u_{\cT}, \cS) = \left\| \left\{\eta(u_{\cT},\tau)\right\}_{\tau\in \cS}\right\|_{\ell^p}$
and $\Omega_{\cS}=\bigcup_{\sigma\in \cS}\sigma$ for $\cS \subset \cT.$
We note that the estimate~\eqref{e:est-up-bnd} implies the global upper-bound
\begin{equation}
\label{eqn:global_upper}
\|u-u_{\cT}\|_{X} \lesssim \eta(u_{\cT},\cT).
\end{equation}
Second, we assume the error indicator $\eta(u_{\cT},\tau)$ satisfies local stability.
More precisely, there exists a function $D\in Z(\Omega)$ such that
\begin{equation}
\label{e:est-stab}
\eta(u_{\cT},\tau) \lesssim  \|u_{\cT}\|_{X(\omega_{\cT}(\tau))} + \|D\|_{Z(\omega_{\cT}(\tau))},
\qquad  \forall \tau \in \cT,
\end{equation}
where
 $\omega_{\cT} (\tau) \subset \Omega$ is the patch (union) of elements in $N_{\cT} (\tau),$ 
\end{subequations}
and where the space $Z$ is the appropriate auxillary space as in~\eqref{eqn:z-space} in Section~\ref{subsec:local-structure}.
\begin{remark}
We remark that the stability assumption~\eqref{e:est-stab} is weaker than the local lower bound bound.
As we can see from the examples in Section~\ref{sec:examples}, one can obtain the stability estimate~\eqref{e:est-stab} from the usual local lower bound estimates.
\end{remark}
%

\subsubsection{\textsf{MARK}}

We suppose that the output 
$$ \cM := \textsf{MARK}\left(\{\eta (u_{\cT},\tau)\}_{\tau \in \cT}, \cT \right) $$
of marked elements has the property
\begin{equation}\label{eqn:mark}
\eta(u_{\cT},\tau) 
\leqs
\xi\big(\max_{\sigma\in \cM}\eta(u_{\cT},\sigma)\big),
\qquad 
\tau\in \cT\setminus \cM,
\end{equation}
where $\xi:\R_+\to\R_+$ is a continuous function satisfying $\xi(0)=0.$
Most marking strategies used in practice satisfy~\eqref{eqn:mark}.
For instance, the maximum strategy or equidistribution strategy, cf.~\cite{Morin.P;Siebert.K;Veeser.A2007a}.
In particular, the following D\"orfler marking strategy also satisfies the assumption~\eqref{eqn:mark}:
Given $\theta \in (0,1]$, a marked subset $\cM$ of elements is constructed to satisfy 
\begin{equation}
   \label{E:dorfler-property}
   \eta(u_{\cT},\cM) \geqs \theta \eta(u_{\cT},\cT).
\end{equation} 
This marking strategy, which was proposed by D\"orfler~\cite{Dorfler.W1996} in his original AFEM convergence paper, is proven to be crucial in the proof of contraction, cf.~\cite{Morin.P;Nochetto.R;Siebert.K2002,Cascon.J;Kreuzer.C;Nochetto.R;Siebert.K2007}.
We refer to Section~\ref{sec:contraction-ex} for more detail.

%% file: convergence.tex
\section{Convergence Analysis}
\label{sec:convergence}

Based on the assumptions on the adaptive algorithm, and on the abstract 
framework discussed in Sections~\ref{sec:weak_framework}, 
we are now ready to state and prove the abstract convergence result based on a
weak-* residual convergence.
\begin{theorem}[Abstract Convergence]
\label{thm:conv}
Let $u$ be a locally unique exact solution of~\eqref{eqn:model}.
Assume that the nonlinear operator $F'(u)$ satisfies the inf-sup condition 
\eqref{eqn:nonlinear_continuous_infsup} and is Lipschitz continuous in a neighborhood of $u.$
Let $\{u_k \}$ be the sequence of approximate solutions generated by 
iteration~\eqref{eqn:adaptive}.

If the finite element spaces $(X_{k}, Y_{k})$ satisfy~\eqref{eqn:group-subadd}, and the modules \textsf{REFINE}, \textsf{SOLVE}, 
\textsf{ESTIMATE}, and \textsf{MARK} satisfy, respectively, 
\eqref{eqn:group-refine}, \eqref{eqn:group-femspace},
\eqref{eqn:group-estimate}, and~\eqref{eqn:mark}, 
then there exists $u_{\infty}\in X$ such that $\displaystyle\lim_{k\to \infty} u_{k} = u_{\infty}.$
Moreover, the sequence $\{u_{k}\}$ satisfies 
\begin{equation}
\label{eqn:res-weak}
\lim_{k\to \infty} \langle F(u_k), v\rangle = 0, \quad \forall v\in Y.
\end{equation}
Consequently, we have $u_{\infty} = u,$ that is $\displaystyle\lim_{k\to \infty}\|u_k - u\|_{X} = 0.$
\end{theorem}


We split the partition $\cT_k$ into two sets
$\cT_k^+$ and $\cT_k^0$,
where 
$$\cT_k^+ = \{\tau\in \cT_k: \tau\in \cT_i,\;\; \forall i\geqs k\}$$ contains all the elements that will not be refined after $k$-th step, 
and $\cT_k^0=\cT_k\setminus \cT_k^+$ is the set of elements that will be refined at least once after $k$-th step.
Here the superscript `$+$' means the measure of the elements in $\cT^+_k$ is positive.
We denote
$$\Omega_k^0=\Omega_{\cT_k^0}:= \bigcup_{\tau \in \cT_k^0} \tau \;\;\mbox{ and } \;\; \Omega_k^+=\Omega_{\cT_k^+} :=  \bigcup_{\tau \in \cT_k^+} \tau.$$
For simplicity, we denote $\Omega^0=\displaystyle \bigcap_{i=0}^{\infty}\Omega_i^0.$

We note that the sequence $\{h_{k}\}\subset L^{\infty}(\Omega)$ of the meshsize function is bounded and monotone decreasing for a.e.\ $x\in \Omega.$
Moreover, we have 
 \begin{lemma}[{\cite[Corollary 4.5]{Morin.P;Siebert.K;Veeser.A2007a}}]
 \label{lm:hk}
 The sequences $\{h_{k}\}$ and $\{\Omega_{k}^{0}\}$ satisfy
$$
\lim_{k\to \infty} \|h_{k}\|_{L^{\infty}(\Omega_{k}^{0})} = 0.
$$
 \end{lemma}

Now we are ready to prove Theorem~\ref{thm:conv}.
\begin{proof}[Proof of Theorem~\ref{thm:conv}]
Theorem~\ref{thm:abstract} shows the existence of the Petrov Galerkin solutions $u_{k}\in X_{k}$ and $u_{\infty}\in X$ such that 
$$\lim_{k\to \infty} u_{k} = u_{\infty}.$$
If we can show~\eqref{eqn:res-weak}, that is, the residuals weak-* converge to 0, then Theorem~\ref{thm:abstract_convergence} implies that $u_{\infty} = u.$
Therefore, we need to prove~\eqref{eqn:res-weak}.
Notice that $\overline{Y}$ is dense in $Y,$ we only need to show that 
\begin{equation}
\label{eqn:res-weaky}
\lim_{k\to \infty} \langle F(u_k), v\rangle = 0, \quad \forall v\in \overline{Y}.
\end{equation}


By definition,  the sets $\cT_{k}^{+}$ are nested, that is for any $j\leqs k,$ 
\begin{equation*}
\cT_j^+ \subset \cT_k^+ \subset \cT_k 
\qquad \textrm{and} \qquad
\Omega_j^0=\Omega_{\cT_k\setminus \cT_j^+}.
\end{equation*}
Applying the upper bound~\eqref{e:est-up-bnd} with $\cT=\cT_k$ and $\cS = \cT_j^+$, for any $v\in \overline{Y}$  we have
\begin{equation}\label{e:res}
\langle F(u_k), v\rangle 
=
\langle F(u_k), v-\bar{v}\rangle
\lesssim
\eta(u_{k},\cT_k\setminus \cT_j^+) \|v- \bar{v}\|_{Y(\Omega_j^0)} 
+ \eta(u_{k},\cT_j^+) \|v- \bar{v}\|_{Y(\Omega_j^+)},
\end{equation}
where $\bar{v}$ is arbitrary in $Y_k.$
Given any $\eps>0,$ we need to show that 
for sufficiently large $k$ and $j$, and for a suitable $\bar{v}\in Y_k,$
each term in the right hand side of the above estimate can be bounded by a multiple of $\eps.$

By the local approximation assumption~\eqref{eqn:Yapp},  there exists a $\bar{v}:= I_{j} v\in Y_j \subset Y_{k}$ such that
$$\|v- \bar{v}\|_{Y(\tau)}\lesssim \|h_{j}^{s}\|_{\infty,\tau} \|v\|_{\overline{Y}(\tau)}.$$ 
So according to Lemma~\ref{lm:hk}  for sufficiently large $j,$ we have $\|v- \bar{v}\|_{Y(\Omega^0)}\leqs \frac{\eps}{2}.$
On the other hand, it is easy to see that $\mu(\Omega_j^0\setminus \Omega^0)\to 0$ as $j\to\infty.$
Therefore, by~\eqref{eqn:normcont} for sufficiently large $j$ one has
$\|v- \bar{v}\|_{Y(\Omega_j^0\setminus \Omega^0)}\leqs \frac{\eps}{2}.$
Hence by~\eqref{eqn:subadditive}, we obtain
\begin{equation*}
\|v- \bar{v}\|_{Y(\Omega_j^0)}
\lesssim
\|v- \bar{v}\|_{Y(\Omega^0)}
+ \|v- \bar{v}\|_{Y(\Omega_j^0\setminus \Omega^0)}
\leqs \eps.
\end{equation*}
Notice that $\eta(u_{k},\cT_k\setminus \cT_j^+)$ is uniformly bounded because
\begin{equation*}
\begin{split}
\eta(u_{k},\cT_k\setminus \cT_j^+)
&\leqs
\left\| \left\{ \eta(u_{k},\tau) \right\}_{\tau\in \cT_k}\right\|_{\ell^p}
\\
&\lesssim
\left\| \left\{ \|u_k\|_{X(\omega_k(\tau))} \right\}_{\tau\in \cT_k}\right\|_{\ell^p}
+ \left\| \left\{ \|D\|_{Z(\omega_k(\tau))} \right\}_{\tau\in \cT_k}\right\|_{\ell^p}\\
&\lesssim
\|u_k\|_{X}
+ \|D\|_{Z}
\\
&\leqs
\|u_k-u_{\infty}\|_{X}
+ \|u_{\infty}\|_{X}
+ \|D\|_{Z},
\end{split}
\end{equation*}
where in the second inequality, we used the inequality~\eqref{e:est-stab}, and in the third inequality, we used~\eqref{eqn:local_quasi_uniform} and~\eqref{eqn:subadditive}.
Now since $\displaystyle \lim_{k\to \infty} \|u_k - u_{\infty}\|_X =0,$ for sufficiently large $j\leqs k$ we have
$$\eta(u_{k},\cT_k\setminus \cT_j^+) \leqs 2\|u_{\infty}\|_X + \|D\|_{Z}.$$
Therefore, the first term in the right hand side of~\eqref{e:res} satisfies:
$$
\eta(u_{k},\cT_k\setminus \cT_j^+) \|v- \bar{v}\|_{Y(\Omega_j^0)} \lesssim (2\|u_{\infty}\|_X + \|D\|_{Z})\epsilon.
$$

We fix this $j$ and consider the second term in the right hand side of~\eqref{e:res}, and let $k\geqs j.$
By marking strategy~\eqref{eqn:mark}, for all $\tau\in \cT_j^+ \subset \cT_{k}^{+},$ we have 
$$
\eta(u_{k},\tau)  \leqs \xi\big(\max_{\sigma\in \cT_k^0}\eta(u_{k},\sigma)\big),
$$ 
and moreover for $\sigma\in \cT_k^0$ we have
\begin{eqnarray*}
\eta(u_{k},\sigma)&
\lesssim&
\|u_{k}\|_{X(\omega_k(\sigma))} + \|D\|_{Z(\omega_k(\sigma))}\\
&\leqs& 
\|u_{k}-u_{\infty}\|_{X(\omega_k(\sigma))}+ \|u_{\infty}\|_{X(\omega_k(\sigma))}+ \|D\|_{Z(\omega_k(\sigma))}.
\end{eqnarray*}
The first term goes to zero because $u_k\to u_{\infty}.$
For the second and third terms, we notice that $\mu(\omega_{k}(\sigma))\to 0$ as $k\to \infty$ by the locally quasi-uniformity~\eqref{eqn:local_quasi_uniform} and Lemma~\ref{lm:hk}.
Hence $\|\cdot\|_{X(\omega_{k}(\sigma))}\to 0$ and $\|\cdot\|_{Z(\omega_{k}(\sigma))}\to 0$ as $\mu(\omega_{k}(\sigma))\to 0$ by~\eqref{eqn:normcont}.
Therefore, we can choose 
$k\geqs j$ sufficiently large such that 
$\eta(u_{k},\cT_j^+) \leqs \eps.$
Finally, we proved that 
$$
\lim_{k\to \infty}\langle F(u_{k}), v\rangle =0,\qquad \forall v\in \overline{Y}.
$$
Therefore, \eqref{eqn:res-weak} holds.
This completes the proof.
\end{proof}


The convergence of $u_k\to u$ as $k\to \infty$ in Theorem~\ref{thm:conv} does not imply the convergence of the estimator.
It is indeed possible for  the error indicators to be not efficient in the sense that they might contain strong overestimation.
In other words, an efficient error indicator should be bounded by the error $\|u-u_k\|_X$ in certain way.
 
\begin{theorem}
\label{thm:convergence_estimator}
Let there exist $D\in Z$
and a continuous function $\phi:\R_+\to\R_+$ with $\phi(0)=0$,
such that for any $\cT\in\cG$ and $\tau\in \cT$
\begin{equation}\label{e:loc-low-bnd}
 \eta(u_{k},\tau) 
 \lesssim 
 \|u-u_{k}\|_{X(\tau)} 
 + \phi(\mu(\tau))
 \left( \|u_{k}\|_{X(\tau)} 
 + \|D\|_{Z(\tau)} \right).
\end{equation}
Then under the hypotheses of Theorem~\ref{thm:conv}, we have 
\begin{equation*}
\lim_{k\to \infty} \eta(u_{k},\cT_k) =0.
\end{equation*}
\end{theorem}

\begin{proof}
For $k\geqs j$ by definition of $\eta$ and~\eqref{e:loc-low-bnd} we have
\begin{equation*}
\begin{split}
\eta (u_{k}, \cT_k ) 
&\lesssim 
\eta ( u_{k},\cT_k \setminus \cT_j^+ ) 
+ \eta (u_{k},\cT_j^+ ) \\
&\lesssim 
\| u - u_k \|_{X( \Omega_{\cM_j^0})} 
+\eta(u_{k},\cT_j^+)\\
&+ \left\| \left\{ \phi(\mu(\tau))
 \left( \|u_{k}\|_{X(\tau)} 
+ \|D\|_{Z(\tau)} \right)\right\}_{\tau\in \cT_k \setminus \cT_j^+}\right\|_{\ell^p}.
\end{split}
\end{equation*}
By Theorem~\ref{thm:conv}, we have
$\| u - u_k \|_{X(\Omega_j^0)} 
\leqs 
\| u - u_k \|_{X} \to 0,$ as $k\to \infty.$
The second term goes to zero since
for any $\tau\in \cT_k\setminus \cT_j^+$, $\mu(\tau) \to 0$ as $j\to \infty.$
We follow the same arguments as in Theorem~\ref{thm:conv} to show that the last term converges to zero.
This completes the proof.
\end{proof}

Note that convergence of both the error (Theorem~\ref{thm:conv}) and the estimator (Theorem~\ref{thm:convergence_estimator}) are important.
The convergence result in Theorem~\ref{thm:conv},
$\displaystyle \lim_{k\to \infty} \|u -u_k \|_X = 0,$ means that the approximate solutions get arbitrarily close to the exact solution.
However, this would be of little practical use without the second convergence result, namely $\displaystyle \lim_{k\to\infty} \eta = 0,$ which is the computable counterpart of the first result and thus allows one to recognize the improvement of the approximate solutions.
In particular, $\displaystyle\lim_{k\to \infty} \eta =0$ ensures that if one includes a stopping test with a given positive tolerance, then the algorithm stops after a finite number of iterations.

%% file: examples.tex
\section{Examples}
\label{sec:examples}

In this section, we present some nonlinear examples, and apply the weak-* convergence framework developed in the previous sections to show convergence of the adaptive algorithm~\eqref{eqn:adaptive} for these problems.
We consider a fairly broad set of nonlinear problems (see~\cite{Caloz.G;Rappaz.J1994,Verfurth.R1994,Rappaz.J2006} for example), and show how the weak-* framework can be applied in each case.
Specifically, we consider a specific semilinear problem with subcritical nonlinearity, the stationary incompressible Navier-Stokes equations, and a quasi-linear stationary heat equation with convection and nonlinear diffusion.
Many other nonlinear equations are also covered by this general framework.

We restrict polygonal (or polyhedral) domains $\Omega\subset \mathbb{R}^d,$ where $d=2,3$ is the space dimension;
however, all of the results extend to more general domains with standard boundary approximation algorithms and analysis techniques.
In the examples presented here, the  function spaces $X$ and $Y$ are the Sobolev spaces $W^{s,p}(\Omega)$ with $s\geqs 0$ and $p>1,$ equipped with the norm $\|\cdot\|_{s,p,\Omega}$ and semi-norm $|\cdot|_{s,p, \Omega}.$
The space $W_0^{s,p}(\Omega)$ is the closure of ${\mathcal D}(\Omega)$ in $W^{s,p}(\Omega),$ and $W^{-s,q}(\Omega)$ is the dual space of $W^{s,p}_0(\Omega)$ with $\frac{1}{p} + \frac{1}{q} =1.$
When $p=2,$ we shall denote $H^s(\Omega)$ and $H_0^s(\Omega)$ instead of $W^{s,2}(\Omega)$ and $W^{s,2}_0(\Omega)$ respectively, with the norm $\|\cdot\|_{s,\Omega}$ and semi-norm $|\cdot|_{s, \Omega}$ instead of  $\|\cdot\|_{s,2, \Omega}$ and $|\cdot|_{s,2, \Omega}.$
The space $Z$ in the convergence analysis is taken to be $L^p(\Omega)$ for suitable choice of $p.$
More detailed presentations of the Sobolev spaces can be found, for example in the monographs~\cite{Lions.J;Magenes.E1973,Adams.R1978} in the case of domains in $\mathbb{R}^d$, or~\cite{Pala65,Hebey96,HNT07b} in the case of manifold domains.
These Sobolev spaces satisfy the subadditive assumption~\eqref{eqn:group-subadd}.

Let the initial partition $\cT_0$ of $\Omega$ be conforming and shape regular.
We restrict ourself to a shape-regular bisection algorithm for the refinement.
There is a vast literature on bisection algorithms; cf.~\cite{Arnold.D;Mukherjee.A;Pouly.L2000,Plaza.A;Carey.G2000,Chen.L2006a,Stevenson.R2008} and the references cited therein.
It is well known that the bisection algorithm as well as the shape-regularity of $\cT_0$ guarantee assumption~\eqref{eqn:group-refine} holds for any partitions generated by the algorithm.
Without loss of generality, we  assume that $\|h_0\|_{\infty,\Omega} \leqs 1$ is fine enough such that~\eqref{eqn:initial_partition} holds.

Starting from the initial triangulation $\cT_0,$ the adaptive algorithm~\eqref{eqn:adaptive} generates a sequences of shape-regular triangulations $\{\cT_k\}_{k}$ of $\Omega,$ as well as a sequence of approximate solutions $\{u_k\}.$
For the marking strategy, the condition~\eqref{eqn:mark} is satisfied for example if we use \emph{D\"orfler's strategy} (cf.~\cite{Dorfler.W1996}) \eqref{E:dorfler-property}, or the \emph{Maximum strategy} (cf.~\cite{Babuska.I;Rheinboldt.W1978a}).
Apart from the assumptions on mesh refinement and the marking strategy discussed above, for each individual example below, we need to construct the specific finite element spaces $X_k\subset X$ and $Y_k\subset Y$ which satisfy the conditions~\eqref{eqn:group-femspace}.
We also need to define the specific error indicator $\eta,$ which satisfies~\eqref{eqn:group-estimate}.
More precisely, according to Theorem~\ref{thm:conv} and~\ref{thm:convergence_estimator} we only need to:    
\begin{enumerate}
\item Verify the continuous inf-sup condition~\eqref{eqn:nonlinear_continuous_infsup} and the uniform discrete inf-sup condition~\eqref{eqn:nonlinear_discrete_infsup} of the bilinear form $b(\cdot,\cdot)$ defined by~\eqref{eqn:b};
\item Define appropriate error estimator $\eta$ which satisfies~\eqref{eqn:group-estimate};
\item Verify that $\eta$ satisfies~\eqref{e:loc-low-bnd} to prove the convergence of error indicator Theorem~\ref{thm:convergence_estimator}.
We note that the standard local lower bounds for the error indicator will guarantee~\eqref{e:loc-low-bnd}.
\end{enumerate}

In the remainder of this section, we will follow the general framework presented in Section~\ref{subsec:weak_aposteriori} (cf. \cite{Verfurth.R1994,Verfurth.R1996}) to derive {\em a posteriori} error estimates for each example.
We then verify the basic assumptions on the error estimators and the nonlinear equations.
As a consequence, we then conclude convergence of the adaptive algorithm for each example.
\subsection{Semilinear Examples: Single Equations and Systems}
In this subsection, we give two semi-linear examples.
The general formulation of a semi-linear equation is as follows:
\begin{equation}\label{eqn:semi_linear}
F(u) := Lu + N(u)=0,
\end{equation} 
where $L:X\to Y^{*}$ is a bounded linear operator, and $N(\cdot):X\to Z\subset Y^{*}$ is a $C^1$ mapping from $X$ onto a subspace $Z$ of $Y^{*}.$
We assume that 
\begin{enumerate}
\item[(S1)] $L$ satisfies the continuous as well as the discrete inf-sup conditions:
\begin{equation}\label{eqn:L_continuous_infsup}
  \inf_{x\in X, \|x\|_X=1} \sup_{y\in Y, \|y\|_Y=1} \langle Lx,y\rangle =  \inf_{y\in Y, \|y\|_Y=1} \sup_{x\in X, \|x\|_X=1} \langle Lx,y \rangle =\alpha_0>0.
\end{equation}
\begin{equation}\label{eqn:L_discrete_infsup}
  \inf_{x\in X_k, \|x\|_X=1} \sup_{y\in Y_k, \|y\|_Y=1} \langle Lx,y\rangle =  \inf_{y\in Y_k, \|y\|_Y=1} \sup_{x\in X_k, \|x\|_X=1} \langle Lx,y \rangle =\alpha_1>0.
\end{equation}
\item[(S2)] The embedding $Z\subset Y^{*}$ is compact
            as in~\eqref{eqn:z-space}.
\end{enumerate} 
First of all, we establish well-posedness of the equation~\eqref{eqn:semi_linear} under the above assumptions on $L$ and $N.$
\begin{theorem} \label{thm:semi_linear}
Let $F$ satisfy (S1) and (S2), and $X_0$ satisfy~\eqref{eqn:X0}.
If $N'(u)$ is Lipschitz continuous in a neighborhood of $u,$ then the Petrov-Galerkin problem~\eqref{eqn:nonlinear_galerkin} possesses a unique solution $u_k$ in a neighborhood of $u.$
Moreover, we have the error estimates
$$\|u-u_{k}\|_{X} \lesssim \inf_{\chi_{k}\in X_{k}} \|u - \chi_{k}\|_{X}.$$
\end{theorem}
\begin{proof}
It is straightforward to check that $F'(u)$ is Lipschitz continuous.
The following inf-sup condition was proved in~\cite[Theorem 5.1]{Rappaz.J2006}: 
$$
\inf_{x\in X_k, \|x\|_X=1} \sup_{y\in Y_k, \|y\|_Y=1} \langle F'(u)x,y\rangle = \inf_{y\in Y_k, \|y\|_Y=1} \sup_{x\in X_k, \|x\|_X=1} \langle F'(u) x,y\rangle=\beta_1>0.
$$
Then the conclusion follows by Theorem~\ref{thm:apriori_estimates}.
\end{proof}

\begin{remark}
The (global) existence and uniqueness of the solution can sometimes be proved by standard arguments in the calculus of variations.
The {\em a priori} error estimate in Theorem~\ref{thm:semi_linear} can also be proved in a different way, if {\em a priori} $L^{\infty}$ estimates on the solution $u$ and the discrete solutions $u_k$ hold.
We refer to Section ~\ref{sec:contraction-ex} for the details.
\end{remark}

\begin{example}\label{ex:semi_linear}
Consider the following semi-linear equation 
\begin{equation}\label{eqn:semi_linear2}
F(u) := -\Delta u + u^m -f  =0,
\end{equation}
with homogeneous Dirichlet boundary condition $u|_{\partial \Omega} =0.$
We assume that $m\geqs 2$ is a constant $f\in L^p(\Omega)$ for some $p>1$ satisfies $p\geqs d-\frac{d}{m}.$
\end{example}
For this nonlinear equation, we define the linear and nonlinear components of $F$ as $Lu= -\Delta u$ and $N(u) = u^m -f.$
We let $X=W_0^{1,p}(\Omega)$ with  $\|\cdot\|_{X} = |\cdot|_{1,p,\Omega}$ and $Y=W_0^{1, q}(\Omega)$ with the norm $\|\cdot\|_{Y} = |\cdot|_{1,q,\Omega},$ where $q$ satisfies $\frac{1}{p} + \frac{1}{q} =1.$
By the Sobolev Embedding Theorem and the choice of $m,$ we have $W^{1, p}(\Omega) \hookrightarrow L^{mp}(\Omega).$
Therefore, for any $u\in W_0^{1,p}(\Omega),$ we have  $N(u)\in L^p(\Omega),$ which is compact embedded in $Y^{*} =W^{-1, p}(\Omega).$
A special case when $m=3$ and $p=2$ in $\mathbb{R}^2$ can be found in Rappaz~\cite{Rappaz.J2006}.

Given a conforming triangulation ${\mathcal T}_k,$  let $X_k\subset X$ and $Y_k\subset Y$ be the piecewise linear continuous finite element space defined on $\cT_k$.
Then the finite element approximation of the equation~\eqref{eqn:semi_linear2} reads,
\begin{equation}
\label{eqn:semi_linear_fem}
\mbox{find}\;\; u_k\in X_k, \;\mbox{such that}\; \int_{\Omega}\nabla u_k\cdot \nabla v_k + u_k^m v_k -fv_k dx =0,\quad \forall v_k\in Y_k.
\end{equation}
Based on Theorem~\ref{thm:semi_linear} and Theorem~\ref{thm:apriori_estimates}, we have the following proposition.
\begin{proposition}
\label{prop:semi_linear_solvability}
If the Laplacian operator $\Delta: W_0^{1,p}(\Omega)\to W^{-1,p}(\Omega)$ is an isomorphism, then for $\|h_0(x)\|_{\infty,\Omega}$ sufficiently small, the Petrov-Galerkin problem~\eqref{eqn:semi_linear_fem} have a unique solution $u_k \in X_k$ in the neighbor of $u,$ which satisfies the {\em a priori} error estimate
$$\|u-u_k\|_{1,p,\Omega} \lesssim \min_{\chi_k\in X_k}\|u-\chi_k\|_{1,p,\Omega}.$$
\end{proposition}
\begin{proof}
It is straightforward to check $F'(u)$ is Lipschitz continuous.
By assumption on  $\Delta,$ $L=-\Delta$ satisfies the continuous inf-sup condition~\eqref{eqn:L_continuous_infsup}.
That is, we have 
$$\inf_{w\in X, \|w\|_{X} =1}\sup_{v\in Y, \|v\|_{Y} =1}(\nabla w, \nabla v)=\inf_{v\in Y, \|v\|_{Y} =1}\sup_{w\in X, \|w\|_{X} =1} (\nabla w, \nabla v) \geqs \alpha_0>0.$$
We need to show that $L$ satisfies the discrete inf-sup condition~\eqref{eqn:L_discrete_infsup}.
Let $P_k: W_0^{1,q}(\Omega) \to Y_k$ be the Galerkin projection, i.e., for any $v\in W_0^{1,q}(\Omega)$
$$(\nabla w_k, \nabla (v -P_k v)) =0, \quad\forall w_k\in X_k.$$
It is well known that 
$\|P_k v\|_{Y} \lesssim \|v\|_{Y}$, $\forall v\in Y$,
see~\cite{Rannacher.R;Scott.R1982} for example.
For any $w_k\in X_k$ with $\|w_k\|_{X}=1,$ by the continuous inf-sup condition, there exists a function $v\in W_0^{1,q}(\Omega)$ with $\|v\|_{Y}=1$ such that
$$\frac{\alpha_0}{2}\leqs (\nabla w_k, \nabla v) = (\nabla w_k, \nabla P_k v).$$
Hence
\begin{eqnarray*}
\sup_{v_k\in Y_k, \|v_k\|_{Y} =1}(\nabla w_k, \nabla v_k) &\geqs& \left(\nabla w_k, \frac{\nabla P_k v}{\|P_k v\|_{Y}}\right) \geqs \frac{\alpha_0}{2\|P_k v\|_{Y}}
\geqs  \alpha_1>0,
\end{eqnarray*}
for some constant $\alpha_1.$
In the last step, we used the stability of $P_k.$
This proves the discrete inf-sup condition:
$$
\inf_{v_k\in Y_k, \|v\|_{Y}=1}
\sup_ {w_k\in X_k, \|w_k\|_{X}=1}
(\nabla w_k, \nabla v_k) 
\geqs \alpha_1>0.
$$
The conclusion then follows by Theorem~\ref{thm:semi_linear}.
\end{proof}

Let $\sigma=\overline{\tau}\cap \overline{\tau}'$ be the interface between two elements $\tau$ and $\tau'\in {\mathcal T}_k,$ and $n_{\sigma}$ be a fixed unit normal of $\sigma.$
For any $w\in X_k,$ we denote the jump residual on $\sigma$ as $[\nabla w\cdot n_{\sigma}].$
Now we define the local error indicator
\begin{equation}
\label{eqn:semi_linear_indicator}
\eta(w, \tau)^p := h^p_{\tau} \|w^m -f\|^p_{0,p,\tau} + \sum_{\sigma\subset \partial\tau} h_{\sigma} \left\|\left[\nabla w\cdot n_{\sigma} \right]\right\|^p_{0,p, \sigma}, \quad \forall w\in X_{k},
\end{equation}
and define $\eta(w, \mathcal{S}) := \left(\sum_{\tau\in \mathcal{S}} \eta^p(w, \tau)\right)^{\frac{1}{p}}$ for any subset $\mathcal{S} \subset \cT_k.$
We also introduce the oscillation:
$${\rm osc}_k (\tau) : = h_{\tau}\|f -\pi_0 f\|_{0,p,\tau},$$ 
where $\pi_0$ is the element-wise $L^p$ projection; and denote ${\rm osc}_k (\mathcal{S}) := \left(\sum_{\tau\in \mathcal{S}} {\rm osc}_k^p(\tau)\right)^{\frac{1}{p}}$ for any subset $\mathcal{S} \subset \cT_k.$
Follow the general framework developed in~\cite{Verfurth.R1996}, we have the following {\em a posteriori} error estimates.
\begin{theorem}
\label{thm:semi_linear_posteriori}
Let $u\in W_0^{1,p}(\Omega)$ be a solution to~\eqref{eqn:semi_linear2}, and $u_k\in X_k$ be the solution of Petrov-Galerkin equation~\eqref{eqn:semi_linear_fem}.
Then for any subset $\cS \subset \cT_k$ and $v\in Y,$ we have
$$\langle F(u_k), v\rangle \lesssim \eta(u_{k},\cS) \|v\|_{Y(\Omega_{\cS})}  +  \eta(u_{k},\cT_k\setminus \cS)\|v\|_{Y(\Omega_{\cT_k\setminus \cS})}.$$
Consequently, there exists a constant $C_0$ depending only on $m$ and the shape regularity of $\cT_k$ such that
\begin{equation}
\label{eqn:semi1-upper}
\| u - u_k\|_{1, p,\Omega} \leqs C_0 \eta(u_{k},\cT_k).
\end{equation}
Furthermore, there exist constants $C_1$ and $C_2$ depending only on $m$ and the shape regularity of $\cT_k$ such that
\begin{equation}
\label{eqn:semi1-loc-lower}
\eta(u_{k}, \tau) \leqs C_1 \| u - u_k\|_{1, p,\omega_{\tau}} + C_2 osc_k(\omega_{\tau}).
\end{equation}
\end{theorem}

Finally, based on these observations, the adaptive algorithm for the nonlinear equation~\eqref{eqn:semi_linear2} is convergent.
\begin{corollary}
\label{cor:semi_linear}
The adaptive algorithm for the nonlinear equation~\eqref{eqn:semi_linear2} converges, that is, $u_k\to u$ as $k\to \infty.$
Moreover, we have
$$\lim_{k\to \infty} \eta(u_{k},\cT_k) =0.$$
\end{corollary}
\begin{proof}
By Theorem~\ref{thm:semi_linear_posteriori}, $\eta(u_{k},\cT_k)$ satisfies~\eqref{e:est-up-bnd}.
To show~\eqref{e:est-stab}, by the local lower bound and triangle inequality, we obtain:
\begin{eqnarray*}
\eta(u_{k}, \tau) &\leqs & C_1 \| u - u_k\|_{1, p,\omega_{\tau}} + C_2 osc_k(\omega_{\tau})\\
&\leqs & C_1\|u_k\|_{1, p, \omega_{\tau}} + C_1\|u\|_{1, p, \omega_{\tau}} + C_2\|h(f-\pi_0 f)\|_{0,p,\omega_{\tau}}\\
&\lesssim& \|u_k\|_{1, p, \omega_{\tau}} + \|D\|_{0,p,\omega_{\tau}},
\end{eqnarray*}
where $D$ depends only on $f$ and $u.$
Here we use the fact that $u\in W^{1,p}_0(\Omega),$ i.e., $$\|u\|_{1,p,\omega(\tau)}\leqs C$$ for some constant $C\geqs 0.$
 Also notice that the local lower bound~\eqref{eqn:semi1-loc-lower} implies~\eqref{e:loc-low-bnd}, then the conclusion follows by Theorem~\ref{thm:conv} and Theorem~\ref{thm:convergence_estimator}.
\end{proof}

As a second example of a semi-linear equation involving a system of equations, we consider the stationary incompressible Navier-Stokes problem:
\begin{example}\label{ex:navier_stokes}
Consider
\begin{equation}
\label{eqn:navier_stokes}
\left\{\begin{array} {rll}
-\nu \Delta \bu +(\bu\cdot\nabla) \bu +\nabla p &=& \mathbf{f} \;\; \mbox{ in } \Omega\subset {\mathbb{R}}^d\\
{\rm div} \bu                                           &=& 0\;\; \mbox{ in } \Omega,\\
\bu &= & 0\;\;  \mbox{ on } \partial\Omega 
\end{array}
\right.
\end{equation}
where $\nu$ is a constant viscosity of the fluid and $\mathbf{f}\in L^2(\Omega)^d$ is the given force field.
\end{example}
For this example, we let $X=Y= H_0^1(\Omega)^d\times L_0^2(\Omega)$ with the graph norm $$\|[\bv, q]\|_{X} = \left(\|\bv\|^2_{1,\Omega} + \|q\|^2_{0,\Omega}\right)^{\frac{1}{2}}.$$ Let 
\begin{eqnarray*}
\langle L[\bu, p], [\bv, q]\rangle &= &\int_{\Omega} \nu \nabla \bu:\nabla \bv - p {\rm div} \bv -q {\rm div} \bu dx\\
\langle N([\bu, p]), [\bv, q]\rangle &=& \int_{\Omega} (\bu\cdot \nabla) \bu\cdot \bv -\mathbf{f}\cdot \bv dx.
\end{eqnarray*}
When $d=2,$ since $\bu\in {H_0^1(\Omega)}^d,$ by the Sobolev Embedding Theorem, we have $\bu\in L^p(\Omega)$ for all $1<p<\infty.$
Hence, $(\bu \cdot \nabla) \bu$ is in ${L^q(\Omega)}^2$ for all $1\leqs q < 2.$
Similarly, when $d=3,$ we notice that $\bu\in L^p(\Omega)^3$ for all $1<p\leqs 6.$
Therefore, $(\bu \cdot \nabla) \bu$ is in ${L^q(\Omega)}^3$ for all $1\leqs q \leqs \frac{3}{2}.$
If we set $Z={L^{q}(\Omega)}^d$ for $\frac{6}{5}< q<\frac{3}{2},$ the property (S2) of $N$ is true with $d=2$ or 3.

On the other hand, the operator $L$ is given by the Stokes problem.
It is well-known that the continuous inf-sup condition is satisfied by $L$.
Given the triangulation $\cT_k$, we denote the finite element space $X_{k} :=V_{k}\times Q_{k} \subset X$, and assume that there holds the following discrete inf-sup condition for Stokes operator $L$ on $V_k\times Q_k:$ 
\begin{equation}
\label{eqn:ns-discrete-infsup}
\|q_{k}\|_{0,\Omega} \lesssim \sup_{\bv_{k}\in V_{k}} \frac{(\div \bv_{k}, q_{k})}{\|\bv_{k}\|_{1,\Omega}}.
\end{equation}
We refer to~\cite{Verfurth.R1984,Girault.V;Raviart.P1986,Brezzi.F;Fortin.M1991} for construction of finite element spaces that satisfy~\eqref{eqn:ns-discrete-infsup}.
By this construction, the linear operator $L$ satisfies the assumption (S1).

The Petrov-Galerkin approximation to the equation~\eqref{eqn:navier_stokes} is:
Find $[\bu_k, p_k]\in X_k$ such that 
\begin{equation}
\label{eqn:navier_stokes_fem}
\langle F([\bu_k, p_k]), [\bv_k, q_k] \rangle =0,\quad \forall [\bv_k, q_k] \in X_k,
\end{equation}
where $\langle F([\bu_k, p_k]), [\bv_k, q_k] \rangle := \langle L[\bu_{k}, p_{k}], [\bv_{k}, q_{k}]\rangle + \langle N([\bu_{k}, p_{k}]), [\bv_{k}, q_{k}]\rangle.$
Based on Theorem~\ref{thm:semi_linear} and the abstract framework in Section~\ref{sec:weak_framework}, we have the following proposition.
\begin{proposition}
\label{prop:navier_stokes_solvability}
The Petrov-Galerkin problem~\eqref{eqn:navier_stokes_fem} for the stationary Navier-Stokes equation has a unique solution $[u_k, p_{k}]$ in a neighborhood of $[u, p].$
\end{proposition}

Error indicators for this equation have been developed by several papers, see~\cite{Caloz.G;Rappaz.J1994, Verfurth.R1994, Berrone.S2001}.
For any $\tau\in {\mathcal T}_k,$ we define $\eta([\bu_{k}, p_{k}], \tau)$ as 
\begin{eqnarray*}
\eta([\bu_{k}, p_{k}], \tau)^2
   & := & h_{\tau}^2 \|-\nu \Delta \mathbf{u}_k 
          + (\mathbf{u}_k\nabla)\mathbf{u_k} 
          + \nabla p_k -\mathbf{f}\|_{0,\tau}^2 
\\
   &  + & \|{\rm div} \mathbf{u}_k\|_{0, \tau}^2 
\\
   &  + & \sum_{\sigma\subset \partial\tau} 
          h_{\sigma} \left\|[(\nu \nabla \mathbf{u}_k- p_{k})\cdot n]\right\|_{0,\sigma}^2,
\end{eqnarray*}
and  define the oscillation by
$${\rm osc}_k(\tau) : = h_{\tau} \|f-\pi_0 f\|_{0,\tau}.$$ 
As usual, for any subset $\mathcal{S}\subset {\mathcal T}_k$ we denote 
$$\eta([\bu_{k}, p_{k}], \mathcal{S}) := \left(\displaystyle\sum_{\tau\in \mathcal{S}} \eta^2([\bu_{k},p_{k}], \tau)\right)^{\frac{1}{2}} \mbox{  and   } {\rm osc}_k(\mathcal{S}) = \left(\displaystyle\sum_{\tau\in \mathcal{S}} {\rm osc}_k(\tau)^2\right)^{\frac{1}{2}}.$$ 
We then have the following {\em a posteriori} error estimates.
\begin{theorem}
\label{thm:navier_stokes_posteriori}
Let $[\bu, p]\in X$ be a locally unique solution to~\eqref{eqn:navier_stokes}, and $[\bu_k, p_k]\in X_k$ be the solution to the Petrov-Galerkin problem~\eqref{eqn:navier_stokes_fem}.
Then for any subset $\cS \subset \cT_k$ and $v\in X,$ we have the nonlinear residual estimate
\begin{eqnarray*}
\langle F([\bu_k, p_k]), [\bv, q] \rangle \lesssim &&\eta([\bu_{k},p_{k}],\cS) \|[\bv,q]\|_{X(\Omega_{\cS})} \\ 
&&+  \eta([\bu_{k},p_{k}],\cT_k\setminus \cS)\|[\bv, q]\|_{X(\Omega_{\cT_k\setminus \cS})}.
\end{eqnarray*}
Moreover, we have the following {\em a posteriori} error estimates:
\begin{eqnarray*}
&&\left( \|\bu-\bu_k\|_{1,\Omega}^2 + \|p-p_k\|_{0,\Omega}^2\right)^{\frac{1}{2}}
\lesssim  \eta([\bu_{k}, p_{k}], \cT_k),\\
&&\eta([\bu_{k}, p_{k}], \tau) \lesssim \left( \|\bu-\bu_k\|_{1,\omega_{\tau}}^2 + \|p-p_k\|_{0,\omega_{\tau}}^2\right)^{\frac{1}{2}} + {\rm osc}_k(\omega_{\tau}).
\end{eqnarray*}
\end{theorem}
 
Therefore,  the adaptive algorithm for the nonlinear equation~\eqref{ex:navier_stokes} is convergent.
\begin{corollary}
\label{cor:navier_stokes}
Let $[\bu, p]\in X$ be a locally unique solution to~\eqref{eqn:navier_stokes}, and $[\bu_k, p_k]\in X_k$ be the solution to the Petrov-Galerkin problem~\eqref{eqn:navier_stokes_fem} at each adaptive step $k.$
Then we have $[\bu_k,p_k]\to [\bu, p] $ as $k\to \infty.$
Moreover, we have
$$\lim_{k\to \infty} \eta([\bu_{k}, p_{k}], \cT_k) =0.$$
\end{corollary}
\begin{proof}
The estimate~\eqref{e:est-up-bnd} of the error estimator $\eta([\bu_{k}, p_{k}],\cT_k)$ follows by Theorem~\ref{thm:navier_stokes_posteriori}.
Similar to Corollary~\ref{cor:semi_linear}, we can easily show~\eqref{e:est-stab} by the local lower bound and triangle inequality.
 The conclusion then follows by Theorem~\ref{thm:conv} and Theorem~\ref{thm:convergence_estimator}.
\end{proof}

\subsection{A Quasi-Linear Example}

\begin{example}\label{ex:quasi_linear}
We now consider a quasi-linear example, the stationary heat equation with 
convection and nonlinear diffusion:
\begin{equation}
\label{eqn:heat}
F(u) =
- {\rm div}(\kappa(u) \nabla u)  + \mathbf{b} \cdot\nabla u - f =0 \;\; \mbox{ in } \Omega,
\end{equation}
with homogeneous boundary condition $u|_{\partial\Omega} = 0$.
We assume for all $s\in \mathbb{R}, \; \kappa(s)\in C^2(\mathbb{R})$ satisfies
$\kappa(s)\geqs \alpha\geqs 0$  and $|\kappa^{(l)}(s)|\leqs \gamma_l$,
for $l=0,1,2$, for some constants 
$\alpha, \; \gamma_0,\; \gamma_1,$ and $\gamma_2.$
We assume the vector field $\mathbf{b}\in W^{1,\infty}(\Omega)^d$ such that ${\rm div }\mathbf{b} =0,$ and $f\in L^{\infty}(\Omega).$
\end{example}
Let  $X=W_0^{1,p}(\Omega)$ and $Y = W_0^{1,q}(\Omega),$ with $\frac{1}{p}+\frac{1}{q} =1.$
As before, we let $X_k\subset X$ and $Y_k\subset Y$ be the piecewise linear continuous finite element space defined on $\cT_k$.
The finite element approximation of the equation~\eqref{ex:quasi_linear} reads,
\begin{equation}
\label{eqn:heat_fem}
\mbox{find}\;\; u_k\in X_k, \;\mbox{such that}\; \int_{\Omega}\kappa(u_k)\nabla u_k\cdot \nabla v_k + \mathbf{b}\nabla u_k v_k -fv_k dx =0,\quad \forall v_k\in Y_k.
\end{equation}
We have the following properties.
\begin{proposition}
\label{prop:heat_solvability}
Let $u\in W_0^{1,p}(\Omega)$ be a locally unique solution to the equation~\eqref{eqn:heat}.
Then the mapping $F: X \to Y^{*}$ is of class $C^1$ for $2<p<\infty,$ and $F'(u)$ is an isomorphism form $X$ to $Y^{*}.$
Moreover, if the Laplacian operator 
$\Delta: X\to Y^{*}$ is an isomorphism, then for $\|h_0(x)\|_{\infty,\Omega}$ sufficiently small, the Petrov-Galerkin problem~\eqref{eqn:heat_fem} have a unique solution $u_k \in X_k.$
\end{proposition}
\begin{proof}
We refer to~\cite{Caloz.G;Rappaz.J1994} for the proof of the first part of this proposition.
The second part of the conclusion then follows by Theorem~\eqref{thm:apriori_estimates}.
\end{proof}

For the {\em a posteriori} error estimator, we introduce
\begin{eqnarray*}
\eta(u_{k}, \tau)^p:=& h^p_{\tau} \|- {\rm div} (\kappa(u_k)\nabla u_k) + \mathbf{b} \cdot \nabla  u_k  - f \|^p_{0,p,\tau} \\
&+ \sum_{\sigma\subset \partial\tau} h_{\sigma} \left\|\left[\kappa(u_k)\nabla u_k\cdot n\right]\right\|^p_{0,p,\sigma},
\end{eqnarray*}
and define the oscillation by
\begin{eqnarray*}
{\rm osc}_k^p(\tau):= &&h_{\tau}^p \left\|(I-\pi_0)\left(-{\rm div}(\kappa(u_k)\nabla u_k) + \mathbf{b} \cdot \nabla u_k -f \right)\right\|^p_{0,p,\tau}\\
 &&+ 
\sum_{\sigma\subset \partial\tau} h_{\sigma} \left\|\left[(I-\pi_1)\kappa(u_k)\nabla u_k\cdot n\right]\right\|^p_{0,p,\sigma},
\end{eqnarray*}
where $\pi_0$ and $\pi_1$ are the element-wise $L^p$ projections onto the $\mathcal{P}_0$ and $\mathcal{P}_1$ spaces respectively.
Also, we denote 
$$\eta(u_{k},\mathcal{S}) := \left(\sum_{\tau\in \mathcal{S}} \eta^p(u_{k},\tau)\right)^{\frac{1}{p}} \mbox{  and  }{\rm osc}_k (\mathcal{S}) := \left(\sum_{\tau\in \mathcal{S}} {\rm osc}_k^p(\tau)\right)^{\frac{1}{p}}$$ for any subset $\mathcal{S} \subset \cT_k.$
Again, following the general framework in~\cite{Verfurth.R1994,Verfurth.R1996}, 
we obtain {\em a posteriori} error estimates.
\begin{theorem}
\label{thm:heat_posteriori}
Let $u\in W_{0}^{1,p}(\Omega)$ be a locally unique solution to~\eqref{eqn:heat}.
Then we have
$$\langle F(u_k), v\rangle \lesssim \eta(u_{k},\cS) \|v\|_{Y(\Omega_{\cS})}  +  \eta(u_{k},\cT_k\setminus \cS)\|v\|_{Y(\Omega_{\cT_k\setminus \cS})},$$
for any subset $\cS \subset \cT_k$ and $v\in Y.$
Furthermore, we have the following {\em a posteriori} error estimates:
\begin{eqnarray*}
\|u-u_k\|_{1,p,\Omega}&\lesssim & \eta(u_{k},\cT_k),\\
\eta(u_{k},\tau) &\lesssim & \|u-u_k\|_{1,p,\omega(\tau)} + {\rm osc}_k(\omega_{\tau}).
\end{eqnarray*}
\end{theorem}

Finally, based on the results above, the adaptive algorithm for the quasi-linear stationary heat equation~\eqref{eqn:heat} with nonlinear diffusion is convergent.
\begin{corollary}
\label{cor:heat}
Under the hypothesis of Theorem~\ref{prop:heat_solvability}, if $p\geqs 2d$ then the adaptive algorithm for the nonlinear equation~\eqref{eqn:heat} converges, that is, 
$$\lim_{k\to \infty} u_{k} = u, \quad \mbox{and} \quad \lim_{k\to \infty} \eta(u_{k},\cT_k) =0.$$
\end{corollary}
\begin{proof}
Again, the error estimator $\eta(u_{k},\cT_k)$ satisfies~\eqref{e:est-up-bnd} due to Theorem~\ref{thm:heat_posteriori}.
Now we prove that Assumption~\eqref{e:est-stab} holds.
We start with the local lower bound in Theorem~\ref{thm:heat_posteriori}:
\begin{eqnarray*}
\eta(u_{k},\tau) &\lesssim & \|u-u_k\|_{1,p,\omega(\tau)} + {\rm osc}_k(\omega_{\tau})\\
&\lesssim& \|u_k\|_{1,p,\omega(\tau)} + \|u\|_{1,p,\omega(\tau)} + {\rm osc}_k(\omega_{\tau}).
\end{eqnarray*}
Since $\mathbf{b}\in W^{1,\infty}(\Omega),$ we have 
\begin{eqnarray*}
\|(I - \pi_0) f\|_{0,p,\tau} &\leqs& \|f\|_{0,p,\tau},\\
\|(I-\pi_0)\mathbf{b} \cdot \nabla u_k\|_{0,p,\tau} &\lesssim & \|\nabla u_k\|_{0,p,\tau}.
\end{eqnarray*}
On the other hand, we have 
\begin{eqnarray*}
&&\left\|(I-\pi_0)\left(-{\rm div}(\kappa(u_k)\nabla u_k)\right)\right\|_{0,p,\tau} \\
&&\leqs  \left\|(I-\pi_{0})\kappa'(u_k)\right\|_{0,p,\tau} |\nabla u_k|^{2}\\
&&\leqs  C h_{\tau} \left\|\kappa'' \nabla u_{k}\right\|_{0,p,\tau} |\nabla u_k|^2\\
&&\leqs  C\gamma_2 h_{\tau} \|\nabla u_{k}\|_{0,p,\tau} |\nabla u_k|^2\\
&&\leqs  C\gamma_2 \|\nabla u_k\|_{0,p,\tau} \|\nabla u_k\|^{2}_{0,p,\tau}\\
&&\lesssim  \|\nabla u_k\|_{0,p,\tau}.
\end{eqnarray*}
In the third step, we used the boundedness of $\kappa''(s)$ and the fact that $\nabla u_k$ is constant in $\tau;$ in the fourth step, we used the assumption $p\ge 2d;$ and in the last step, we used the {\em a priori} estimate of $u_k,$ namely, 
$$ \|u_k\|_{1,p,\Omega} \leqs \|u-u_k\|_{1,p,\Omega} + \|u\|_{1,p,\Omega} \leqs C$$
for some constant $C$ (cf.~Theorem~\ref{thm:apriori_estimates}).
Similarly, by noticing $\kappa$ is uniformly bounded and the {\em a priori} error estimate of $u_k,$ one can easily obtain
$$\sum_{\sigma\subset \partial\tau} h_{\sigma} \left\|\left[(I-\pi_1)\kappa(u_k)\nabla u_k\cdot n\right]\right\|^p_{0,p,\sigma} \lesssim \|\nabla u_k\|_{0,p,\omega_{\tau}}.$$
Therefore, we obtain the stability estimate
$$\eta(u_{k},, \tau) \lesssim \|u_k\|_{1,p,\omega_{\tau}}  +\|D\|_{0,p,\omega_{\tau}},$$
where $D$ depending only on $u$ and $f.$
 The conclusion then follows by Theorem~\ref{thm:conv} and Theorem~\ref{thm:convergence_estimator}.
\end{proof}

%% file: contraction.tex
\section{An Abstract Contraction Framework}
\label{sec:contraction-abstract}

In this section, we develop a second distinct abstract convergence
framework to allow for establishing contraction under 
additional minimal assumptions.
The framework generalizes the AFEM contraction arguments used
in~\cite{Mekchay.K;Nochetto.R2005,Cascon.J;Kreuzer.C;Nochetto.R;Siebert.K2007,Nochetto.R;Siebert.K;Veeser.A2009,Holst.M;Tsogtgerel.G2009,Holst.M;McCammon.J;Yu.Z;Zhou.Y2009}
to general approximation techniques for abstract nonlinear problems.
The three key ingredients to the contraction argument are as in the
existing linear frameworks: 
quasi-orthogonality,
error indicator domination of the error, 
and a type of error indicator reduction.

\subsection{Quasi-Orthogonality}
\label{subsec:contraction-quasi}

One of the main tools for establishing contraction in
adaptive algorithms is perturbed- or {\em quasi}-orthogonality.
Let $X_1 \subset X_2 \subset X$ and $Y_1 \subset Y_2 \subset Y$
be triples of Banach spaces, and consider (for the moment,
arbitrary and unrelated) $u_1 \in X_1$,
$u_2 \in X_2$, and $u \in X$.
If $X$ also had Hilbert-space structure, so that
the native norm $\|\cdot\|_X$ on $X$ was
induced by an inner product $\| \cdot \|_X = (\cdot,\cdot)_X^{1/2}$,
and if orthogonality were to hold $(u-u_2,u_2-u_1)_X = 0$, then one
would have the Pythagorean Theorem:
\begin{equation}
\label{E:pythag}
\|u-u_1\|_X ^2 = \|u-u_2\|_X^2 + \|u_2-u_1\|_X^2.
\end{equation}
Quasi-orthogonality is a more general concept whereby one
gives up orthogonality~\eqref{E:pythag}, and instead works with
inequalities involving a (semi-)norm $\tbar\cdot\tbar$ that could
be the native norm $\|\cdot\|_X$, or more generally could be
an energy norm or semi-norm particularly suited to the problem at hand.
From the triangle inequality in the Banach space $X$ together with
the discrete Holder inequality, one always has the following inequality:
\begin{equation}
\label{E:ineq1}
\lambda \| u-u_1\|_X^2 
   \leqs \| u-u_2\|_X^2 + \| u_2-u_1\|_X^2,
\end{equation}
with $\lambda = 1/2$.
Quasi-orthogonality then refers to establishing the more difficult 
inequality in the other direction to supplement~\eqref{E:ineq1}:
\begin{equation}
\Lambda \| u-u_1\|_X^2 
   \geqs \| u-u_2\|_X^2 + \| u_2-u_1\|_X^2,
\end{equation}
for some $\Lambda \geqs 1$, which is convenient to write in the form
\begin{equation}
\label{E:quasi-banach}
\| u-u_2\|_X^2 
   \leqs \Lambda \| u-u_1\|_X^2 - \| u_2-u_1\|_X^2.
\end{equation}
We wish now to develop conditions for establishing~\eqref{E:quasi-banach}.
We will see shortly that it will be critical for us to be able to 
establish~\eqref{E:quasi-banach} with constant $\Lambda$ close to one; 
this will only be possible if $u-u_2$ and $u_2-u_1$ are nearly 
``orthogonal'' in some generalized sense, which implies that we must work 
with a norm related to the Petrov-Galerkin (PG) ``projection'' process,
and may require that we work with a norm other than the 
native norm $\|\cdot\|_X$.

To this end, consider a continuous bilinear form $b(\cdot,\cdot)$
on $X \times Y$:
\begin{equation}
\label{E:contraction_continuity}
b : X \times Y \rightarrow \mathbb{R},
\quad \quad
b(u,v) \leqs M \|u\|_X \|v\|_Y, \quad \quad \forall u \in X, v \in Y.
\end{equation}
Assume $b$ satisfies {\em inf-sup} conditions on $X$ and $Y$:
There exists $\beta_0>0$ such that
\begin{equation}\label{E:contraction_infsup}
  \inf_{u\in X, \|u\|_X=1} \sup_{v\in Y, \|v\|_Y=1} b(u,v) =  \inf_{v\in Y, \|v\|_Y=1} \sup_{u\in X, \|u\|_X=1} b(u,v) =\beta_0>0.
\end{equation}
In the subspaces $X_k$ and $Y_k$, $k=1,2$, we assume $b$ satisfies 
a discrete {\em inf-sup} condition:
There exists a constant $\beta_1>0$ such that
\begin{equation}\label{E:contraction_infsup_discrete}
  \inf_{u\in X_k, \|v\|_X=1} \sup_{v\in Y_k, \|v\|_Y=1} b(u,v) = \inf_{v\in Y_k, \|v\|_Y=1} \sup_{u\in X_k, \|u\|_X=1} b(u,v)\geqs\beta_1>0.
\end{equation}
Given now $f \in Y^*$,
we assume that $u\in X$ is the solution to the operator equation involving
$b$ and $f$,
and that $u_1 \in X_1$ and $u_2 \in X_2$ are corresponding PG approximations:
\begin{eqnarray}
\textrm{Find~} u \in X \textrm{~such~that~}
b(u,v) & = & f(v), \quad \quad \forall ~~ v \in Y.
\label{E:contraction_opeqn}
\\
\textrm{Find~} u_1 \in X_1 \textrm{~such~that~}
b(u_1,v_1) & = & f(v_1), \quad \quad \forall v_1 \in Y_1 \subset Y_2 \subset Y.
\label{E:contraction_opeqn_pg1}
\\
\textrm{Find~} u_2 \in X_2 \textrm{~such~that~}
b(u_2,v_2) & = & f(v_2), \quad \quad \forall v_2 \in Y_2 \subset Y.
\label{E:contraction_opeqn_pg2}
\end{eqnarray}
With this setup, we can establish the quasi-orthogonality
inequality in the norm $\|\cdot\|_X$ for PG approximations
defined by any continuous bilinear form satisfying {\em inf-sup} conditions.
\begin{theorem}
\label{T:quasi-banach}
Assume the bilinear form $b : X \times Y \rightarrow \mathbb{R}$ satisfies
the continuity~\eqref{E:contraction_continuity}
and {\em inf-sup} conditions
~\eqref{E:contraction_infsup} and~\eqref{E:contraction_infsup_discrete}.
Assume that $u$, $u_1$, and $u_2$ are defined by
\eqref{E:contraction_opeqn},
\eqref{E:contraction_opeqn_pg1},
and
\eqref{E:contraction_opeqn_pg2},
respectively.
Then quasi-orthogonality~\eqref{E:quasi-banach} holds with
\begin{equation}
\Lambda = \left( 1 + \frac{2 M}{\beta_1} \right)^2 \geqs 1.
\end{equation}
\end{theorem}

\begin{proof}
With no inner-product we have only the following type of 
generalized ``orthogonality'' to exploit:
\begin{eqnarray}
b(u-u_2,v_2) & = & 0, \quad \quad \forall v_2 \in Y_2 \subset Y,
\\
b(u-u_1,v_1) & = & 0, \quad \quad \forall v_1 \in Y_1 \subset Y_2 \subset Y,
\end{eqnarray}
which are obtained by subtracting
\eqref{E:contraction_opeqn_pg1}
and
\eqref{E:contraction_opeqn_pg2}
from~\eqref{E:contraction_opeqn}.
This leads us to:
\begin{equation}
\label{E:keyplay}
b(u-u_1,v_2) = -b(u_1-u_2,v_2)
             =  b(u_2-u_1,v_2),
\quad \quad \forall v_2 \in Y_2.
\end{equation}
Combining~\eqref{E:keyplay} with the 
{\em inf-sup} condition~\eqref{E:contraction_infsup_discrete}
and continuity~\eqref{E:contraction_continuity} gives
the estimate
\begin{equation}
\label{E:estimate}
\beta_1 \|u_2 - u_1\|_X
\leqs
\sup_{0 \ne v_2 \in Y_2} \frac{b(u_2-u_1,v_2)}{\|v_2\|_Y}
=
\sup_{0 \ne v_2 \in Y_2} \frac{b(u-u_1,v_2)}{\|v_2\|_Y}
\leqs
M \|u-u_1\|_X.
\end{equation}
Starting with the triangle inequality 
\begin{equation}
\label{E:triangle}
\|u-u_2\|_X \leqs \|u-u_1\|_X + \|u_2-u_1\|_X,
\end{equation}
we add twice~\eqref{E:estimate} to~\eqref{E:triangle} to obtain
\begin{equation}
\label{E:quasi-nosquare}
\|u-u_2\|_X + \|u_2-u_1\|_X \leqs \hat{\Lambda} \|u-u_1\|_X,
\end{equation}
with
$\hat{\Lambda} = 1 + 2 M / \beta_1 \geqs 1$.
If we square both sides we obtain
\begin{equation}
\|u-u_2\|_X^2 + 2 \|u_2-u_1\|_X \|u-u_2\|_X + \|u_2-u_1\|_X^2
\leqs \hat{\Lambda}^2 \|u-u_1\|_X^2.
\end{equation}
The second term on the left is non-negative; we drop it
to give~\eqref{E:quasi-banach} with $\Lambda = \hat{\Lambda}^2$.
\end{proof}

It is clear from the proof of Theorem~\ref{T:quasi-banach} that 
to establish~\eqref{E:quasi-banach} with constant
$\Lambda \geqs 1$ close one, one must
establish a version of~\eqref{E:contraction_continuity} that,
when restricted to particular arguments from subspaces, 
will hold with constant $M \geqs 0$ close to zero.
This then resembles some type of strengthened Cauchy inequality.
Note that if $X=Y$, then the bilinear form $b(\cdot,\cdot)$ 
defines an energy (semi-)norm through:
\begin{equation}
\label{E:energy-norm}
b : X \times X \rightarrow \mathbb{R},
\quad \quad \tbar w \tbar^2 = b(w,w), \quad \quad \forall w \in X.
\end{equation}
It will be more fruitful to consider quasi-orthogonality
with respect to this semi-norm:
\begin{equation}
\label{E:quasi}
\tbar u-u_2\tbar^2 
   \leqs \Lambda \tbar u-u_1\tbar^2 - \tbar u_2-u_1\tbar^2.
\end{equation}
To establish~\eqref{E:quasi},
it will be useful if strengthened Cauchy inequalities hold:
\begin{equation}
\label{E:cauchy}
b(u-u_1,v_1) \leqs \gamma \tbar u-u_1 \tbar \tbar v_1 \tbar,
\quad
b(v_1,u-u_1) \leqs \gamma \tbar u-u_1 \tbar \tbar v_1 \tbar,
\end{equation}
\begin{equation*}
\forall v_1 \in X_1, \quad \gamma \in [0,1).
\end{equation*}
In this case, one immediately has the following 
without {\em inf-sup} conditions:
\begin{theorem}
\label{T:quasi}
Let $X_1 \subset X_2 \subset X$ be a triple of Banach spaces,
and
let $u \in X$, $u_1 \in X_1$, and $u_2 \in X_2$ be such that
the bilinear form $b : X \times X \rightarrow \mathbb{R}$ 
satisfies the strengthened Cauchy inequality~\eqref{E:cauchy}
for some $\gamma \in [0,1)$.
Then quasi-orthogonality~\eqref{E:quasi} holds in the energy
semi-norm $\tbar\cdot\tbar^2=b(\cdot,\cdot)$ with
$\Lambda = 1 / (1 - \gamma) \geqs 1$.
\end{theorem}
\begin{proof}
We begin with the identity
\begin{eqnarray}
\tbar u - u_1 \tbar^2 
  & = & \tbar u-u_2 \tbar^2 + \tbar u_2 - u_1 \tbar^2
\nonumber
\\
  &   & + b(u-u_2,u_2-u_1) + b(u_2-u_1,u-u_2).
\label{E:expandX}
\end{eqnarray}
Using~\eqref{E:cauchy} in~\eqref{E:expandX} and Cauchy-Schwarz inequality gives
\begin{eqnarray}
\tbar u - u_1 \tbar^2 
  & \geqs & \tbar u-u_2 \tbar^2 + \tbar u_2 - u_1 \tbar^2
 - \gamma \tbar u-u_2 \tbar^2
 - \gamma \tbar u_2 - u_1 \tbar^2
\nonumber
\\
  & = & \left( 1 - \gamma \right) \left( \tbar u-u_2 \tbar^2 + \tbar u_2 - u_1 \tbar^2\right),
\end{eqnarray}
which gives~\eqref{E:quasi} after multiplication 
by $\Lambda = 1 / (1-\gamma) \geqs 1$.
\end{proof}

While Theorem~\ref{T:quasi} gives quasi-orthogonality in the
energy norm $\tbar\cdot\tbar$ without {\em inf-sup} conditions, 
it is important to point out that establishing the Cauchy inequality 
in the energy norm usually goes through the native norm $\|\cdot\|_X$,
and then relating the Cauchy inequalities in the two norms
requires additional structure such as {\em inf-sup} conditions.
To make this more clear, let us assume that $X=Y$ and $b(\cdot, \cdot):X\to X$ is
coercive for $m=\beta_0=\beta_1>0:$
\begin{equation}
\label{E:coercivity}
m \|w\|_X^2 \leqs \tbar w \tbar^2 = b(w,w),
\quad \quad \forall w \in X.
\end{equation}
To allow for a weaker condition than coercivity~\eqref{E:coercivity},
it is useful have a (Gelfand) triple of Banach
spaces $X \subset Z \subset X^*$,
with continuous embedding of $X$ into the intermediate space $Z$.
\begin{subequations}
\label{group:coercivity}
A G\aa{}rding inequality with $m>0$ and $C_G \geqs 0$
for the form $b(\cdot,\cdot)$ is then a possibility:
\begin{equation}
\label{E:gaarding}
m \|w\|_X^2 \leqs \tbar w \tbar^2 + C_G \|w\|_Z^2,
\quad \quad \forall w \in X.
\end{equation}
If $C_G=0$, then inequality~\eqref{E:gaarding}
reduces to~\eqref{E:coercivity}.
To exploit~\eqref{E:gaarding},
we need a lifting inequality between $X$ and $Z$
when $u_2 \in X_2$ and $u_1 \in X_1$ are approximations of $u \in X$:
\begin{equation}\label{E:lifting}
\|u-u_2\|_Z \leqs C_L \sigma \|u-u_2\|_X,
\quad \quad
\|u_2-u_1\|_Z \leqs C_L \sigma_0 \|u_2-u_1\|_X,
\end{equation} 
\end{subequations}
where it is assumed that $\sigma_0 = c \sigma$ for fixed $c>0$,
and that $\sigma$ can be made arbitrarily small
for sufficiently large subspaces $X_1 \subset X_2 \subset X$, where typically
$u_1 \in X_1$ and $u_2 \in X_2$ are PG approximations to $u \in X$.
Inequalities~\eqref{E:lifting} can be established using
the ``Nitsche trick'' with certain regularity assumptions;
cf.~\cite{Mekchay.K;Nochetto.R2005,Holst.M;Tsogtgerel.G2009}.
The usefulness of~\eqref{group:coercivity}
is made clear by the following Lemma~\ref{L:lifting}, due essentially 
to Schatz~\cite{Schatz.A1974}.
\begin{lemma}
\label{L:lifting}
Assume the bilinear form $b : X \times X \rightarrow \mathbb{R}$
satisfies~\eqref{E:gaarding}--\eqref{E:lifting}.
Then~\eqref{E:coercivity} holds with $w=u-u_2$ and $w=u_2-u_1$
for $\sigma$ sufficiently small,
with constant $\overline{m} = m - C_G C_L^2 \sigma^2 > 0$.
\end{lemma}
\begin{proof}
We observe that
\begin{equation}
m\|u-u_2\|_X^2 
\leqs \tbar u-u_2 \tbar^2 + C_G \|u-u_2\|_Z^2
\leqs \tbar u-u_2 \tbar^2 + C_G C_L^2 \sigma^2 \|u-u_2\|_X^2,
\end{equation}
which implies the result for $\sigma > 0$ sufficiently small.
We note that $0 < \overline{m} \leqs m$, with $\overline{m}=m$
when coercivity holds ($C_G=0$).
The same argument for $u_2-u_1$ gives the same result with
slightly different constants appearing in the argument.
\end{proof}

\subsection{Global Quasi-Orthogonality for Semilinear Problems}
\label{subsec:global-quasi}

We now consider a nonlinear problem for which we can
establish a strengthened Cauchy inequality, and then subsequently
quasi-orthogonality, globally in $X$.
We use a Lifting-type argument requiring the PG approximation space 
be sufficiently good (large).
Such an approach is used in~\cite{Mekchay.K;Nochetto.R2005}
for nonsymmetric linear problems,
and in~\cite{Holst.M;Tsogtgerel.G2009,Holst.M;McCammon.J;Yu.Z;Zhou.Y2009}
for semilinear problems.

Let $X_1 \subset X_2 \subset X$ be a triple of Banach spaces,
and consider $F:X\rightarrow X^*$ such that
\begin{equation}\label{E:semilinear_form}
F(u) = Lu + N(u), \quad L \in L(X,X^*), \quad N:X \rightarrow X^*.
\end{equation}
The operator $L$ induces a bilinear form $b : X \times X \rightarrow \mathbb{R}$
and subsequently an energy (semi-) 
norm $\tbar\cdot\tbar : X \rightarrow \mathbb{R}$, through the relations
\begin{equation}
\label{E:aform}
b(u,v) = \langle L u, v \rangle,
\quad
\forall u,v \in X,
\quad
\quad
\tbar u \tbar = b(u,u)^{1/2},
\quad
\forall u \in X.
\end{equation}
We have the equations for $u \in X$ and its PG approximations $u_1 \in X_1$ 
and $u_2 \in X_2$:
\begin{eqnarray}
b(u,v) + \langle N(u),v \rangle &=& 0, \quad \forall v \in X,
\label{E:pg-u}
\\
b(u_1,v_1) + \langle N(u_1),v_1 \rangle &=& 0, 
   \quad \forall v_1 \in X_1  \subset X_2 \subset X,
\label{E:pg-uH}
\\
b(u_2,v_2) + \langle N(u_2),v_2 \rangle &=& 0, 
   \quad \forall v_2 \in X_2 \subset X.
\label{E:pg-uh}
\end{eqnarray}
We will need the following Lipschitz property
(globally in $X$) for term $N(\cdot)$:
\begin{equation}\label{E:lipschitz}
\langle N(u) - N(u_2), v_2 \rangle \leqs K \|u-u_2\|_Z \|v_2\|_X,
\quad \forall v_2 \in X_2,
\end{equation}
where $u$ is the exact solution and $u_2$ is the
PG approximation in $X_2$, and where $Z$ is part of the triple
$X \subset Z \subset X^*$ as in Section~\ref{subsec:contraction-quasi}.
By splitting $F$ into a linear part $L$
satisfying continuity and G\aa{}rding assumptions, and a
remainder $N$ satisfying only the Lipschitz assumption,
we will be able to establish both Cauchy inequalities
and subsequently quasi-orthogonality, globally in $X$, for a large 
class of nonlinear problems.
\begin{theorem}
\label{T:quasi-semilinear}
Let $u$, $u_1$, and $u_2$ satisfy~\eqref{E:pg-u}--\eqref{E:pg-uh},
and let the Lipschitz~\eqref{E:lipschitz} 
and Lifting~\eqref{E:lifting} conditions hold.
Let the energy norm $\tbar\cdot\tbar$ induced by $b(\cdot,\cdot)$ 
as in~\eqref{E:aform} 
satisfy the G\aa{}rding inequality~\eqref{E:gaarding}.
Then for $\sigma$ sufficiently small,
$b(\cdot,\cdot)$ satisfies the 
strengthened Cauchy inequality~\eqref{E:cauchy} with 
$\gamma = K C_L \sigma/(m - C_G C_L^2 \sigma^2) \in (0,1)$,
and the quasi-orthogonality inequality~\eqref{E:quasi} holds with
\begin{equation}
\label{E:lambda_sigma}
\Lambda = \frac{1}{1-\gamma}
        = \frac{1}{1 - K C_L \sigma /(m-C_G C_L^2 \sigma^2)} \geqs 1.
\end{equation}
For sufficiently small $\sigma$,
$\gamma$ can be made arbitrarily small
and $\Lambda$ can be made arbitrarily close to one.
\end{theorem}

\begin{proof}
Subtracting~\eqref{E:pg-u} and~\eqref{E:pg-uh} with $v=v_2=u_2-u_1$, we have
\begin{eqnarray*}
|b(u-u_2,u_2-u_1)|
  & = & |- \langle b(u) - b(u_2), u_2-u_1 \rangle |
\\
  & \leqs & K \| u - u_2 \|_Z \| u_2-u_1 \|_X
\\
  & \leqs & \gamma_X \| u - u_2 \|_X \| u_2-u_1 \|_X,
\end{eqnarray*}
after using~\eqref{E:lipschitz} and~\eqref{E:lifting}, 
where $\gamma_X = K C_L \sigma \in (0,1)$ for $\sigma$ 
sufficiently small.
Since $b(\cdot,\cdot)$ satisfies the G\aa{}rding inequality~\eqref{E:gaarding},
then by Lemma~\ref{L:lifting}, we have
\begin{eqnarray*}
|b(u-u_2,u_2-u_1)|
  & \leqs & \gamma \tbar u - u_2 \tbar \tbar u_2-u_1 \tbar,
\end{eqnarray*}
where
$\gamma = \gamma_X / \overline{m} 
   = K C_L \sigma / (m - C_G C_L^2 \sigma^2) \in (0,1)$ 
for $\sigma$ sufficiently small.
By Theorem~\ref{T:quasi}, we have~\eqref{E:quasi}
holds with $\Lambda$ as in~\eqref{E:lambda_sigma},
which can be made arbitrarily close to one
for $\sigma > 0$ sufficiently small.
The conclusion then follows by Theorem~\ref{T:quasi}.
\end{proof}

\subsection{Local Quasi-Orthogonality for General Nonlinear Problems}
\label{subsec:local-quasi}

Consider now general operators $F:X\rightarrow X^*$,
where $X_1 \subset X_1 \subset X$ is a triple of Banach spaces.
We have equations for $u \in X$ and its 
PG approximations in $u_1 \in X_1$ and $u_2 \in X_2$:
\begin{eqnarray}
\langle F(u),v \rangle &=& 0, \quad \forall v \in X,
\label{E:gen-pg-u}
\\
\langle F(u_1),v_1 \rangle &=& 0, 
\quad \forall v_1 \in X_1 \subset X_2 \subset X,
\label{E:gen-pg-uH}
\\
\langle F(u_2),v_2 \rangle &=& 0, 
\quad \forall v_2 \in X_2 \subset X.
\label{E:gen-pg-uh}
\end{eqnarray}
Not having access to any additional structure in $F$ to exploit
as in the semilinear case,
we will need to work locally in an $\epsilon_0$-ball around $u \in X$
for some $\epsilon_0>0$: We assume
\begin{equation}
\label{E:gen-locality}
\|u-u_1\|_X\leqs \epsilon_0,
\quad \quad
\|u-u_0\|_X\leqs \epsilon_0.
\end{equation}
We assume $F'$ is Lipschitz in the ball:
There exists a Lipschitz constant $L>0$ such that
\begin{equation}
\label{E:gen-lipschitz}
\|F'(u)-F'(w)\|_{\mathcal{L}(X,X^*)} \leqs L\|u-w\|_X,
\quad \forall w\in X ~\mathrm{s.t.}~ \|u-w\|_X\leqs \epsilon_0.
\end{equation}
Define now the bilinear form
\begin{equation}
\label{E:gen-aform}
b : X \times X \rightarrow \mathbb{R},
\quad
b(w,v) = \langle F'(u) w, v \rangle.
\end{equation}
We then have the Cauchy inequality locally 
in the $\epsilon_0$-ball, leading to quasi-orthogonality.
\begin{theorem}
\label{T:quasi-nonlinear}
Let $u$, $u_1$, and $u_0$ satisfy~\eqref{E:gen-pg-u}--\eqref{E:gen-pg-uh},
and let the Locality~\eqref{E:gen-locality} and
Lipschitz~\eqref{E:gen-lipschitz} conditions hold,
with $b(\cdot,\cdot)$ defined as in~\eqref{E:gen-aform}.

\begin{enumerate}
\item[(1)] If $b(\cdot,\cdot)$ satisfies
coercivity~\eqref{E:coercivity},
then $b(\cdot,\cdot)$ satisfies the 
Cauchy inequality~\eqref{E:cauchy} with 
$\gamma = \epsilon_0 L/(2m) \in (0,1)$,
and quasi-orthogonality~\eqref{E:quasi} holds
for $\epsilon_0 > 0$ sufficiently small with 
\begin{equation}
\label{E:lambda_a}
\Lambda = \frac{1}{1 - \epsilon_0 L/(2m)}
     \geqs 1.
\end{equation}

\item[(2)] If $b(\cdot,\cdot)$ satisfies
G\aa{}rding~\eqref{E:gaarding}
and the Lifting~\eqref{E:lifting} inequalities,
then $b(\cdot,\cdot)$ satisfies the 
Cauchy inequality~\eqref{E:cauchy} with 
$\gamma = \epsilon_0 L/(2[m - C_G C_L^2 \sigma^2]) \in (0,1)$,
and quasi-orthogonality~\eqref{E:quasi} holds
for $\epsilon_0 > 0$ and $\sigma > 0$ sufficiently small with
\begin{equation}
\label{E:lambda_ah}
\Lambda = \frac{1}{1 - \epsilon_0 L/(2[m-C_G C_L^2 \sigma^2])}
    \geqs 1.
\end{equation}
\end{enumerate}

In either case, the constant
$\gamma$ can be made arbitrarily small,
and the constant $\Lambda$ 
can be made arbitrarily close to one,
for sufficiently small $\epsilon_0$ and $\sigma$.
\end{theorem}

\begin{proof}
Subtracting~\eqref{E:gen-pg-u} and~\eqref{E:gen-pg-uh} with
$v=v_2=u_2-u_1 \in X_2$ we have
\begin{equation}
\label{E:gen-galerkin}
\langle F(u)-F(u_2),u_2-u_1 \rangle = 0.
\end{equation}
We also have the mean-value formula:
\begin{equation}
\label{E:taylor}
F(u+w) = F(u) + F'(u)w 
   + \int_0^1 \left[ F'(u+sw) - F'(u) \right] w ~ds.
\end{equation}
Using~\eqref{E:taylor} with $w=u_2-u$ 
together with~\eqref{E:gen-galerkin} gives:
\begin{eqnarray*}
|b(u-u_2,u_2-u_1)|
& = & |-b(u_2-u,u_2-u_1)|
\\
& = & |-\langle F'(u) (u_2-u),u_2-u_1 \rangle|
\\
& = & |-\langle F(u+[u_2-u]) - F(u),u_2-u_1 \rangle
 \\
&  & + \langle 
   \int_0^1 \left[ F'(u+s[u_2-u]) - F'(u) \right] (u_2-u) ~ds, u_2-u_1
       \rangle |
\\
& \leqs &
 \left( \int_0^1 \| F'([1-s]u+su_2) - F'(u) \|_{\mathcal{L}(X,X^*)} ~ds \right)
\\
& &
 \cdot \| u-u_2 \|_X \| u_2-u_1 \|_X.
\end{eqnarray*}
Using now~\eqref{E:gen-lipschitz} and~\eqref{E:gen-locality} 
we can establish the Cauchy inequality~\eqref{E:cauchy} as follows:
\begin{eqnarray*}
|b(u-u_2,u_2-u_1)| 
& \leqs & \left( L \|u-u_2\|_X \int_0^1 s ~ds \right)
          \| u-u_2 \|_X \| u_2-u_1 \|_X
\\
& \leqs & \gamma_X \| u-u_2 \|_X \| u_2-u_1 \|_X,
\end{eqnarray*}
where $\gamma_X = \epsilon_0 L / 2 \in (0,1)$ for $\epsilon_0$ 
sufficiently small.

If $b(\cdot,\cdot)$ satisfies the 
coercivity inequality~\eqref{E:coercivity} for $m>0$, 
then we have established the Cauchy inequality~\eqref{E:cauchy},
with $\gamma = \gamma_X / m = \epsilon_0 L / (2m) \in (0,1)$ 
for $\epsilon_0$ sufficiently small.
By Theorem~\ref{T:quasi}, we have~\eqref{E:quasi}
holds with $\Lambda$ as in~\eqref{E:lambda_a},
which can be made arbitrarily close to one
for $\epsilon_0 > 0$ sufficiently small.

Instead of coercivity,
if $b(\cdot,\cdot)$ satisfies the 
G\aa{}rding~\eqref{E:gaarding} and lifting~\eqref{E:lifting} inequalities,
then by Lemma~\ref{L:lifting} we have
established the Cauchy inequality~\eqref{E:cauchy},
with 
$\gamma = \gamma_X / (m - C_G C_L^2 \sigma^2) 
   = \epsilon_0 L / [2(m-C_G C_L^2 \sigma^2)] \in (0,1)$ 
for $\epsilon_0$ sufficiently small.
By Theorem~\ref{T:quasi}, we have~\eqref{E:quasi}
holds with $\Lambda$ as in~\eqref{E:lambda_ah},
which can be made arbitrarily close to one
for $\epsilon_0 > 0$ and $\sigma > 0$ sufficiently small.
\end{proof}

\subsection{Contraction}
\label{subsec:contraction-contract}

We now establish a contraction result for
approximation techniques for nonlinear equations on Banach spaces,
which is an abstraction of the contraction arguments
in~\cite{Mekchay.K;Nochetto.R2005,Cascon.J;Kreuzer.C;Nochetto.R;Siebert.K2007,Nochetto.R;Siebert.K;Veeser.A2009,Holst.M;Tsogtgerel.G2009,Holst.M;McCammon.J;Yu.Z;Zhou.Y2009}.
Let $X_1 \subset X_2 \subset X$ be a triple of Banach spaces,
let $u \in X$, and let $u_1 \in X_1$ and $u_2 \in X_2$ be approximations 
to~$u$.
We are interested in the quality of the approximations;
as such, the following three distance measures between the 
three solutions are of fundamental importance:
\begin{equation}
\label{E:e-notation}
e_2 = \| u - u_2 \|_X, \quad \quad
e_1 = \| u - u_1 \|_X, \quad \quad
E_1 = \| u_2 - u_1 \|_X,
\end{equation}
where $\| \cdot \|_X$ is a norm on $X$;
this could be either the native Banach norm, or more
generally a norm associated with a problem-specific bilinear form.
We are interested in approximation algorithms which involve
abstract ``error indicator'' functionals that will be taken later
to be practical implementable {\em a posteriori} error indicators
commonly used in AFEM algorithms:
\begin{equation}
\label{E:eta-notation}
\eta_1 : X_1 \mapsto \mathbb{R}, 
\quad \quad 
\eta_2 : X_2 \mapsto \mathbb{R}.
\end{equation}
When written without arguments, these functionals are taken to be 
evaluated at $u_1$ and $u_2$ respectively, and represent approximations
to the error:
\begin{equation}
\eta_1 = \eta_1(u_1) \approx e_1,
\quad \quad
\eta_2 = \eta_2(u_2) \approx e_2.
\end{equation}
In order to build a contraction argument involving the errors,
we will need three fundamental assumptions relating the five 
quantities above:
\begin{assumption}[Quasi-Orthogonality]\label{A:quasi}
There exists $\Lambda \geqs 1$ such that
\begin{equation}
e_2^2 \leqs \Lambda e_1^2 - E_1^2.
\end{equation}
\end{assumption}
\begin{assumption}[Upper-Bound]\label{A:upper}
There exists $C_1 > 0$ such that
\begin{equation}
e_k^2 \leqs C_1 \eta_k^2,
\quad k=1,2.
\end{equation}
\end{assumption}
\begin{assumption}[Indicator Reduction]\label{A:reduct}
There exists $C_2 > 0$ and $\omega \in (0,1)$ such that
\begin{equation}
\eta_2^2 \leqs C_2 E_1^2 + (1-\omega) \eta_1^2.
\end{equation}
\end{assumption}
Using these three assumptions, we have the following
abstract contraction result.
\begin{theorem}[Abstract Contraction]\label{T:contraction-abstract}
Let $X_1 \subset X_2 \subset X$ be a triple of Banach spaces,
let $u \in X$, let $u_1 \in X_1$ and $u_2 \in X_2$ be approximations 
to~$u$ with error defined as in~\eqref{E:e-notation},
and let $\eta_1$ and $\eta_2$ be error indicators as
in~\eqref{E:eta-notation}.
Let the Assumptions~\ref{A:quasi}, \ref{A:upper}, and~\ref{A:reduct} hold.
Let $\beta \in (0,1)$ be arbitrary, and assume the constant
$\Lambda$ in Assumption~\ref{A:quasi} satisfies the bound:
\begin{equation}
\label{E:lambda-interval}
1 \leqs \Lambda < 1 + \frac{\beta \omega}{C_1 C_2}.
\end{equation}
Then there exists $\gamma > 0$ and $\alpha \in (0,1)$ such that:
\begin{equation}
e_2^2 + \gamma \eta_2^2 \leqs \alpha^2 \left( e_1^2 + \gamma \eta_1^2 \right),
\end{equation}
where $\gamma$ can be taken to be anything in the non-empty interval
\begin{equation}
\label{E:gamma-interval}
\frac{(\Lambda - 1)C_1}{\beta \omega} 
< \gamma <
\min \left\{ \frac{1}{C_2}, \frac{\Lambda C_1}{\beta \omega} \right\},
\end{equation}
and where $\alpha$ is subsequently given by
$\alpha^2 = \max \{ \alpha_1^2, \alpha_2^2 \} \in (0,1)$,
with
\begin{equation}
\label{E:alphas}
0 < \alpha_1^2 = \Lambda - \frac{\beta \omega \gamma}{C_1} < 1,
\quad \quad
0 < \alpha_2^2 = 1 - [1 - \beta] \omega < 1.
\end{equation}
\end{theorem}
\begin{proof}
Beginning with Assumption~\ref{A:quasi}, we have for any $\gamma > 0$,
\begin{equation}
e_2^2 + \gamma \eta_2^2 \leqs \Lambda e_1^2 - E_1^2 + \gamma \eta_2^2.
\end{equation}
Using now Assumption~\ref{A:reduct}, we have
\begin{equation}
e_2^2 + \gamma \eta_2^2 
\leqs 
\Lambda e_1^2 - E_1^2 
+ \gamma \left[ C_2 E_1^2 + (1-\omega) \eta_1^2 \right].
\end{equation}
Assume now that $0 < \gamma \leqs 1 / C_2$.
In this case, the negative term involving $E_1^2$ dominates the
positive term, which implies:
\begin{equation}
e_2^2 + \gamma \eta_2^2 
\leqs 
\Lambda e_1^2 + \gamma (1-\omega) \eta_1^2.
\end{equation}
We now split the negative contribution involving $\eta_1^2$ into
two parts, using any $\beta \in (0,1)$:
\begin{equation}
e_2^2 + \gamma \eta_2^2 
\leqs 
\Lambda e_1^2 - \beta \omega \gamma \eta_1^2 
   + \gamma (1-[1-\beta] \omega) \eta_1^2.
\end{equation}
We now finally invoke Assumption~\ref{A:upper} on the first term
involving $\beta$:
\begin{equation}
e_2^2 + \gamma \eta_2^2 
\leqs
\left( \Lambda - \frac{\beta \omega \gamma}{C_1} \right) e_1^2
   + \gamma (1-[1-\beta] \omega) \eta_1^2
= \alpha_1^2 e_1^2 + \alpha_2^2 \gamma \eta_1^2,
\end{equation}
where $\alpha_1^2$ and $\alpha_2^2$ are as in~\eqref{E:alphas}.
Note $\omega \in (0,1)$ from Assumption~\ref{A:reduct}, 
and also for any $\beta \in (0,1)$ it holds 
$1-\beta \in (0,1)$ and $[1 - \beta]\omega \in (0,1)$.
Therefore, for any $\beta \in (0,1)$ we have that
$\alpha_2^2$ satisfies the second inequality in~\eqref{E:alphas}.
It remains to determine $\gamma > 0$ so that $0 < \alpha_1^2 < 1$,
with $\alpha_1^2$ as given in~\eqref{E:alphas}, leading to
\begin{equation}
\frac{(\Lambda - 1)C_1}{\beta \omega} 
< \gamma < \frac{\Lambda C_1}{\beta \omega}.
\end{equation}
We have already imposed $\gamma > 0$ and $\gamma \leqs 1 / C_2$.
Recalling $\Lambda \geqs 1$,
to ensure $\alpha_1^2 \in (0,1)$ we must have $\gamma$ in the 
the interval~\eqref{E:gamma-interval}.
If $\Lambda = 1$, this interval is clearly
non-empty for any $C_1 > 0$, $C_2 > 0$, $\beta \in (0,1)$, 
and $\omega \in (0,1)$.
If $\Lambda > 1$, 
since the term involving $\Lambda$ in the upper-bound always dominates
the lower bound, to ensure the interval for $\gamma$ is non-empty we must
restrict $\Lambda$ so that
$(\Lambda - 1)C_1/(\beta \omega) < 1/C_2$.
This holds if $\Lambda$ lies in the interval~\eqref{E:lambda-interval}.
We now simply note that this interval for $\Lambda$ is
non-empty for any $C_1 > 0$, $C_2 > 0$, $\beta \in (0,1)$, 
and $\omega \in (0,1)$.
To finish the proof, we now take 
$\alpha^2 = \max\{\alpha_1^2,\alpha_2^2\} \in (0,1)$.
\end{proof}

We now establish the main contraction and convergence result we are after.

\begin{theorem}[Abstract Convergence]\label{T:convergence-abstract}
Let $\{X_k\}_{k=1}^{\infty}$, $X_k \subset X_{k+1} \subset X$, 
$\forall k\geqs 0$, be a nested sequence of Banach spaces.
Let $u \in X$, 
and let $\{ u_k \}_{k=0}^{\infty}$ be a sequence of
approximations to~$u$ from $X_k$.
Let the Assumptions~\ref{A:quasi}, \ref{A:upper}, and~\ref{A:reduct} hold
with the same constants $\Lambda$, $C_1$, $C_2$, and $\omega$,
for any successive pair of approximations $u_k$ and $u_{k+1}$ and their
corresponding error indicators $\eta_k$ and $\eta_{k+1}$.
Let $\alpha$, $\beta$, $\gamma$, and $\Lambda$ be as in
Theorem~\ref{T:contraction-abstract}.
Then the sequence $\{ u_k \}_{k=1}^{\infty}$ contracts 
toward $u \in X$ according to:
\begin{equation}
e_{k+1}^2 + \gamma \eta_{k+1}^2
\leqs \alpha^2 \left(
e_k^2 + \gamma \eta_k^2
\right),
\end{equation}
and therefore converges to $u \in X$ at the following rate:
\begin{equation}
    e_k^2 + \gamma \eta_k^2 \leqs C \alpha^{2k},
\end{equation}
for some constant 
$C = C(u_1, \eta_1, \Lambda, C_1, C_2, \alpha, \beta, \gamma, \omega)$.
\end{theorem}
\begin{proof}
Both results follow immediately from Theorem~\ref{T:contraction-abstract}.
\end{proof}

%% file: contraction-ex.tex
\section{Convergence Based on Contraction and Some Examples} 
\label{sec:contraction-ex}

Here we use the abstract contraction result 
(Theorem~\ref{T:contraction-abstract})
established in Section~\ref{sec:contraction-abstract} to prove
a contraction result (Theorem~\ref{T:contraction-afem} below)
for the adaptive algorithm described in Section~\ref{sec:adaptive}.
Theorem~\ref{T:contraction-abstract} was
based on three core assumptions: Quasi-orthogonality,
Indicator domination of the error,
and Indicator Reduction.
We showed how to establish the first of these, namely the
Quasi-Orthogonality Assumption~\ref{A:quasi}, 
for PG approximations for two general classes of nonlinear problems 
in Section~\ref{subsec:contraction-quasi}.
The second assumption, namely the Indicator Domination Error 
Assumption~\ref{A:upper}, is a standard result for residual-type indicators; 
our adaptive algorithm produces indicators with this property, 
cf.~\eqref{eqn:global_upper}.
We focus on establishing the third assumption, namely
the Indicator Reduction Assumption~\ref{A:reduct},
in Section~\ref{subsec:contract} below, and then
prove the main contraction result in Theorem~\ref{T:contraction-afem}
for the adaptive algorithm~\eqref{eqn:adaptive},
based on Theorem~\ref{T:contraction-abstract}.
We then apply this contraction result to several nonlinear PDE examples
in Sections~\ref{subsec:semilinear-contraction}--\ref{subsec:hc}.

\subsection{Contraction of AFEM} 
\label{subsec:contract}

What remains in order to use the abstract contraction result in
Theorem~\ref{T:contraction-abstract} for AFEM is the third assumption, namely
the Indicator Reduction Assumption~\ref{A:reduct}.
Following~\cite{Cascon.J;Kreuzer.C;Nochetto.R;Siebert.K2007,Nochetto.R;Siebert.K;Veeser.A2009}, we will first reduce establishing 
Assumption~\ref{A:reduct} to a simpler {\em local Lipschitz} assumption 
on the indicator, namely Assumption~\ref{A:local-lipschitz} below.
Establishing Assumption~\ref{A:reduct} will then reduce to an assumption
on the marking strategy in the AFEM algorithm; we satisfy this assumption
by using the standard D\"orfler strategy~\eqref{E:dorfler-property}.
{\em Admissible discrete functions} in Assumption~\ref{A:local-lipschitz}
refer to discrete functions which are known {\em a priori} to satisfy
specific properties of discrete PG approximations, such as discrete
{\em a priori} bounds.
We will later show how to establish Assumption~\ref{A:local-lipschitz} for 
several nonlinear problems in Section~\ref{sec:examples}
using continuous and discrete {\em a priori} bounds.

To simplify the presentation below, we will denote
$$
e_k = \tbar u-u_k \tbar,
\hspace{0.5cm}
E_k = \tbar u_k - u_{k+1} \tbar,
$$
$$ 
\eta_k = \eta(u_k,\cT_k),
\hspace{0.5cm}
\eta_k(\cM_k) = \eta(u_k,\cM_k),
\hspace{0.5cm}
\eta_0(\mathbf{D}) = \eta_0(\mathbf{D},\cT_0),
$$
where $\mathbf{D}$ represents the set of problem coefficients and nonlinearity.
We also denote $V_k:=V_D(\cT_k)$ for simplicity. 


\begin{assumption}[Local Lipschitz]
   \label{A:local-lipschitz}
Let $\cT$ be a conforming partition.
For all $\tau \in \cT$ and for any pair of admissible discrete functions
$v,w \in X(\cT),$ it holds that
\begin{eqnarray}
|\eta(v,\tau) - \eta( w,\tau)|
&\leqs& \bar{\Lambda}_1 \eta(\mathbf{D},\tau) \|v-w\|_{1,2,\omega_{\tau}},
\end{eqnarray}
where $\bar{\Lambda}_1 > 0$ depends only on the shape-regularity of $\cT_0,$
and where $\eta(\mathbf{D},\tau)$ depends only on appropriate norm behavior
of the equation coefficients over the local one-ring of elements surrounding 
$\tau$, and on the Lipschitz properties on $\tau$ of the nonlinearity 
acting on admissible functions in $X(\cT)$.
The parameter $\eta(\mathbf{D},\tau)$ is assumed to be monotone
non-increasing with mesh refinement.
\end{assumption}

Based on Assumption~\ref{A:local-lipschitz},
we have the following indicator reduction result
(see also~\cite{Holst.M;Tsogtgerel.G2009,Holst.M;McCammon.J;Yu.Z;Zhou.Y2009}), 
which extends the linear case appearing 
in~\cite{Cascon.J;Kreuzer.C;Nochetto.R;Siebert.K2007,Nochetto.R;Siebert.K;Veeser.A2009}
to the nonlinear case.
The proof is essentially identical to 
that of~\cite[Corollary~4.4]{Cascon.J;Kreuzer.C;Nochetto.R;Siebert.K2007},
except that it allows for nonlinearity in Assumption~\ref{A:local-lipschitz};
we include it for completeness.
The main difficulty in the nonlinear case will be establishing
Assumption~\ref{A:local-lipschitz} and simultaneously satisfying the 
assumption on the parameter $\lambda$ appearing in 
Lemma~\ref{L:indicator-reduction}.
\begin{lemma}[Nonlinear Indicator Reduction]
   \label{L:indicator-reduction}
Let $\cT$ be a partition, and let 
the parameters $\theta \in (0,1]$ and $\ell \geqs 1$ be given. 
Let $\cM=\textsf{MARK}(\{\eta(v,\tau)\}_{\tau\in \cT},\cT,\theta)$,
and let $\cT_* = \textsf{REFINE}(\cT,\cM, \ell).$
If $\Lambda_1 = (d+1)\bar{\Lambda}_1^2/\ell$ with $\bar{\Lambda}_1$
from Assumption~\ref{A:local-lipschitz} and $\lambda = 1-2^{-(\ell/d)} > 0$,
then for all admissible $v \in X(\cT),$ 
$v_* \in X(\cT_*),$ 
and any $\delta > 0$, it holds that
$$
\eta^2(v_*,\cT_*) 
  \leqs (1 + \delta) [ \eta^2(v,\cT) - \lambda \eta^2(v,\cM) ]
 + (1 + \delta^{-1}) \Lambda_1 \eta^2(\mathbf{D},\cT_0)
                  \tbar v_* - v \tbar^2.
$$
\end{lemma}
\begin{proof}
The proof follows that 
in~\cite[Corollary~4.4]{Cascon.J;Kreuzer.C;Nochetto.R;Siebert.K2007},
with minor adjustment to allow the Lipschitz parameter
in Assumption~\ref{A:local-lipschitz} to depend on point-wise
behavior of admissible functions in an $L^{\infty}$ interval;
we outline the argument here for completeness.
Using Assumption~\ref{A:local-lipschitz} with $v$ and $v_*$
taken to be in $X(\cT_*),$ gives
$$
\eta(v_*,\tau_*)
\leqs 
\eta(v,\tau_*)
+
\bar{\Lambda}_1 \eta(\mathbf{D},\tau_*)
    \|v_*-v\|_{1,2,\omega_{\tau_*}} \quad \forall \tau_*\in \cT_*.
$$
After squaring both sides and
applying Young's inequality with arbitrary $\delta > 0$ we have
$$
\eta^2(v_*,\tau_*)
\leqs 
(1+\delta) \eta^2(v,\tau_*)
+
(1+\delta^{-1}) \bar{\Lambda}_1^2 \eta^2(\mathbf{D},\tau_*)
    \|v_*-v\|_{1,2,\omega_{\tau_*}}^2 \quad \forall \tau_*\in \cT_*.
$$
We now sum over the elements $\tau_* \in \cT_*,$ using the fact
that for shape regular partitions there is a small finite
number of elements in the overlaps of the patches $\omega_{\tau_*}$
that are multiply represented in the sum.
This gives
$$
\eta^2(v_*,\cT_*)
\leqs 
(1+\delta) \eta^2(v,\cT_*)
+
(1+\delta^{-1}) \Lambda_1^2 \eta^2(\mathbf{D},\cT_*)
    \tbar v_*-v \tbar^2,
$$
where we have used equivalence between the energy norm and the
norm on $H^1$ (based on either coercivity or a G\aa{}rding inequality
together with lifting; cf.~Lemma~\ref{L:lifting}),
and then absorbed both the norm equivalence constant
and the finite over-representation of elements in the sum into the 
new constant $\Lambda_1$.

Now take admissible $v\in X(\cT)$;
a short argument from the proof of 
Corollary 4.4 in~\cite{Cascon.J;Kreuzer.C;Nochetto.R;Siebert.K2007} 
gives
\begin{equation}
   \label{eqn:eta-reduct}
\eta^2(v,\cT_*) \leqs \eta^2(v,\cT \setminus \cM) + 2^{-(\ell/d)} \eta^2(v,\cM)
  = \eta^2(v,\cT) - \lambda \eta^2(v,\cM).
\end{equation}
Finally, monotonicity
$\eta(\mathbf{D},\cT_*) \leqs \eta(\mathbf{D},\cT_0),$ combined with~\eqref{eqn:eta-reduct} yields the result.
\end{proof}
\begin{remark}
\label{rem:lemma-indicator-reduction}
The difficulty in the nonlinear case will be establishing
Assumption~\ref{A:local-lipschitz} and simultaneously satisfying the 
assumption on the parameter $\lambda$ appearing in 
Lemma~\ref{L:indicator-reduction}.
In the case of problems for which we can control the nonlinearity
using {\em a priori} $L^{\infty}$ control of solutions and discrete
approximations, we will be able to establish 
Assumption~\ref{A:local-lipschitz}; several such examples of
increasing difficulty are analyzed
in Sections~\ref{subsec:semilinear-contraction}--\ref{subsec:hc}.
The assumption on $\lambda$ appearing in 
Lemma~\ref{L:indicator-reduction} is essentially the assumption that
the residual indicator contains only terms that decay as with $h^{\alpha}$
for some $\alpha > 0$.
\end{remark}

We will now make use of the D\"orfler marking strategy 
\eqref{E:dorfler-property}.
This simple marking strategy will ensure that
the abstract indicator reduction Assumption~\ref{A:reduct}
holds.
\begin{lemma}
\label{L:indicator-reduction-abstract}
Let the conditions for Lemma~\ref{L:indicator-reduction} hold.
Let the Dorfler marking property~\eqref{E:dorfler-property}
hold for some $\theta \in (0,1]$, and restrict $\delta > 0$
in Lemma~\ref{L:indicator-reduction} so that
\begin{equation}
\label{E:delta-interval}
0 < \delta < \frac{\lambda \theta^2}{1 - \lambda \theta^2}.
\end{equation}
Then Indicator Reduction Assumption~\ref{A:reduct} holds with 
$C_2 = (1+\delta^{-1}) \Lambda_1 \eta^2(\mathbf{D},\cT_0)$
and
\begin{equation}
\omega = 1 - (1+\delta)(1-\lambda \theta^2) \in (0,1).
\end{equation}
\end{lemma}
\begin{proof}
By Lemma~\ref{L:indicator-reduction} we have for any $\delta > 0$:
$$
\eta^2(v_*,\cT_*) 
  \leqs (1 + \delta) [ \eta^2(v,\cT) - \lambda \eta^2(v,\cM) ]
 + (1 + \delta^{-1}) \Lambda_1 \eta^2(\mathbf{D},\cT_0)
                  \tbar v_* - v \tbar^2.
$$
The Dorfler marking property~\eqref{E:dorfler-property} gives
$$
\eta^2(v_*,\cT_*) 
  \leqs (1 + \delta) (1 - \lambda \theta^2) \eta^2(v,\cT) 
 + (1 + \delta^{-1}) \Lambda_1 \eta^2(\mathbf{D},\cT_0)
                  \tbar v_* - v \tbar^2,
$$
which we will write as
$$
\eta_{k+1}^2 \leqs C_2 E_k^2 + (1-\omega) \eta_k^2,
$$
with
\begin{equation}
\eta_{k+1} = \eta(v_*,\cT_*),
\quad \quad
\eta_k = \eta(v,\cT),
\quad \quad
E_k    = \tbar v_* - v \tbar,
\end{equation}
\begin{equation}
C_2    = (1+\delta^{-1}) \Lambda_1 \eta^2(\mathbf{D},\cT_0),
\quad \quad
(1-\omega) = (1+\delta) (1 - \lambda \theta^2).
\end{equation}
To ensure that $\omega = 1 - (1+\delta) (1 - \lambda \theta^2) \in (0,1)$,
we restrict $\delta > 0$ so that
\begin{equation}
0 < (1+\delta) (1 - \lambda \theta^2) < 1,
\end{equation}
or so that
\begin{equation}
-1 < \delta < \frac{1}{1 - \lambda \theta^2} -1 
     = \frac{1 - [1-\lambda \theta^2]}{1 - \lambda \theta^2}
     = \frac{\lambda \theta^2}{1 - \lambda \theta^2}.
\end{equation}
Since we must also take $\delta > 0$, we have then the range 
for $\delta$ is as in~\eqref{E:delta-interval}
to ensure Assumption~\ref{A:reduct} with 
$\omega = 1 - (1+\delta) (1 - \lambda \theta^2) \in (0,1)$.
\end{proof}
\begin{remark}
By first establishing Theorem~\ref{T:contraction-abstract} based only
on three simple assumptions relating the error and error indicator,
the main contraction argument in Theorem~\ref{T:contraction-abstract} is
general, applies to nonlinear problems, and does not involve details of 
the adaptive algorithm that produces the approximations 
or the error indicators.
The local Lipschitz and marking assumptions we use above to establish
the indicator reduction assumption bring in the details of the particular 
adaptive algorithm and the problem only at the last moment, and helps
clarify the impact of the various parameters on the contraction argument 
and rate.
\end{remark}


The supporting results we need are now in place;
we can now establish the second of the two main convergence results 
of the paper, this one concerning contraction.
\begin{theorem}[Contraction]\label{T:contraction-afem}
Let Assumption~\ref{A:quasi} (Quasi-Orthogonality)
and Assumption~\ref{A:upper} (Upper-Bound) hold,
and assume that the conditions of 
Lemma~\ref{L:indicator-reduction-abstract} hold.
Let $\beta \in (0,1)$ be arbitrary, 
and assume the constant $\Lambda$ in Assumption~\ref{A:quasi} satisfies:
\begin{equation}
1 \leqs \Lambda < 1 + \frac{\beta \omega}{C_1 C_2},
\end{equation}
where the constant $C_1$ is as in Assumption~\ref{A:upper},
and $C_2$ and $\omega$ are as in Lemma~\ref{L:indicator-reduction-abstract}.
Then there exists $\gamma > 0$ and $\alpha \in (0,1)$ such that:
\begin{equation}
\tbar u - u_{k+1} \tbar^2 + \gamma \eta_{k+1}^2
\leqs \alpha^2 \left(
\tbar u - u_k \tbar^2 + \gamma \eta_k^2
\right),
\end{equation}
where $\gamma$ can be taken to be anything in the non-empty interval
\begin{equation}
\frac{(\Lambda - 1)C_1}{\beta \omega} 
< \gamma <
\min \left\{ \frac{1}{C_2}, \frac{\Lambda C_1}{\beta \omega} \right\},
\end{equation}
and where $\alpha$ is subsequently given by
\begin{equation}
0 < \alpha^2 = \max \{ \alpha_1^2, \alpha_2^2 \} < 1,
\end{equation}
with
\begin{equation}
0 < \alpha_1^2 = \Lambda - \frac{\beta \omega \gamma}{C_1} < 1,
\quad \quad
0 < \alpha_2^2 = 1 - [1 - \beta] \omega < 1.
\end{equation}
\end{theorem}
\begin{proof}
By Lemma~\ref{L:indicator-reduction-abstract},
Assumption~\ref{A:local-lipschitz} and Property~\eqref{E:dorfler-property}
together imply that Assumption~\ref{A:reduct} holds.
The result then follows by Theorem~\ref{T:contraction-abstract}.
\end{proof}

We now apply the Contraction Theorem~\ref{T:contraction-afem}
to establish contraction of the 
adaptive algorithm~\eqref{eqn:adaptive} for specific nonlinear PDE examples.
Note that the more general weak*-convergence framework is also applicable to 
each of these examples, as we discussed in Section~\ref{sec:examples};
what we gain here is fixed-rate contraction
of the error at each iteration of AFEM, and subsequently the possibility
of establishing optimality of AFEM.
In each of the following examples, we use the standard residual error 
indicator, denoted by $\eta.$
For the marking strategy in the AFEM algorithm, we use the standard D\"orfler 
marking strategy~\eqref{E:dorfler-property}.

\subsection{A Semi-Linear Example}
\label{subsec:semilinear-contraction}

Our first example is a special case of equation~\eqref{eqn:semi_linear2}.
\begin{example}\label{ex:semi_linear_contraction}
    Let $\Omega\subset \R^2$ be a convex polygonal domain, and 
    $f\in L^2(\Omega).$
    Consider the weak form formulation of the semi-linear 
    equation~\eqref{eqn:semi_linear2}
    \begin{equation}\label{eqn:semi_linear_weak}
    \mbox{Find } u\in H_0^1(\Omega), \mbox{ s.t. } (\nabla u, \nabla v) 
    + (u^3, v)  = (f, v),\quad \forall v\in H_0^1(\Omega).   
     \end{equation}
\end{example}
\noindent
Here, the solution and test spaces are the 
Hilbert space $X=Y:=H_{0}^{1}(\Omega).$
Let $X_k=Y_{k}\subset H_0^1(\Omega)$ be the continuous piecewise linear
finite element spaces defined on $\cT_k,$ which we assume to be an exact
partition of $\Omega$.
For convenience, we denote 
$a(u,v) := (\nabla u, \nabla v)$ and $N(u) = u^3.$
It is not difficult to see that 
\begin{eqnarray}
(N(u) - N(v), u-v) &\geqs& 0, \quad\forall u,v \in H_0^1(\Omega),
  \label{eqn:pbe-monotone} \\
\|N(u)- N(v)\|_{\cL(H^1(\Omega),H^{-1}(\Omega))} 
&\lesssim& \|u - v\|_{0,\Omega}, \quad \forall u,v \in L^{\infty}(\Omega).
  \label{eqn:pbe-lipschitz}
\end{eqnarray}
The Galerkin approximation of the equation~\eqref{ex:semi_linear_contraction} 
then reads
\begin{equation}
\label{eqn:semi_linear_weak_fem}
\mbox{Find}\;\; u_k\in X_k, \;\mbox{such that}\; 
   a(u_k, v_k) + (N(u_k), v_k) = (f,v_k),\quad \forall v_k\in X_k.
\end{equation}
Existence and uniqueness of solutions to~\eqref{eqn:semi_linear_weak} 
and~\eqref{eqn:semi_linear_weak_fem} follow by standard variational
or fixed-point arguments, cf.~\cite{Stru00,Kesa89}.
To establish both {\em a priori} and {\em a posteriori} error estimates,
we will need $L^{\infty}$ control of the continuous solution $u$ and 
as well as the discrete solutions $u_k$.
\begin{lemma}[Continuous {\em A Priori} Estimates]
    \label{lm:apriori-u}
    Let $u\in H_0^1(\Omega)$ be the exact solution 
    to~\eqref{eqn:semi_linear_weak}.
    Then $u\in L^{\infty}(\Omega)$.
\end{lemma}
\begin{proof}
    We split the solution $u= u^l + u^n,$ where $u^l$ is the solution to the linear equation 
    $$(\nabla u^l, v) = (f, v),\quad \forall v\in H_0^1(\Omega).$$
    Since $\Omega$ is convex, elliptic regularity theory implies $u^l \in H^2(\Omega)\cap H_0^1(\Omega),$ hence $u^l\in L^{\infty}(\Omega).$
    It remains to show that $u^n\in L^{\infty}.$
    Using arguments similar to~\cite{Jerome.J1985,Chen.L;Holst.M;Xu.J2007},
    define 
    \begin{equation}
\alpha = \mathop{\textrm{arg~max}}_{c}
    \left\{(c + \sup_{x\in \Omega} u^l)^3 \leqs 0\right\},
\quad
\beta  = \mathop{\textrm{arg~min}}_{c}
    \left\{(c + \inf_{x\in \Omega} u^l)^3 \geqs 0\right\}.
    \end{equation}
    Let $\overline{\phi} = (u^n -\beta)^+ :=\max\{u^n - \beta, 0\}$ and  $\underline{\phi} = (u^n -\alpha)^- :=\min\{u^n - \alpha, 0\}.$
    Then obviously $\overline{\phi}, \underline{\phi}\in H_0^1(\Omega).$
Hence, for $\phi = \overline{\phi}$ or $\phi = \underline{\phi}$ we have
    $$(\nabla u^n , \nabla \phi)  = - ((u^n + u^l)^3, \phi)\leqs 0.$$
    This implies $0\leqs \|\nabla \phi\| \leqs 0,$ so $\phi =0.$
    Thus $\alpha \leqs u^n \leqs \beta$ almost everywhere in $\Omega$.
\end{proof}
In order to establish {\em a priori} $L^{\infty}$ bounds for $u_k,$ we require the mesh satisfy the regularity condition
\begin{equation}
\label{eqn:mesh-condition}
a_{i,j} = \int_{\Omega}\nabla \phi_i \nabla \phi_j \leqs 0,\quad  j\neq i.
\end{equation}
See for example~\cite{Chen.L;Holst.M;Xu.J2007} for a discussion of this
condition.
We then have the following {\em a priori} $L^{\infty}$ estimate for the 
discrete solution $u_k.$
\begin{lemma}[Discrete {\em A Priori} Estimates]
    \label{lm:discrete-apriori}
    Let $u_k\in X_k \subset H_0^1(\Omega)$ be the exact solution 
    to~\eqref{eqn:semi_linear_weak_fem}.
    Assume the triangulation $\cT_k$ of $\Omega$ 
    satisfies~\eqref{eqn:mesh-condition}. 
    Then $u_k\in L^{\infty}(\Omega)$.
\end{lemma}
\begin{proof}
    See~\cite{Jerome.J1985,Chen.L;Holst.M;Xu.J2007}.
\end{proof}
Lemma~\ref{lm:apriori-u} and Lemma~\ref{lm:discrete-apriori} provide
{\em a priori} $L^{\infty}$ bounds for $u$ and $u_{k}.$
That is, if $u$ and $u_k$ are exact solutions to 
\eqref{eqn:semi_linear_weak} and~\eqref{eqn:semi_linear_weak_fem},
then they must satisfy
$$
u_{-}(x) \leqs u(x), u_{k}(x) \leqs u_{+}(x), 
\textrm{~for~almost~every~} x \in \Omega,
$$
where $u_{-}, u_{+}\in L^{\infty}$ are fixed {\em a priori} bounds.
In other words, we know that any solutions $u$ and $u_k$ to
\eqref{eqn:semi_linear_weak} and~\eqref{eqn:semi_linear_weak_fem}
can be found in $[u_{-}, u_{+}] \cap H_{0}^{1}(\Omega)$,
so that we do not have to look in the larger space $H_0^1(\Omega)$
for $u$ and $u_k$.
We now have the tools in place for establishing the following
quasi-optimal {\em a priori} error estimate for Galerkin approximations.
\begin{proposition}[Quasi-Optimal {\em A Priori} Error Estimates]
    \label{prop:quasi-optimal}
    Let $u$ and $u_k$ be exact solutions to~\eqref{eqn:semi_linear_weak} 
    and~\eqref{eqn:semi_linear_weak_fem}, respectively. 
    If both $u,u_k\in L^{\infty}(\Omega),$ then we have 
     $$|u-u_k|_{1,\Omega} \lesssim \min_{\chi_k\in X_k}|u-\chi_k |_{1,\Omega}.$$
\end{proposition}
\begin{proof}
    Note that the error $u-u_k$ satisfies that 
    $$
        b(u-u_k, v_k) + (N(u) - N(u_k), v_k) =0, \quad \forall v_k \in X_k.
    $$
    Therefore, we have 
    \begin{eqnarray*}
        |u-u_k|_{1,\Omega}^2 &=& a(u-u_k,u-u_k ) \\
        &=& a(u-u_k, u-v_k) + a(u-u_k, v_k-u_k)\\
        &\leqs& |u-u_k|_{1,\Omega} |u-v_k|_{1,\Omega} - (N(u) - N(u_k), v_k -u_k)\\
        &=& |u-u_k|_{1,\Omega} |u-v_k|_{1,\Omega} - (u^3 - u_k^3, u -u_k) + (u^3 - u_k^3, u-v_k)\\
        &\lesssim& |u-u_k|_{1,\Omega} |u-v_k|_{1,\Omega}.
    \end{eqnarray*}
    Here, we noted that
$(N(u) - N(u_k), u -u_k) \geqs 0$ by~\eqref{eqn:pbe-monotone} 
and 
    \begin{eqnarray*}
         (N(u) - N(u_k), u-v_k) &\lesssim& \sup_{x\in \Omega}(\theta u(x) + (1-\theta) u_k(x))^2\|u-u_k\| \|u-v_k\|\\
          &\lesssim& |u-u_k|_{1,\Omega} |u-v_k|_{1,\Omega},
    \end{eqnarray*}
by {\em a priori} $L^{\infty}$ bounds for $u$ and $u_k$ and the Poincar\'e inequality.
Since $\chi_{k}\in X_{k}$ is arbitrary, we have $|u-u_k|_{1,\Omega} \lesssim \min_{\chi_k\in X_k}|u-\chi_k |_{1,\Omega}.$
\end{proof}
\begin{remark}
We note a major difference between Proposition~\ref{prop:quasi-optimal} 
and Proposition~\ref{prop:semi_linear_solvability} is that
Proposition~\ref{prop:quasi-optimal} does not require the 
initial mesh to be sufficiently small; however we need the 
{\em a priori} $L^{\infty}$ bound for $u_k,$ which was built
in Lemma~\ref{lm:discrete-apriori}.
\end{remark}
Using the results in Section~\ref{subsec:global-quasi}, we can now easily 
establish quasi-orthogonality.
\begin{lemma}[Quasi-Orthogonality]
  \label{L:ex1-quasi-orthog}
Let $u$ be the solution to equation~\eqref{eqn:semi_linear_weak}, and $u_{k+1}$ and $u_k$ be the solutions to~\eqref{eqn:semi_linear_weak_fem} on $\cT_{k+1}$ and $\cT_k$ respectively.  Let  $X_{k}\subset X_{k+1},$ and the triangulations $\cT_{k}$ satisfy the condition~\eqref{eqn:mesh-condition}. Assume that there exist a $\sigma_{k+1}>0$ with $\sigma_{k+1}\to 0$ as $k\to \infty$ such that 
\begin{equation}
\label{eqn:aubin-nitsche}
    \|u-u_{k+1}\|_{0,\Omega} \leqs \sigma_{k+1} \|\nabla u-\nabla u_{k+1} \|_{0,\Omega},
\end{equation}
Then there exists a constant $C^*>0$, 
such that for sufficiently small $h,$ we have 
$$
| u-u_{k+1} |_{1,\Omega}^2 \leqs \Lambda_{k+1} |u-u_k |_{1,\Omega}^2 
   - | u_{k+1} - u_k |_{1,\Omega}^2,
$$
where $\Lambda_{k+1} = (1 - C^* \sigma_{k+1} K)^{-1}>0$ with $K = 3\sup_{\chi\in [u_{-}, u_{+}]}\|\chi^{2}\|_{\infty,\Omega}.$
\end{lemma}

\begin{proof}
    From the definition, $a(\cdot, \cdot)$ is a symmetric coercive bilinear form. The energy norm $\tbar v \tbar:=a(v,v) = |v|_{1,\Omega}^{2}$ is equivalent to the $H^{1}$-norm in $H_{0}^{1}(\Omega).$
    Now we verify the Lipschitz continuity~\eqref{E:lipschitz} for $N(u) = u^{3}.$
It follows by the {\em a priori} error estimates Lemma~\ref{lm:apriori-u} and Lemma~\ref{lm:discrete-apriori} of $u$ and $u_{k+1}:$
\begin{eqnarray*}
  |( N(u) - N(u_{k+1}), v_{k+1})|&\leqs& \sup_{\chi\in [u_-, u_+]}\|3\chi^{2}\|_{\infty} \|u-u_{k+1}\|_{0,\Omega} \|v_{k+1}\|_{0,\Omega}\\
  &\leqs& K \|u-u_{k+1}\|_{0,\Omega} \|v_{k+1}\|_{1,\Omega},
\end{eqnarray*}
where $K = \sup_{\chi\in [u_-, u_+]}\|3\chi^{2}\|_{\infty}<\infty.$
The conclusion follows by Theorem~\ref{T:quasi-semilinear}.
\end{proof}

Now that the Quasi-Orthogonality Assumption~\ref{A:quasi} is in place,
recall that the residual-based {\em a posteriori} error indicator for 
equation~\eqref{eqn:semi_linear_weak} is given 
by~\eqref{eqn:semi_linear_indicator}:
$$\eta_k(u_{k}, \tau)^2 := h^2_{\tau} \|u_k^3 -f\|^2_{0,\tau} + \sum_{\sigma\subset \partial\tau} h_{\sigma} \left\|\left[\nabla u_k\cdot n_{\sigma} \right]\right\|^2_{0, \sigma},$$
with $\eta_k(v, \mathcal{S}) := \left(\sum_{\tau\in \mathcal{S}} \eta_k^2(v, \tau)\right)^{\frac{1}{2}}$ for any subset $\mathcal{S} \subset \cT_k.$
The second ingredient of the contraction argument, namely the
Upper Bound Assumption~\ref{A:upper}, is provided by 
Theorem~\ref{thm:semi_linear_posteriori}.
To apply the contraction Theorem~\ref{T:contraction-afem}, we only need to 
verify the Local Lipschitz Assumption~\ref{A:local-lipschitz}, which 
implies the Indicator Reduction Lemma~\ref{L:indicator-reduction}. 
To this end, we introduce the PDE-related indicator:
\begin{eqnarray*}
    \eta^2({\bD},\tau) &:=& h_{\tau}^2\sup_{\chi\in [u_-, u_+]} \|3 \chi^{2}\|_{\infty,\tau}^2.
\end{eqnarray*}
For any subset $\cS\subset \cT,$ let $\eta({\bD},\cS):= \max_{\tau\in \cS}\{\eta({\bD},\tau)\}.$
By the definition, it is obvious that $\eta({\bD},\cT)$ is monotone decreasing, i.e.,
\begin{equation}
    \label{eqn:data-monotone}
        \eta({\bD},\cT_*) \le \eta({\bD},\cT)
\end{equation}
for any refinement $\cT_*$ of $\cT.$

\begin{lemma}[Local Lipschitz]
   \label{L:local-lipschitz-ex1}
Let $\cT$ be a conforming partition.
For all $\tau \in \cT$ and for any pair of discrete functions
$v,w \in [u_-,u_+] \cap X(\cT),$ it holds that
\begin{eqnarray}
|\eta(v,\tau) - \eta( w,\tau)|
&\leqs&\bar{\Lambda}_1 \eta(\mathbf{D},\tau) |v-w |_{1,\omega_{\tau}},
\end{eqnarray}
where $\bar{\Lambda}_1 > 0$ depends only on the shape-regularity of $\cT_0,$
 and the maximal values that $u^{3}$ can
obtain on the $L^{\infty}$-interval $[u_-,u_+].$
\end{lemma}
\begin{proof}
By the definition of $\eta,$  we have
\begin{eqnarray*}
    \eta(v,\tau) & \lesssim & \eta(w,\tau) + h_{\tau} \|v^{3} - w^{3}\|_{0,\tau} + \frac{1}{2}\sum_{\sigma\subset \partial \tau}h_{\sigma}^{\frac{1}{2}}\|n_{\sigma}\cdot [\nabla (v-w)]\|_{0,\sigma}
\end{eqnarray*}
Notice that 
$$
    \|v^{3} - w^{3}\|_{0,\tau}\le \left(\sup_{\chi\in [u_-, u_+]} \|3\chi^{2}\|_{\infty,\tau}\right) \|v-w\|_{0,\tau}.
$$
On the other hand, we also have
$$\|n_{\sigma}\cdot [\nabla (v-w)]\|_{0,\sigma} \le h_{\tau}^{-\frac{1}{2}}\|\nabla v - \nabla w\|_{0,\omega_{\tau}}.$$
 Therefore ,we get the desired estimate for $\eta.$
\end{proof}

Combining all of the above we obtain a contraction 
result for the AFEM algorithm:
\begin{theorem}[Contraction and Convergence]
    \label{T:ex1-contraction}
Let $\{\cT_k, V_k, u_k \}_{k\geqs 0}$ be the
sequence of finite element meshes, spaces, and solutions,
respectively, produced by AFEM($\theta$,$l$) with
marking parameter $\theta \in (0,1]$ and bisection level $\ell \geqs 1$.
Let $h_0$ be sufficiently fine so that Lemma~\ref{L:ex1-quasi-orthog}
holds for $\{\cT_k, V_k, u_k\}_{k\geqs 0}.$
Then, there exist constants $\gamma > 0$ and $\alpha \in (0,1)$,
depending only on $\theta$, $\ell$, and the shape-regularity of
the initial triangulation $\cT_0$, such that
$$
    |u - u_{k+1} |_{1,\Omega}^2 + \gamma \eta_{k+1}^2
   \leqs
\alpha^2 \left( |u - u_k|_{1,\Omega}^2 + \gamma \eta_k^2 \right).
$$
Consequently, we have the following convergence of AFEM algorithm:
$$
    |u - u_k |_{1,\Omega}^2 + \gamma \eta_k^2
   \leqs
   C_0 \alpha^{2k},
$$
for some constant $C_0 = C_0(u_0, h_0, \theta, l, \cT_0).$
\end{theorem}
\begin{proof}
The results follow from 
Theorems~\ref{T:contraction-afem}.
\end{proof}

\subsection{The Poisson-Boltzmann Equation}
\label{subsec:pbe}

The second example we consider is the nonlinear 
Poisson-Boltzmann equation (PBE), which is widely 
used for modeling the electrostatic interactions among charges particles;
it is important in many areas of science and engineering, including
biochemistry and biophysics.
The nonlinear PBE is
\begin{equation} \label{eqn:pbe}
    \begin{array}{rlll} 
    -\nabla \cdot (\epsilon \nabla \tilde{u}) + \kappa^2 \sinh \tilde{u} 
         &=& f, & \text{~in~} \Omega, \\
    u &=& 0, & \textrm{~on~} \partial \Omega,
    \end{array}
\end{equation} 
where $f = \sum_{i=1}^{N} q_i \delta(x_i)$, with 
$x_{i}\in \Omega_{m}\subset \Omega.$
Here, $\epsilon = \epsilon(x)>0$ is a strictly positive 
spatially-dependent dielectric coefficient, 
with the modified Debye-Huckel constant taking the value $\kappa = 0$ in the 
solute (molecule) region $\Omega_{m}$ and then strictly positive in the solvent 
region $\Omega_{s}:=\Omega\setminus \Omega_{m}.$
We will denote the interface between the molecular and solvent
regions as $\Gamma=\partial\Omega_{m}$.
 
One of the main analysis and approximation theory difficulties with the PBE
arises from the singular function $f$, which does not belong to 
$H^{-1}(\Omega)$; this implies~\eqref{eqn:pbe} does not have a solution 
in $H^{1}(\Omega),$ or at least does not have a normal $H^1$ weak formulation
with test functions coming form $H^1$.
To address this and other features of the PBE, 
Chen, Xu and Holst~\cite{Chen.L;Holst.M;Xu.J2007} used a two-scale
decomposition (see also~\cite{GDLM93,ZPVKL96}) to split the solution into a 
self-energy corresponding to the electrostatic potential $u^{s}$, and 
a screening potential due to high dielectric and mobile ions in the 
solution region. 
The singular component $u^s$ of the electrostatic potential satisfies
\begin{eqnarray}
    \label{eqn:us_1}
    -\nabla \cdot (\epsilon_m \nabla u^s) = \sum_{i=1}^{N} q_i \delta(x_i),
\end{eqnarray}
which can be assembled from the Green's functions
$u^s : = \sum_{i=1}^{N} q_i/(\epsilon_m |x-x_i|)$.
Subtracting~\eqref{eqn:us_1} from~\eqref{eqn:pbe} gives the
equation for $u$:
\begin{eqnarray}
\label{eqn:ur_1}
-\nabla \cdot (\epsilon \nabla u) + \kappa^2 \sinh (u + u^s) = \nabla \cdot ((\epsilon - \epsilon_m) \nabla u^s). 
\end{eqnarray}
In~\cite{Chen.L;Holst.M;Xu.J2007}, a new solution theory,
approximation theory, and convergent AFEM algorithm for the nonlinear PBE 
was established, based on this decomposition.
However, it was discovered later numerically that this decomposition 
requires that the regular component must be solved at an extremely high 
accuracy.
This defect is built into the decomposition itself due to the large
scale separation between the two components of the splitting.
A related decomposition scheme without this stability problem was 
studied numerically for finite difference schemes
in~\cite{Chern.I;Liu.J;Wan.W2003}, and then analyzed
carefully in~\cite{Holst.M;McCammon.J;Yu.Z;Zhou.Y2009}.
This 3-term decomposition uses the same first component in the molecular 
$u^{s}$ as defined in~\eqref{eqn:us_1};
the second component $u^h$ is the harmonic extension of the trace of
$u^s$ on $\Gamma$ (the interface between $\Omega_m$ and $\Omega_s$)
into the molecular region, with $u^h$ satisfying
\begin{equation} \label{eqn:uh_1}
    \begin{array}{rlll} 
        -\Delta u^h & =&  0 &\mbox{in}~\Omega_m,  \\
              u^h & =&-u^s  &\mbox{on}~\Gamma.
    \end{array}
\end{equation} 
One sets $u^s + u^h =0$ in $\Omega_s,$ 
with the harmonic extension $u^h$ continuous across the interface
by construction.
The remaining regular component satisfies the Regularized PBE (RPBE):
\begin{equation}
    \label{eqn:reg-pbe}
    \begin{array}{rlll}
    -\nabla\cdot (\epsilon \nabla u) + \kappa^2 \sinh (u) &=& 0 &\mbox{in } \Omega\\
    
    [u]_{\Gamma} =0 \mbox{ and }{\left[\epsilon \frac{\partial u}{\partial n_{\Gamma}}\right]}_{\Gamma} &=& g_{\Gamma} &\mbox{on } \Gamma\\

    u|_{\partial \Omega} & = & g &\mbox{on } \partial\Omega.\\
    \end{array},
\end{equation}
where $$g_{\Gamma} = g_{\Gamma} :=\varepsilon_m \frac{\partial 
(u^s + u^{h})}{\partial n_{\Gamma}}|_{\Gamma}.$$
Apart from the techniques required to handle the singular features 
described above, the remaining complexities (the discontinuous dielectric 
and Debye-Huckel constants and super-critical nonlinearity) can be
handled directly by the framework described in this paper;
in particular, both forms of the regularized 
problem~\eqref{eqn:ur_1} and~\eqref{eqn:reg-pbe},
analyzed in~\cite{Chen.L;Holst.M;Xu.J2007} 
and in~\cite{Holst.M;McCammon.J;Yu.Z;Zhou.Y2009} respectively,
fit into the class of semilinear problems described in
Section~\ref{subsec:semilinear-contraction}. 
The results remaining to be established for use of the AFEM
contraction framework essentially all follow from Lipschitz control 
of the nonlinearity~\eqref{E:lipschitz};
this control is gained through establishing continuous and
discrete {\em a priori} $L^{\infty}$ estimates for the weak solution $u$
to~\eqref{eqn:ur_1} and~\eqref{eqn:reg-pbe},
and for the Galerkin approximation $u_k$ of these solutions.
Such {\em a priori} $L^{\infty}$ estimates are established in
analyzed in~\cite{Chen.L;Holst.M;Xu.J2007} 
and~\cite{Holst.M;McCammon.J;Yu.Z;Zhou.Y2009}, following
cutoff-function arguments similar to those used in
Lemma~\ref{lm:apriori-u} above.
Both the quasi-orthogonality result in
Theorem~\ref{T:quasi-semilinear} and the nonlinear local Lipschitz
Assumption~\ref{A:local-lipschitz}
follow from the these {\em priori} $L^{\infty}$ bounds;
for details see~\cite{Holst.M;McCammon.J;Yu.Z;Zhou.Y2009}.
Contraction (hence convergence) of AFEM then follows by the
contraction Theorem~\ref{T:contraction-afem}; 
see~\cite{Holst.M;McCammon.J;Yu.Z;Zhou.Y2009} for the
complete argument.
For a short derivation of the equation, and a more detailed discussion of 
the solution theory, the approximation theory, and adaptive methods,
see~\cite{Chen.L;Holst.M;Xu.J2007,Holst.M;McCammon.J;Yu.Z;Zhou.Y2009}.

\subsection{The Hamiltonian Constraint Equation}
\label{subsec:hc}

The third example we consider is the scalar Hamiltonian constraint equation, 
which together with the vector momentum constraint, appears as the 
coupled Einstein constraint equations which arise in general relativity.
The derivation of the constraint equations is based on a 
\emph{conformal decomposition} technique, 
introduced by Lichnerowicz and York~\cite{aL44,jY71,jY72}.
In certain physical situations (constant mean extrinsic curvature of
the 3-manifold spatial domain), the constraints decouple so that the
(linear) momentum constraint can be solved first for a vector potential $w$,
leaving the Hamiltonian constraint to be solved separately
for a scalar conformal factor $u$.
Let $\Omega \subset \R^d$ be bounded and polyhedral, with $d \geqs 2.$
We consider then AFEM algorithms the scalar Hamiltonian constraint
equation:
Find $u$ such that
\begin{equation}
\label{eqn:H-constraint}
    \left\{\begin{array}{rll}
- \nabla \cdot (A \nabla u) + N(u) &=& 0 \quad \quad \mathrm{in}~ \Omega,\\
  n \cdot (A \nabla u) + G(u) &=& 0 \quad \quad \mathrm{on}~ \partial_N \Omega,\\
                            u &=& 0 \quad \quad \mathrm{on}~ \partial_D \Omega.                 
          \end{array}\right.
\end{equation}
The boundary conditions of primary interest in both mathematical and
numerical relativity include the 
cases $\partial_{D}\Omega=\emptyset$ or $\partial_{N}\Omega=\emptyset,$ 
which covers various combinations of boundary conditions considered 
in the literature~\cite{YoPi82,dM05b,sD04} for the constraint equations.
The tensor $A$ is a Riemannian metric, so it appears here as a uniformly 
positive definite symmetric matrix function on $\Omega$:
\begin{equation}
c_1 |\xi|^2 \leqs A_{ij}(x) \xi_i \xi_j \leqs c_2 |\xi|^2, 
\quad \mathrm{a.e.~in}~ \Omega,
\end{equation}
with component functions $A_{ij} \in L^{\infty}$.
The principle part of equation~\eqref{eqn:H-constraint} is the 
Laplace-Beltrami operator with certain Riemannian metric $h_{ab}.$
The (nonlinear) boundary function $G$ is assumed to be 
$C^2(\partial_N \Omega).$
The nonlinear function $N(\cdot)$ in the Hamiltonian constraint equation reads:
\begin{equation*}
N(\phi) =  
a_{\tiR}\phi
+ a_{\tau}\phi^5 
- a_{\rho}\phi^{-3} 
- a_{w}\phi^{-7},
\end{equation*}
where $a_\tau,a_\rho,a_w\in H^{-1}_{D}(\Omega)$  are nonnegative functions, 
and $a_{\tiR} := \frac{1}{8}R\in H^{-1}_{D}(\Omega)$, with the scalar 
curvature $R$ of the metric $h_{ab}.$
Here $$u_{-}, u_{+}\in H^{1}_{D}(\Omega)\cap L^{\infty}(\Omega) 
\mbox{  with  }0<u_{-}\leqs u_{+}<\infty.$$
The construction of the subsolution $u_{-}$ and the supersolution $u_{+}$ 
for the constraint equations was discussed in 
detail in~\cite{Holst.M;Nagy.G;Tsogtgerel.G2009}.

Note that $N(u)$ is well-defined only on essentially
bounded subsets of $L^2:$
\begin{equation}
N : [u_-,u_+] \subset L^2(\Omega) \rightarrow H_D^{-1}(\Omega).
\end{equation}
Such a restriction will give rise to a Lipschitz property of $N$ on
this set:
\begin{equation}
   \label{E:NK}
\|N(u)-N(v)\|_{\cL(H^1_D(\Omega),H^{-1}_{D}(\Omega))} \leqs K \|u-v\|_{L^2(\Omega)}, 
    \quad \forall u,v \in [u_-,u_+] \cap L^2(\Omega),
\end{equation}
which is a key tool 
for controlling the nonlinearity in the solution theory, when
combined with {\em a priori} $L^{\infty}$ bounds to establish
the interval $[u_-,u_+],$ as we saw in 
Section~\ref{subsec:semilinear-contraction}. 

A weak formulation of equation~\eqref{eqn:H-constraint} is then: 
Find $u \in [u_-,u_+] \cap H^{1}_{D}(\Omega)$ such that
\begin{equation}
\label{eqn:H-weak}
a(u,v) + \langle f(u),v \rangle = 0, \quad \forall v \in H^{1}_D(\Omega),
\end{equation}
where 
\begin{eqnarray*}
a(u,v) & = & \int_{\Omega} A \nabla u\cdot \nabla v dx,
\\
\langle f(u),v \rangle & = &
     \int_{\Omega} N(u)v dx + \int_{\partial_N \Omega}G(u)v ds.
\end{eqnarray*}
Thanks to the control of the nonlinearity provided by
the {\em a priori} $L^{\infty}$ bounds we established on any 
solution $u$ to the Hamiltonian constraint, it was showed in~\cite{Holst.M;Nagy.G;Tsogtgerel.G2009} that equation~\eqref{eqn:H-weak} is a well-posed problem.
In particular, there exists a solution 
$u \in [u_-,u_+] \cap H_{D}^{1}(\Omega).$

The remaining ideas in design the AFEM algorithm and its convergence analysis
are the same as before. 
Namely, we develop the {\em a priori} $L^{\infty}$ bounds of $u$ and the 
finite element approximation $u_{k}.$
Based on these {\em a priori} bounds, we then establish
quasi-orthogonality and the local Lipschitz property.
Finally, contraction and convergence of the AFEM algorithm follows by the 
contraction Theorem~\ref{T:contraction-afem}.
For a detailed discussion of the equation, the solution theory, approximation
theory, and convergence analysis of AFEM, 
see~\cite{Holst.M;Tsogtgerel.G2009,Holst.M;Nagy.G;Tsogtgerel.G2009}.

%% file: conc.tex
\section{Conclusion}
\label{sec:conc}

In this article we developed convergence theory for a general class
of adaptive approximation algorithms for abstract nonlinear 
operator equations on Banach spaces, and then used the theory to
obtain convergence results for practical adaptive finite element
methods (AFEM) applied to a several classes of nonlinear elliptic equations.
In the first part of the paper,
we developed a weak-* convergence framework for nonlinear 
operators, whose Gateaux derivatives are locally Lipschitz and satisfy a local inf-sup condition.
The framework can be viewed as extending the recent convergence results for
linear problems of Morin, Siebert and Veeser 
to a general nonlinear setting.
We formulated an abstract adaptive approximation algorithm 
for nonlinear operator equations in Banach spaces with local structure.
The weak-* convergence framework was then applied to this class of 
abstract locally adaptive algorithms, giving a general convergence result.
The convergence result was then applied to a standard AFEM algorithm in
the case of several semilinear and quasi-linear scalar elliptic equations 
and elliptic systems, including a semilinear problem with polynomial nonlinearity, the 
steady Navier-Stokes equations, and a more general quasilinear problem.
This yielded several new AFEM convergence results for these nonlinear 
problems.

In the second part of the paper, we developed a second abstract convergence 
framework based on strong contraction, extending the recent contraction 
results for linear problems of Cascon, 
Kreuzer, Nochetto, and Siebert and of Mekchay and Nochetto 
to abstract nonlinear problems.
We then established conditions under which it is possible to apply the 
contraction framework to the abstract adaptive algorithm defined earlier,
giving a contraction result for AFEM-type algorithms applied to nonlinear
problems.
The contraction result was then applied to a standard AFEM algorithm in
the case of several semilinear scalar elliptic equations,
including a semilinear problem with polynomial nonlinearity, the 
Poisson-Boltzmann equation, and the Hamiltonian constraint in
general relativity, yielding AFEM contraction results in each case.

%% file: ack.tex
\section{Acknowledgments}
   \label{sec:ack}

MH was supported in part by NSF Awards~0715146 and 0915220,
by DOE Award DE-FG02-04ER25620,
and by DOD/DTRA Award HDTRA-09-1-0036.
GT and YZ were were supported in part by NSF Award~0715146.

%% file: m.bbl
\begin{thebibliography}{10}

\bibitem{Adams.R1978}
R.~A. Adams.
\newblock {\em Sobolev Spaces}.
\newblock Academic Press, 1978.

\bibitem{Ainsworth.M;Oden.J2000}
M.~Ainsworth and J.~Oden.
\newblock {\em A Posteriori Error Estimation in Finite Element Analysis}.
\newblock John Wiley \& Sons, Inc., 2000.

\bibitem{Arnold.D;Mukherjee.A;Pouly.L2000}
D.~N. Arnold, A.~Mukherjee, and L.~Pouly.
\newblock Locally adapted tetrahedral meshes using bisection.
\newblock {\em SIAM Journal of Scientific Computing}, 22(2):431--448, 2000.

\bibitem{Babuska.I;Rheinboldt.W1978a}
I.~Babu{\v s}ka and W.~C. Rheinboldt.
\newblock Error estimates for adaptive finite element computations.
\newblock {\em SIAM Journal on Numerical Analysis}, 15:736--754, 1978.

\bibitem{Babuska.I;Rheinboldt.W1978}
I.~Babu{\v s}ka and W.~C. Rheinboldt.
\newblock A posteriori error error estimates for the finite element method.
\newblock {\em International Journal for Numerical Methods in Engineering},
  12:1597--1615, 1978.

\bibitem{Babuska.I;Vogelius.M1984}
I.~Babu{\v s}ka and M.~Vogelius.
\newblock Feeback and adaptive finite element solution of one-dimensional
  boundary value problems.
\newblock {\em Numerische Mathematik}, 44:75--102, 1984.

\bibitem{Bank.R;Rose.D1981}
R.~E. Bank and D.~J. Rose.
\newblock Global approximate {N}ewton methods.
\newblock {\em Numerische Mathematik}, 37:279--295, 1981.

\bibitem{Bank.R;Smith.R1993}
R.~E. Bank and R.~K. Smith.
\newblock A posteriori estimates based on hierarchical basis.
\newblock {\em SIAM Journal on Numerical Analysis}, 30:921--935, 1993.

\bibitem{Bansch.E;Siebert.K1995}
E.~B{\"a}nsch and K.~Siebert.
\newblock {\em {A Posteriori Error Estimation for Nonlinear Problems by Duality
  Techniques}}.
\newblock Albert-Ludwigs-Univ., Math. Fak., 1995.

\bibitem{Berrone.S2001}
S.~Berrone.
\newblock Adaptive discretization of stationary and incompressible
  navier--stokes equations by stabilized finite element methods.
\newblock {\em Computer Methods in Applied Mechanics and Engineering},
  190(34):4435--4455, 2001.

\bibitem{Binev.P;Dahmen.W;DeVore.R2004}
P.~Binev, W.~Dahmen, and R.~DeVore.
\newblock Adaptive finite element methods with convergence rates.
\newblock {\em Numerische Mathematik}, 97(2):219--268, 2004.

\bibitem{Brezzi.F;Fortin.M1991}
F.~Brezzi and M.~Fortin.
\newblock {\em Mixed and hybrid finite element methods}.
\newblock Springer-Verlag, 1991.

\bibitem{Caloz.G;Rappaz.J1994}
G.~Caloz and J.~Rappaz.
\newblock {Numerical analysis for nonlinear and bifurcation problems}.
\newblock {\em Handbook of Numerical Analysis}, 1994.

\bibitem{Carstensen.C2009}
C.~Carstensen.
\newblock {Convergence of adaptive finite element methods in computational
  mechanics}.
\newblock {\em Applied Numerical Mathematics}, 59(9):2119--2130, 2009.

\bibitem{Carstensen.C;Orlando.A;Valdman.J2006}
C.~Carstensen, A.~Orlando, and J.~Valdman.
\newblock {A convergent adaptive finite element method for the primal problem
  of elastoplasticity}.
\newblock {\em International Journal for Numerical Methods in Engineering},
  67(13), 2006.

\bibitem{Cascon.J;Kreuzer.C;Nochetto.R;Siebert.K2007}
J.~M. Casc\'on, C.~Kreuzer, R.~H. Nochetto, and K.~G. Siebert.
\newblock Quasi-optimal convergence rate for an adaptive finite element method.
\newblock {\em Preprint 9, University of Augsburg}, 2007.

\bibitem{Cascon.J;Kreuzer.C;Nochetto.R;Siebert.K2008}
J.~M. Cascon, C.~Kreuzer, R.~H. Nochetto, and K.~G. Siebert.
\newblock Quasi-optimal convergence rate for an adaptive finite element method.
\newblock {\em SIAM Journal on Numerical Analysis}, 46(5):2524--2550, 2008.

\bibitem{Charina.M;Conti.C;Fornasier.M2008}
M.~Charina, C.~Conti, and M.~Fornasier.
\newblock {Adaptive frame methods for nonlinear variational problems}.
\newblock {\em Numerische Mathematik}, 109(1):45--75, 2008.

\bibitem{Chen.L2006a}
L.~Chen.
\newblock Short implementation of bisection in {MATLAB}.
\newblock {\em report}, 2006.

\bibitem{Chen.L;Holst.M;Xu.J2007}
L.~Chen, M.~Holst, and J.~Xu.
\newblock The finite element approximation of the nonlinear poisson-boltzmann
  equation.
\newblock {\em SIAM Journal on Numerical Analysis}, 45(6):2298--2320, 2007.

\bibitem{Chern.I;Liu.J;Wan.W2003}
I.-L. Chern, J.-G. Liu, and W.-C. Wan.
\newblock Accurate evaluation of electrostatics for macromolecules in solution.
\newblock {\em Methods and Applications of Analysis}, 10:309--328, 2003.

\bibitem{Ciarlet.P1978}
P.~G. Ciarlet.
\newblock {\em The Finite Element Method for Elliptic Problems}, volume~4 of
  {\em Studies in Mathematics and its Applications}.
\newblock North-Holland Publishing Co., Amsterdam-New York-Oxford, 1978.

\bibitem{sD04}
S.~Dain.
\newblock Trapped surfaces as boundaries for the constraint equations.
\newblock {\em Classical Quantum Gravity}, 21(2):555--573, 2004.

\bibitem{Deuflhard.P2004}
P.~Deuflhard.
\newblock {\em Newton Methods for Nonlinear Problems. Affine Invariance and
  Adaptive Algorithms}.
\newblock Springer Series in Computational Mathematics. Springer, 2004.

\bibitem{Diening.L;Kreuzer.C2008}
L.~Diening and C.~Kreuzer.
\newblock {Linear Convergence of an Adaptive Finite Element Method for the\$
  p\$-Laplacian Equation}.
\newblock {\em SIAM Journal on Numerical Analysis}, 46:614, 2008.

\bibitem{Dorfler.W1996}
W.~D\"orfler.
\newblock A convergent adaptive algorithm for {Poisson}'s equation.
\newblock {\em SIAM Journal on Numerical Analysis}, 33:1106--1124, 1996.

\bibitem{GDLM93}
M.~K. Gilson, M.~E. Davis, B.~A. Luty, and J.~A. McCammon.
\newblock Computationn of electrostatic forces on solvated molecules using the
  poisson-boltzmann equation.
\newblock {\em J. Phys. Chem.}, 97:3591--3600, 1993.

\bibitem{Girault.V;Raviart.P1986}
V.~Girault and P.~A. Raviart.
\newblock {\em Finite element methods for {Navier}--{Stokes} equations}.
\newblock Springer-Verlag, Berlin, 1986.
\newblock Theory and algorithms.

\bibitem{Hebey96}
E.~Hebey.
\newblock {\em Sobolev spaces on {R}iemannian manifolds}, volume 1635 of {\em
  Lecture notes in mathematics}.
\newblock Springer, Berlin, New York, 1996.

\bibitem{Holst.M2001}
M.~Holst.
\newblock Adaptive numerical treatment of elliptic systems on manifolds.
\newblock {\em Advances in Computational Mathematics}, 15:139--191, 2001.

\bibitem{Holst.M;McCammon.J;Yu.Z;Zhou.Y2009}
M.~Holst, J.~McCammon, Z.~Yu, Y.~Zhou, and Y.~Zhu.
\newblock {Adaptive Finite Element Modeling Techniques for the
  Poisson-Boltzmann Equation}.
\newblock {\em Preprint}, 2009.

\bibitem{HNT07b}
M.~Holst, G.~Nagy, and G.~Tsogtgerel.
\newblock Rough solutions of the {Einstein} constraints on closed manifolds
  without near-{CMC} conditions.
\newblock {\em Comm. Math. Phys.}, 288(2):547--613, 2009.
\newblock Available as \href{http://arxiv.org/abs/0712.0798}{arXiv:0712.0798
  [gr-qc]}.

\bibitem{Holst.M;Nagy.G;Tsogtgerel.G2009}
M.~Holst, G.~Nagy, and G.~Tsogtgerel.
\newblock {Rough solutions of the Einstein constraints on closed manifolds
  without near-CMC conditions}.
\newblock {\em Communications in Mathematical Physics}, 288(2):547--613, 2009.

\bibitem{HoTs07a}
M.~Holst and G.~Tsogtgerel.
\newblock Adaptive finite element approximation of nonlinear geometric {PDE}.
\newblock Preprint.

\bibitem{Holst.M;Tsogtgerel.G2009}
M.~Holst and G.~Tsogtgerel.
\newblock Convergent adaptive finite element approximation of the einstein
  constraints.
\newblock {\em Preprint}, 2009.

\bibitem{HoTs07b}
M.~Holst, G.~Tsogtgerel, and Y.~Zhu.
\newblock Convergent adaptive finite element approximation of the {Einstein}
  constraints.
\newblock Preprint.

\bibitem{Jerome.J1985}
J.~W. Jerome.
\newblock Consistency of semiconductor modeling: An existence/stability
  analysis for the stationary van roosbroeck system.
\newblock {\em SIAM Journal on Applied Mathematics}, 45(4):565--590, Aug. 1985.

\bibitem{Kesa89}
S.~Kesavan.
\newblock {\em Topics in Functional Analysis and Applications}.
\newblock John Wiley \& Sons, Inc., New York, NY, 1989.

\bibitem{Kim.K2007}
K.~Kim.
\newblock {A posteriori error estimators for locally conservative methods of
  nonlinear elliptic problems}.
\newblock {\em Applied Numerical Mathematics}, 57(9):1065--1080, 2007.

\bibitem{Krizek.M;Nemec.J;Vejchodsky.T2001}
M.~K{\v{r}}{\'\i}{\v{z}}ek, J.~N{\v{e}}mec, and T.~Vejchodsk{\`y}.
\newblock {A Posteriori Error Estimates for Axisymmetric and Nonlinear
  Problems}.
\newblock {\em Advances in Computational Mathematics}, 15(1):219--236, 2001.

\bibitem{aL44}
A.~Lichnerowicz.
\newblock L'integration des \'equations de la gravitation relativiste et le
  probl\`eme des n corps.
\newblock {\em J. Math. Pures Appl.}, 23:37--63, 1944.

\bibitem{Lions.J;Magenes.E1973}
J.~Lions and E.~Magenes.
\newblock {\em Non-Homogeneous Boundary Value Problems and Applications I}.
\newblock Springer-Verlag Berlin Heidelberg New York, 1973.

\bibitem{LiRh94}
J.~Liu and W.~Rheinboldt.
\newblock A posteriori finite element error estimators for indefinite elliptic
  boundary value problems.
\newblock {\em Numer. Funct. Anal. and Optimiz.}, 15(3):335--356, 1994.

\bibitem{dM05b}
D.~Maxwell.
\newblock Solutions of the {{\ E}}instein constraint equations with apparent
  horizon boundaries.
\newblock {\em Comm. Math. Phys.}, 253(3):561--583, 2005.

\bibitem{Mekchay.K;Nochetto.R2005}
K.~Mekchay and R.~Nochetto.
\newblock Convergence of adaptive finite element methods for general second
  order linear elliptic {PDE}.
\newblock {\em SIAM Journal on Numerical Analysis}, 43(5):1803--1827, 2005.

\bibitem{Morin.P;Nochetto.R;Siebert.K2002}
P.~Morin, R.~H. Nochetto, and K.~G. Siebert.
\newblock Convergence of adaptive finite element methods.
\newblock {\em SIAM Review}, 44(4):631--658, 2002.

\bibitem{Morin.P;Siebert.K;Veeser.A2008}
P.~Morin, K.~Siebert, and A.~Veeser.
\newblock {A Basic Convergence Result for Conforming Adaptive Finite Elements}.
\newblock {\em Mathematical Models and Methods in Applied Sciences},
  18(5):707--737, 2008.

\bibitem{Morin.P;Siebert.K;Veeser.A2007a}
P.~Morin, K.~G. Siebert, and A.~Veeser.
\newblock A basic convergence result for conforming adaptive finite elements.
\newblock {\em Math. Models Methods Appl. Sci}, 18:707--737, 2007.

\bibitem{Morin.P;Siebert.K;Veeser.A2007}
P.~Morin, K.~G. Siebert, and A.~Veeser.
\newblock Convergence of finite elements adapted for weak norms.
\newblock {\em Preprint, University of Augsburg}, 2007.

\bibitem{Nochetto.R;Siebert.K;Veeser.A2009}
R.~Nochetto, K.~Siebert, and A.~Veeser.
\newblock {Theory of adaptive finite element methods: An introduction}.
\newblock In R.~DeVore and A.~Kunoth, editors, {\em Multiscale, Nonlinear and
  Adaptive Approximation}, pages 409--542. Springer, 2009.
\newblock Dedicated to Wolfgang Dahmen on the Occasion of His 60th Birthday.

\bibitem{OrRh70}
J.~M. Ortega and W.~C. Rheinboldt.
\newblock {\em Iterative Solution of Nonlinear Equations in Several Variables}.
\newblock Academic Press, New York, NY, 1970.

\bibitem{Ortner.C;Praetorius.D2008}
C.~Ortner and D.~Praetorius.
\newblock On the convergence of non-conforming finite element methods.
\newblock {\em Preprint}, (14/2008), 2008.

\bibitem{Pala65}
R.~Palais.
\newblock {\em Seminar on the {A}tiyah-{S}inger index theorem}.
\newblock Princeton University Press, Princeton, 1965.

\bibitem{Plaza.A;Carey.G2000}
A.~Plaza and G.~F. Carey.
\newblock Local refinement of simplicial grids based on the skeleton.
\newblock {\em Applied Numerical Mathematics}, 32(2):195--218, 2000.

\bibitem{Pousin.J;Rappaz.J1994}
J.~Pousin and J.~Rappaz.
\newblock {Consistency, stability, a priori and a posteriori errors for
  {Petrov}-{Galerkin} methods applied to nonlinear problems}.
\newblock {\em Numerische Mathematik}, 69(2):213--231, 1994.

\bibitem{Rannacher.R;Scott.R1982}
R.~Rannacher and R.~Scott.
\newblock Some optimal error estimates for piecewise linear finite element
  approximations.
\newblock {\em Mathematics of Computation}, 38(158):437--445, 1982.

\bibitem{Rappaz.J2006}
J.~Rappaz.
\newblock Numerical approximation of pdes and cl{\'e}ment's interpolation.
\newblock {\em Partial Differential Equations and Functional Analysis}, pages
  237--250, 2006.

\bibitem{Repin.S2000a}
S.~Repin.
\newblock {A posteriori error estimation for nonlinear variational problems by
  duality theory}.
\newblock {\em Journal of Mathematical Sciences}, 99(1):927--935, 2000.

\bibitem{Repin.S2008}
S.~Repin.
\newblock {\em A posteriori estimates for partial differential equations},
  volume~4 of {\em Radon Series on Computational and Applied Mathematics}.
\newblock Walter de Gruyter GmbH \& Co. KG, Berlin, 2008.

\bibitem{Schatz.A1974}
A.~Schatz.
\newblock An observation concerning {R}itz-{G}alerkin methods with indefinite
  bilinear forms.
\newblock {\em Mathematics of Computation}, 28(205):952--962, 1974.

\bibitem{Siebert.K2009}
K.~Siebert.
\newblock {A convergence proof for adaptive finite elements without lower
  bound}.
\newblock {\em Preprint Universit{\"a}t Duisburg-Essen}, 2009.

\bibitem{Stevenson.R2007}
R.~Stevenson.
\newblock Optimality of a standard adaptive finite element method.
\newblock {\em Found. Comput. Math.}, 7(2):245--269, 2007.

\bibitem{Stevenson.R2008}
R.~Stevenson.
\newblock The completion of locally refined simplicial partitions created by
  bisection.
\newblock {\em Mathemathics of Computation}, 77:227--241, 2008.

\bibitem{Stru00}
M.~Struwe.
\newblock {\em Variational Methods}.
\newblock Springer-Verlag, Berlin, Germany, third edition, 2000.

\bibitem{Vain73}
M.~Vainberg.
\newblock {\em Variational method and method of monotone operators in the
  theory of nonlinear equations}.
\newblock John Wiley \& Sons Ltd, New York, NY, 1973.

\bibitem{Veeser.A2002}
A.~Veeser.
\newblock Convergent adaptive finite elements for the nonlinear {Laplacian}.
\newblock {\em Numerische Mathematik}, 92:743--770, 2002.

\bibitem{Verfurth.R1984}
R.~Verf{\"u}rth.
\newblock {Error estimates for a mixed finite element approximation of the
  Stokes problem}.
\newblock {\em RAIRO Anal. Numer}, 18(175-182), 1984.

\bibitem{Verfurth.R1994}
R.~Verf{\"u}rth.
\newblock A posteriori error estimates for nonlinear problems. finite element
  discretizations of elliptic equations.
\newblock {\em Mathematics of Computation}, 62(206):445--475, Apr. 1994.

\bibitem{Verfurth.R1996}
R.~Verf{\"u}rth.
\newblock {\em A review of a posteriori error estimation and adaptive mesh
  refinement tecniques}.
\newblock B. G. Teubner, 1996.

\bibitem{jY71}
J.~York.
\newblock Gravitational degrees of freedom and the initial-value problem.
\newblock {\em Phys. Rev. Lett.}, 26(26):1656--1658, 1971.

\bibitem{jY72}
J.~York.
\newblock Role of conformal three-geometry in the dynamics of gravitation.
\newblock {\em Phys. Rev. Lett.}, 28(16):1082--1085, 1972.

\bibitem{YoPi82}
J.~W. {York, Jr.} and T.~Piran.
\newblock The initial value problem and beyond.
\newblock In R.~A. Matzner and L.~C. Shepley, editors, {\em Spacetime and
  Geometry: The {A}lfred {S}child Lectures}, pages 147--176, Austin, Texas,
  1982. University of Texas Press.

\bibitem{ZPVKL96}
Z.~Zhou, P.~Payne, M.~Vasquez, N.~Kuhn, and M.~Levitt.
\newblock Finite-difference solution of the poisson-boltzmann equation:
  Complete elimination of self-energy.
\newblock {\em J.\ Comput.\ Chem.}, 17:1344--1351, 1996.

\end{thebibliography}
